\documentclass[10pt]{article}
\usepackage{amsmath,amsthm,amssymb,amsfonts}
\usepackage{enumerate}
\usepackage{mathrsfs}
\usepackage[cmtip,all]{xy}
\usepackage{tikz}
\usepackage{cleveref}
\usetikzlibrary{shapes.geometric}

\numberwithin{equation}{section}
\numberwithin{figure}{section}

\normalsize

\newtheoremstyle{theoremstyle}
  {10pt}      
  {5pt}       
  {\itshape}  
  {}          
  {\bfseries} 
  {:}         
  {.5em}      
  {}          

\newtheoremstyle{examplestyle}
  {10pt}      
  {5pt}       
  {}          
  {}          
  {\bfseries} 
  {:}         
  {.5em}      
  {}          

\theoremstyle{theoremstyle}
\newtheorem{theorem}{Theorem}[section]
\newtheorem*{theorem*}{Theorem}
\newtheorem{lemma}[theorem]{Lemma}
\newtheorem{proposition}[theorem]{Proposition}
\newtheorem*{proposition*}{Proposition}
\newtheorem{corollary}[theorem]{Corollary}
\newtheorem*{corollary*}{Corollary}

\theoremstyle{examplestyle}
\newtheorem{example}[theorem]{Example}
\newtheorem{definition}[theorem]{Definition}
\newtheorem{definition*}{Definition}
\newtheorem{remark}[theorem]{Remark}
\newtheorem{remark*}{Remark}

\let\smash=\wedge
\let\iso=\cong
\let\tensor=\otimes
\let\directsum=\oplus
\let\minus=\smallsetminus

\definecolor{sq3sq1color}{rgb}{0.5,0,0.5}
\definecolor{sq2color}{rgb}{0.1,0.7,0.1}
\colorlet{taucolor}{red}
\colorlet{partialsq2sq1color}{sq3sq1color!53!black}
\colorlet{incsq2sq1color}{sq3sq1color!67!green}
\colorlet{sq2rhosq1color}{taucolor!45!sq2color}
\colorlet{sq2prcolor}{sq2color!53!black}
\colorlet{incsq2color}{sq2color!67!yellow}
\colorlet{sq2partialcolor}{sq2color!42!blue}
\colorlet{tauprcolor}{taucolor!53!black}
\colorlet{taupartialcolor}{taucolor!42!yellow}

\newcommand{\Mil}{\mathsf{M}}
\newcommand{\MilWitt}{\mathsf{MW}}
\newcommand{\KMil}{\mathbf{K}^{\Mil}_{\ast}}

\newcommand{\KMW}{\mathbf{K}^{\MilWitt}_{\ast}}

\newcommand{\image}{\operatorname{im}}

\newcommand{\unit}{\mathbf{1}}
\newcommand{\SH}{\mathbf{SH}}
\newcommand{\eff}{\mathbf{eff}}
\newcommand{\Cpx}{\mathbf{Ch}}

\newcommand{\holim}{\operatornamewithlimits{holim}}
\newcommand{\Ext}{{\operatorname{Ext}}}
\newcommand{\Spec}{{\operatorname{Spec}}}
\newcommand{\colim}{\operatornamewithlimits{colim}}
\newcommand{\hocolim}{\operatornamewithlimits{hocolim}}

\newcommand{\PP}{\mathbf{P}}
\newcommand{\MGL}{\mathbf{MGL}}

\newcommand{\MU}{\mathbf{MU}}
\newcommand{\BP}{\mathbf{BP}}

\newcommand{\KQ}{\mathbf{KQ}}
\newcommand{\KT}{\mathbf{KW}}

\newcommand{\Gr}{\mathbf{Gr}}
\newcommand{\KGL}{\mathbf{KGL}}
\newcommand{\kgl}{\mathbf{kgl}}

\newcommand{\B}{\mathsf{B}}

\newcommand{\C}{\mathsf{C}}
\newcommand{\D}{\mathsf{D}}
\newcommand{\E}{\mathsf{E}}

\newcommand{\hyper}{\mathsf{h}}

\newcommand{\M}{\mathbf{M}}
\newcommand{\MW}{\mathbf{MW}}
\newcommand{\F}{\mathsf{F}}

\newcommand{\HHH}{\mathbf{H}}

\newcommand{\dd}{\mathsf{d}}

\newcommand{\s}{\mathsf{s}}
\newcommand{\f}{\mathsf{f}}
\newcommand{\T}{\mathbf{T}}

\newcommand{\uHom}{\underline{\mathrm{Hom}}}

\newcommand{\Mod}{\mathbf{mod}}

\newcommand{\cell}{\mathrm{cell}}

\newcommand{\conn}{\mathrm{conn}}

\newcommand{\NN}{{\mathbb N}}
\newcommand{\ZZ}{{\mathbb Z}}
\newcommand{\pr}{{\mathrm{pr}}}
\newcommand{\id}{{\mathrm{id}}}
\newcommand{\Id}{{\mathrm{Id}}}
\newcommand{\inc}{{\mathrm{inc}}}
\newcommand{\MS}{{\mathbf{MS}}}
\newcommand{\Tot}{{\mathrm{Tot}}}

\newcommand{\cd}{{\mathrm{cd}}}

\newcommand{\A}{\mathbf{A}}
\newcommand{\G}{\mathbf{G}_{\mathfrak{m}}}

\newcommand{\Char}{\mathsf{char}}

\newcommand{\CC}{\mathbb{C}}

\newcommand{\MZ}{\mathbf{M}\mathbb{Z}}
\newcommand{\MA}{\mathbf{M}A}
\newcommand{\MB}{\mathbf{M}B}
\newcommand{\MC}{\mathbf{M}C}
\newcommand{\MQ}{\mathbf{M}\mathbf{Q}}

\newcommand{\Z}{\mathbb{Z}}
\newcommand{\N}{\mathbb{N}}
\newcommand{\RR}{\mathbb{R}}

\newcommand{\SL}{\mathbf{SL}}
\newcommand{\Q}{\mathbf{Q}}
\newcommand{\QQ}{\mathbb{Q}}
\newcommand{\Sq}{\mathsf{Sq}}

\newcommand{\Top}{\mathsf{top}}

\newcommand{\Sm}{\mathbf{Sm}}

\newcommand{\Thom}{\mathsf{Th}}
\newcommand{\slicecomp}{\mathsf{sc}}

\newcommand{\Hom}{\operatorname{Hom}}

\newcommand{\Ab}{\mathsf{Ab}}
\newcommand{\Spt}{\mathsf{Spt}}

\newcommand{\Pre}{\mathsf{Pre}}

\newcommand{\Alg}{\mathsf{Alg}}

\newcommand{\Com}{\mathsf{Com}}

\title{{\bf The first stable homotopy groups \\ of motivic spheres}}
\author{Oliver R\"ondigs, Markus Spitzweck, Paul Arne {\O}stv{\ae}r}
\date{\today}
\begin{document}
\maketitle
\begin{abstract}
We compute the $1$-line of stable homotopy groups of motivic spheres over fields of characteristic not two in terms of hermitian and Milnor $K$-groups.
This is achieved by solving questions about convergence and differentials in the slice spectral sequence.
\end{abstract}

{\footnotesize{\tableofcontents}}
\newpage

\section{Introduction}
\label{section:introduction}

\subsection{Background and motivation}
We present a systematic approach to computing stable homotopy classes of maps between motivic spheres.
These invariants are the universal ones for smooth varieties subject to $\mathbf{A}^{1}$-invariance in the sense of motivic homotopy theory \cite{MorelICM2006}, \cite{VoevodskyICM1998}.
The idea of studying universal invariants is at the heart of the philosophy behind the theory of motives (motivic homotopy theory can be viewed as a non-abelian generalization).
Decades of research in algebraic topology have lead to beautiful mathematical structures related to the notoriously difficult problem of classifying maps between spheres up to homotopy, 
see e.g., \cite{HopkinsICM2002}.
Our approach to the corresponding algebro-geometric problem employs the slice spectral sequence.
The formalism of this approach is remarkably convenient.
It allows us to calculate differentials using the action of the motivic Steenrod algebra on motivic cohomology groups \cite{hko.positivecharacteristic}, \cite{Voevodsky.reduced}.
Our focus in this paper is not just restricted to specific computations, but also the context we lay out to formulate and carry them out.
There is still much to be done in this program, and many more questions could be answered by pressing these computations further.

\subsection{Main results and outline of the paper}
For simplicity let $F$ be a field of characteristic zero.
The motivic spheres $S^{p,q}$ over $F$ form a bigraded family of objects indexed by integers $p$ and $q$, 
and so do the stable homotopy groups $\pi_{\star}\unit$ of the motivic sphere spectrum over $F$.
By work of Morel \cite{MorelICM2006}, 
$\pi_{p,q}\unit=0$ if $p<q$ and the $0$-line $\bigoplus_{n\in\ZZ}\pi_{n+r,n}\unit$ is the Milnor-Witt $K$-theory $\mathbf{K}^{\MW}_{\ast}(F)$ of $F$, 
which incorporates a great deal of nuanced arithmetic information about the base field.
By the $r$-line we mean $\bigoplus_{n\in\ZZ}\pi_{n+r,n}\unit$ for $r\geq 0$.
A next logical step is to compute the terms $\pi_{n+1,n}\unit$ on the $1$-line.
One of the major inspirations for this paper is Morel's $\pi_{1}$-conjecture in weight zero. 
It states that there is a short exact sequence
\begin{equation}
\label{eq:first-stable-stem2} 
0 
\to 
\mathbf{K}^\M_{2}(F)/24 
\to 
\pi_{1,0}\unit 
\to 
F^{\times}/2\oplus\Z/2
\to
0.
\end{equation}
Here $\mathbf{K}^{\M}_{\ast}(F)$ is the Milnor $K$-theory of $F$ \cite{milnor.k-quadratic}.
In this paper we prove Morel's $\pi_{1}$-conjecture.
The surjection in \eqref{eq:first-stable-stem2} arises from the unit map for the hermitian $K$-theory spectrum $\KQ$ of quadratic forms \cite{hornbostel.hermitian}.
In general,
\eqref{eq:first-stable-stem2} does not split.
One can make precise the statement that $\mathbf{K}^\M_{2}(F)$ is generated by the second motivic Hopf map $\nu\in\pi_{3,2}\unit$ 
(obtained from the Hopf construction on 
$\SL_ {2}$ \cite[Remark 4.4, Definition 4.7]{dugger-isaksen.hopf}), 
while $F^{\times}/2\oplus\Z/2$ is generated by the topological Hopf map $\eta_{\Top}\in\pi_{1,0}\unit$.
We note the relations $24\nu=0$ and $12\nu=\eta^{2}\eta_{\Top}$,
where $\eta\in\pi_{1,1}\unit$ is the first motivic Hopf map.
The Hopf construction should witness that $\nu$ is in the image of a motivic $J$-homomorphism, 
so one may speculate whether the relation $24\nu=0$ is a shadow of some motivic version of the Adams conjecture \cite{zbMATH03222619}. 

More generally, 
for every $n\in\Z$,
we show an exact sequence of Nisnevich sheaves on smooth schemes of finite type  
\begin{equation}
\label{eq:first-stable-stem} 
0 
\to 
\mathbf{K}^{\M}_{2-n}/24
\to 
\mathbf{\pi}_{n+1,n}\unit
\to 
\mathbf{\pi}_{n+1,n}\f_0(\KQ).
\end{equation}
Here $\f_0(\KQ)$ is the effective cover of hermitian $K$-theory arising in the slice filtration of the motivic stable homotopy category of $F$
(this does not affect homotopy groups of nonnegative weight).
The rightmost map in \eqref{eq:first-stable-stem} is surjective for $n\geq -4$
(compare with \cite[Corollary 6]{asok-fasel.degree}). 
The exact sequence \eqref{eq:first-stable-stem} vastly generalizes computations in \cite{oo} for fields of cohomological dimension at most two, 
and in \cite{dugger-isaksen.real} and \cite{ho.galois} for the real numbers. 
Our computations of motivic stable homotopy groups are carried out on $F$-points.
This implies \eqref{eq:first-stable-stem} since the motivic homotopy sheaves are strictly $\mathbf{A}^1$-invariant \cite[Theorem 6.2.7]{morel.connectivity}, \cite[Theorem 2.11]{morel.field}.

It is interesting to compare with the computations of unstable motivic homotopy groups of punctured affine spaces in 
\cite{asok-fasel.splitting} and \cite{asok-wickelgren-williams}.
If $d>3$, 
the extension for the unstable homotopy sheaf $\pi_d(\mathbf{A}^d\minus\{0\})$ conjectured by Asok-Fasel~\cite[Conjecture~7, p.~1894]{mfo2013} coincides with~(\ref{eq:first-stable-stem}). 
As noted in \cite{asok-fasel.splitting}, 
the exact sequence~(\ref{eq:first-stable-stem}) and a conjectural Freudenthal $\PP^1$-suspension theorem imply Murthy's conjecture on splittings of vector bundles \cite[Conjecture 1]{asok-fasel.splitting}.

Our approach is to divide and conquer the slice spectral sequence for the motivic sphere spectrum.
More precisely, 
we
\begin{itemize}
\item[(1)] identify the slices of $\unit$ and their multiplicative structure,
\item[(2)] give widely applicative conditions under which the slice spectral sequence converges, and
\item[(3)] express the first slice differentials in terms of motivic Steenrod operations.
\end{itemize}

An explicit answer to (1) is out of reach because the slices of $\unit$ involve the $E_{2}$-page of the topological Adams-Novikov spectral sequence as conjectured by Voevodsky \cite{voevodsky.open} 
and verified by Levine \cite{levine.comparison} for fields.
Our solution applies to Dedekind domains of mixed characteristic. 
Much more work remains on the multiplicative structure;
we collect just enough results for the purpose of this paper by introducing techniques that will be useful for further computations.
This is carried out in Section~\ref{sec:slic-motiv-sphere}.

Salient among the results for (2) is Levine's convergence result \cite{levine.sliceconvergence} for fields of finite cohomological dimension.
Our solution for general fields,
see Section~\ref{sec:convergence}, 
clarifies the distinct role of $\eta$ in the context of the slice filtration:
Theorem~\ref{theorem:eta-slice-completion} identifies the slice completion of any motivic spectrum admitting a cell presentation of finite type with its $\eta$-completion.
For $\unit$ it follows that the slice spectral sequence takes the form of a conditionally convergent spectral sequence
\begin{equation}
\label{equation:splicespectralsequence}
E^{1}_{p,q,n}(\unit) 
=
\pi_{p,n}\s_{q}(\unit)
\Longrightarrow
\pi_{p,n} \unit^\wedge_\eta
\end{equation}
in the sense of Boardman \cite{Boardman}.
Here $\s_{q}(\unit)$ is the $q$th slice of $\unit$ and $\unit^\wedge_\eta$ is the $\eta$-complete sphere.
This firmly answers Voevodsky's question of convergence of the slice spectral sequence \cite{voevodsky.open}.  
We employ \eqref{equation:splicespectralsequence} to identify Morel's plus part of the rational sphere spectrum with rational motivic cohomology, 
and show a motivic homotopy finiteness result over finite fields.
The slice spectral sequence is an algebro-geometric analogue of the topological Atiyah-Hirzebruch spectral sequence \cite{zbMATH03176572}, 
with applications to $K$-theory \cite{Voevodsky:motivicss}, 
fully faithfulness of the constant functor from the stable topological homotopy category to the stable motivic homotopy category over algebraically closed fields \cite{levine.comparison},
and Milnor's conjecture on quadratic forms \cite{roendigs-oestvaer.hermitian}.

In Section \ref{sec:slice-spectr-sequ} we partially solve (3) in a range suitable for computing the $1$-line; 
the Adem relations, hermitian $K$-theory, and the solution of Milnor's conjecture on quadratic forms in \cite{roendigs-oestvaer.hermitian} help a great deal in our analysis. 
(Prior to Proposition~\ref{prop:Einfty-KQ} it is explained why the exact sequence~(\ref{eq:first-stable-stem}) involves the effective cover $\f_0(\KQ)$ instead of $\KQ$.)
Give and take technical details, 
the problem of computing the $E_{2}$-page of the slice spectral sequence has now been turned into questions about motivic Steenrod operations.

Finally, 
in Sections~\ref{sec:slice-spectr-sequ} and \ref{section:1line} all the pieces fall into place as we show all higher slice differentials relevant for the $1$-line are trivial and determine hidden multiplicative extensions. 
One of the techniques we employ breaks the sphere spectrum $\unit$ into simpler pieces by completing and inverting with respect to the Hopf map $\eta$.
The arithmetic square for $\eta$ allows us to put the pieces back together again.
This is used repeatedly throughout the paper, 
especially in Section~\ref{section:1line} as a preparation for the identification of the $1$-line.

Throughout the paper we employ the following notation.
\vspace{0.05in}

\begin{tabular}{l|l}
$F$, $S$ & field, finite dimensional separated Noetherian base scheme \\
$\Sm_{S}$ & smooth schemes of finite type over $S$ \\
$S^{s,t}$, $\Omega^{s,t}$, $\Sigma^{s,t}$ & motivic $(s,t)$-sphere, $(s,t)$-loop space, $(s,t)$-suspension  \\
$\SH$, $\SH^{\eff}$ & motivic and effective motivic stable homotopy categories of $S$\\ 
$\E$, $\unit=S^{0,0}$ & generic motivic spectrum, the motivic sphere spectrum  \\
$A$, $\unit_{A}$ & abelian group, motivic Moore spectrum of $A$ \\
$\Lambda$, $\mathbf{M}A$ & ring, motivic Eilenberg-MacLane spectra of a $\Lambda$-module $A$ \\
$\MGL$, $\MU$, $\BP$ & algebraic and complex cobordism, Brown-Peterson spectrum \\
$\KGL$, $\KQ$, $\KT$ & algebraic and hermitian $K$-theory, Witt-theory \\
$\f_{q}$, $\f^{q}$, $\s_{q}$ & $q$th effective cover, effective co-cover, and slice functors \\
$H^{\ast,\ast}$, $h^{\ast,\ast}$, $h^{\ast,\ast}_{n}$ & integral, mod-$2$, mod-$n$ motivic cohomology groups, $n>2$ \\
$\partial^a_b\colon h_{a}^{s,t} \to h^{s+1,t}_{b}$ & connecting homomorphism, $a,b\in \NN\cup \{\infty\}$, $h_\infty = H$ \\
$\inc^a_b,\pr^a_b\colon h^{s,t}_{a}\to h^{s,t}_{b}$ & inclusion, projection homomorphism, $a,b\in \NN\cup \{\infty\}$, $h_\infty = H$.
\end{tabular}
\vspace{0.05in}
\noindent

The ring $\Lambda$ is a localization of the integers $\mathbb{Z}$.
Let $\E_{\Lambda}$ be short for $\E\wedge\unit_{\Lambda}$.
Note that $\MGL_{\Lambda}$ is an $E_{\infty}$ motivic spectrum.
Our standard suspension convention is such that $\mathbf{P}^1\simeq \T\simeq S^{2,1}$ and $\mathbf{A}^{1}\minus \{0\}\simeq S^{1,1}$ for the Tate object 
$\T=\mathbf{A}^{1}/\mathbf{A}^{1}\minus \{0\}$.
We write $\dd^{\E}_{1}(q)\colon\s_{q}\E\to\Sigma^{1,0}\s_{q+1}\E$ or simply $\dd^{\E}_{1}$ for the first slice differential of $\E$ \cite[\S2]{roendigs-oestvaer.hermitian}, \cite[\S7]{voevodsky.open}, 
and $d^{r}_{p,q,n}(\E)\colon E^{r}_{p,q,n}(\E)\to E^{r}_{p-1,q+r,n}(\E)$ or simply $d^{r}(\E)$ for the $r$th differential in the $n$-th slice spectral sequence.

\section{Slices of spheres and the first Hopf map} 
\label{sec:slic-motiv-sphere}

In this section we verify Voevodsky's conjecture \cite[Conjecture 9]{voevodsky.open} on the slices of the motivic sphere spectrum over base schemes with compatible coefficients 
in the sense of Definition \ref{def:compatible}.
This applies to Dedekind domains of mixed characteristic as well as to fields as in \cite[\S 8]{levine.comparison}. 
Our analysis includes a discussion of the multiplicative structure of the slices of $\unit_{\Lambda}$.
Relating this to the $E^{2}$-page of the topological Adams-Novikov spectral sequence is a key input in our computations of stable motivic homotopy groups.

\subsection{Compatible pairs and local stable motivic homotopy}

Recall the complex cobordism spectrum has homotopy ring $\pi_{\ast}\MU=\ZZ[x_{j} \vert j\geq 1]$, 
$\vert x_{j}\vert=2j$,
\cite[Theorem 3.1.5(a)]{ravenel.green}.
Quillen proved that $\pi_{\ast}\MU$ is isomorphic to the Lazard ring carrying a universal formal group law,
cf.~\cite[Theorem 1.3.4]{ravenel.green}.
The bigraded homotopy ring $\pi_{\star} \MGL$ is naturally a $\pi_{\ast}\MU$-algebra.
(See \cite[(24) on page 572]{NSO}, \cite[(3.5)]{voevodsky.open} for details.) 
Here we do not distinguish notationally between $x_j\in \pi_{2j}\MU$ and its image $x_j\in \pi_{2j,j}\MGL$. 

\begin{definition}
\label{def:compatible}
Let $S$ be a base scheme and let $\Lambda$ be a localization of the integers $\mathbb{Z}$.
The pair $(S,\Lambda)$ is called {\em compatible} if the following conditions hold:
\begin{enumerate}
\item[(1)] There is a canonically induced equivalence
\[ 
\Phi_{S} \wedge \unit_{\Lambda}
\colon 
\bigl(\MGL/(x_1,x_2,\ldots) \MGL\bigr) \wedge \unit_{\Lambda} 
\to 
\M \Lambda.
\]
\item[(2)]
The first effective cover $\f_{1}(\M \Lambda)$ is contractible \cite[\S2]{voevodsky.open}.
\item[(3)]
The weight zero motivic cohomology groups with $\Lambda$-coefficients of every connected component of $S$ agrees with $\Lambda$ in degree $0$ and the trivial group in nonzero degrees.
\end{enumerate}
\end{definition}

\begin{remark}
\label{rem:compatible}
The pair $(S,\Lambda)$ is compatible if $S$ is ind-smooth over a field or a Dedekind domain of mixed characteristic, and has the property that every positive residue characteristic of 
$S$ is invertible in $\Lambda$ \cite[Theorem 11.3]{spitzweckmz}.
If $F$ is a field of exponential characteristic $p\geq 1$, 
then $\bigl(\Spec(F),\mathbb{Z}[\tfrac{1}{p}]\bigr)$ is a compatible pair.
Let $\mathcal{O}_{T}$ be the ring of $T$-integers in a number field $F$, 
where $T$ is a finite set of primes containing all infinite primes and all primes above a prime number $\ell$.  
Then $(\mathrm{Spec}(\mathcal{O}_{T}),\mathbb{Z}_{(\ell)})$ is a compatible pair.
If $S$ is regular, 
the pair $(S,\mathbb{Q})$ is compatible by \cite[Theorem 10.5]{NSO}, 
\cite[Lemma 6.2]{spitzweck.relations}. 
We conjecture that $(S,\mathbb{Z})$ is compatible provided all the connected components of $S$ are irreducible.
\end{remark}

Next we define the $\Lambda$-local stable motivic homotopy category.
Let $\Spt_{\PP^1}^{\Sigma}\MS$ be the category of motivic symmetric spectra with the stable model structure \cite[Theorem 4.15]{jardine.motivicsymmetricspectra}.
Here $\MS$ is short for the category of pointed simplicial presheaves on $\Sm_{S}$, 
a.k.a.~motivic spaces over $S$.
The following model structure exists by \cite[Theorem 4.1.1(1)]{hirschhorn}.
\begin{definition}
\label{definition:RlocalSH}
The $\Lambda$-local stable model structure is the left Bousfield localization of the stable model structure on $\Spt_{\PP^1}^{\Sigma}\MS$ with respect to the set of naturally induced maps
\[ 
\Sigma^{s,t}\Sigma^{\infty}X_ {+}
\to
\Sigma^{s,t}\Sigma^{\infty}X_ {+} \wedge \unit_{\Lambda}
\]
for all integers $s$, $t$, and $X\in\Sm_{S}$.
Denote the corresponding homotopy category by $\SH_{\Lambda}$.
\end{definition}
\begin{remark}
A map $\alpha\colon\E\to\F$ is a weak equivalence in the $\Lambda$-local stable model structure if and only if $\alpha\wedge \unit_{\Lambda}\colon\E_{\Lambda}\to\F_{\Lambda}$ 
is a stable motivic weak equivalence.
When $\Lambda=\QQ$, 
this defines the rational stable motivic homotopy category.
By \cite[Theorem 3.3.19(1)]{hirschhorn} there exists a left Quillen functor from the stable to the $\Lambda$-local stable model structure on $\Spt_{\PP^1}^{\Sigma}\MS$.
We shall refer to its derived functor as the $\Lambda$-localization functor. 
\end{remark}

The slice filtration in $\SH_{\Lambda}$ is defined exactly the same way as in $\SH$ \cite[\S2]{voevodsky.open} by taking $\Lambda$-localizations of compact generators.
The $\Lambda$-localization functor preserves effective objects, 
so the filtrations are compatible in the sense that $\f_{q}$ and $\s_{q}$ commute with $\Lambda$-localization.

\subsection{Slices of the motivic sphere spectrum}
To determine the slices of the motivic sphere spectrum we compare with algebraic cobordism, as in \cite[\S8]{levine.comparison}. We start with the zero slice.

\begin{lemma}
\label{lem:cone-unit-mgl}
The cone of the unit map $\unit\to\MGL$ lies in $\Sigma^{{2,1}}\SH^{\eff}$.
\end{lemma}
\begin{proof}
Theorem 3.1 in \cite{spitzweck.relations} shows the cone is contained in the full localizing triangulated subcategory of $\SH$ generated by 
the motivic spheres $S^{p,q}$, for $p\in\ZZ$, $q > 0$.
\end{proof} 

\begin{remark}
\label{rem:cone-unit-mgl}
In fact, 
the cone of the unit map $\unit\to \MGL$ is contained in the smallest full subcategory of $\SH$ that contains the spheres $S^{2q,q}$, 
for $q>0$, 
and is closed under homotopy colimits and extensions.
See the proof of \cite[Theorem 5.7]{SO} and also Lemma \ref{lem:connectivity-cone-eta-mgl}.
\end{remark}

By definition the zeroth slice functor $\s_{0}$ is trivial on $\Sigma^{{2,1}}  \SH^{\eff}$. 
Thus $\s_{0}(\unit)\to \s_{0}(\MGL)$ is an isomorphism according to Lemma \ref{lem:cone-unit-mgl}.
The slices of $\MGL$ are known for compatible pairs by \cite[Corollary 4.7]{spitzweck.relations} and the Hopkins-Morel-Hoyois isomorphism \cite[Theorem 7.12]{hoyois},
see \cite[Theorem 3.1]{spitzweckalgebraiccobordism}:
There exists an isomorphism of $\M \Lambda$-modules
\begin{equation}
\label{equation:slicesofMGL}
\s_{q}(\MGL_{\Lambda})
\cong
\Sigma^{{2q,q}} \M \Lambda\otimes\MU_{2q}
=
\bigvee_{I=(i_{1},\cdots,i_{r})}\Sigma^{{2q,q}}\M \Lambda.
\end{equation}
Here the coproduct runs over indices $i_{j}\geq 0$ such that $\sum_{j=1}^{r}2j\cdot i_{j}=2q$,
i.e., 
monomials of total degree $2q$ in the polynomial generators $x_{j}$ for $\pi_{\ast}(\MU)$. 
In particular, 
$\s_{0}(\MGL_{\Lambda})\cong\M \Lambda$ and we obtain a proof of \cite[Conjecture 10]{voevodsky.open} for compatible pairs, 
see e.g., 
\cite[\S6 (iv),(v)]{grso} and \cite[Theorem 3.6.13(6)]{Pelaez} for the $\M \Lambda$-module structure:
\begin{theorem}
\label{theorem:s0isomorphism}
The unit map $\unit\to\MGL$ induces an isomorphism on zero slices.
Hence if the pair $(S,\Lambda)$ is compatible,
$\s_{0}(\unit_{\Lambda})\cong\M \Lambda$ and the slices of any $\unit_{\Lambda}$-module have canonical $\M \Lambda$-module structures.
\end{theorem}
\begin{remark}
The module theory of motivic cohomology in relation to motives was worked out in \cite{RO1}, \cite{RO2}.
In \cite[\S6 (iv),(v)]{grso} and \cite[Theorem 3.6.13(6)]{Pelaez} it is shown that $\s_{q}(\E)$ is a module over the motivic ring spectrum $\s_{0}(\unit)$.
Using the identification~(\ref{equation:slicesofMGL}) it follows that the slice spectral sequence for $\MGL_{\mathbb{Q}}$ degenerates.
\end{remark}

The positive slices of the motivic sphere spectrum will be determined
in several steps, starting with the standard cosimplicial resolution
already employed in \cite[\S3]{voevodsky.open} and recalled in~(\ref{eq:scr}) below.

\begin{proposition}
\label{prop:adams-resol}
The standard cosimplicial $\MGL$-resolution of the motivic sphere spectrum
\begin{displaymath}
\label{equation:hpoikn}
\xymatrix{
\unit
\ar[r] &
\MGL
\ar@<3pt>[r] \ar@<-3pt>[r] & 
\MGL^{\wedge 2} 
\ar@<6pt>[r] \ar[r] \ar@<-6pt>[r] & 
\MGL^{\wedge 3} \cdots }
\end{displaymath}
induces a natural isomorphism on all slices
\begin{equation}
\label{equation:hdgwgv}
\s_{q}(\unit_{\Lambda})
\overset{\cong}{\to}
\underset{\Delta}{\holim}\,
\s_{q}(\MGL_{\Lambda}^{\wedge\bullet}).
\end{equation}
\end{proposition}
\begin{proof}
The basic idea is to work with the Adams tower for the fiber $\overline{\MGL}$ of the unit map of the algebraic cobordism spectrum:
\begin{center}
$\overline{\MGL} 
\to  
\unit  
\to  
\MGL$ \\
$\overline{\MGL}^{\wedge 2}  
\to  
\overline{\MGL}  
\to  
\overline{\MGL}\wedge\MGL$ \\
$\cdots$ \\
$\overline{\MGL}^{\wedge (n+1)}  
\to  
\overline{\MGL}^{\wedge n} 
\to  
\overline{\MGL}^{\wedge n}\wedge\MGL$ \\
$\cdots$ \\
\end{center}

The commutative algebras $\Com\Alg(\Spt_{\PP^1}^{\Sigma}\MS)$ form a combinatorial model category. 
This holds by applying \cite[Theorem 3.6]{hornbostel.preorient} to the commutative operad.
We may assume $\MGL_{\Lambda}$ is a cofibrant object of $\Com\Alg(\Spt_{\PP^1}^{\Sigma}\MS)$ and likewise for $\overline{\MGL}_{\Lambda}$ in $\Spt_{\PP^1}^{\Sigma}\MS$.
Then $\MGL_{\Lambda}^{\wedge n}$ is a commutative monoid with the correct homotopy type in the sense that it is equivalent to the $n$-fold derived smash product of $\MGL_{\Lambda}$ with itself.
We obtain strict diagrams 
\begin{equation}
\label{eq:scr}
\Delta
\to
\Com\Alg(\Spt_{\PP^1}^{\Sigma}\MS);
\;
[n]
\mapsto
\MGL_{\Lambda}^{\wedge (n+1)},
\end{equation}
and
\begin{equation}
\label{equation:tyuhbvo}
\Delta
\to
\Spt_{\PP^1}^{\Sigma}\MS;
\;
[n]
\mapsto
\overline{\MGL}_{\Lambda}^{\wedge m}\wedge\MGL_{\Lambda}^{\wedge (n+1)}.
\end{equation}

We show there is a natural isomorphism 
\begin{equation}
\label{equation:fdghfg}
\s_{q}(\overline{\MGL}_{\Lambda}^{\wedge m})
\overset{\cong}{\to}
\underset{\Delta}{\holim}\,
\s_{q}(\overline{\MGL}_{\Lambda}^{\wedge m}\wedge\MGL_{\Lambda}^{\wedge (\bullet+1)})
\end{equation}
by downward induction on $m$.
If $m\geq q+1$,
the $q$th slices in (\ref{equation:fdghfg}) are trivial according to Lemma \ref{lem:cone-unit-mgl}.
(Note that $\overline{\MGL}\in\Sigma^{2,1}\SH^{\eff}$.) 
For the induction step we claim there exists a naturally induced commutative diagram of distinguished triangles in $\SH$:
\begin{equation}
\label{equation:pjfygf}
\xymatrix{
\s_{q}(\overline{\MGL}_{\Lambda}^{\wedge (m+1)})
\ar[r] \ar[d] &
\underset{\Delta}{\holim}\,
\s_{q}(\overline{\MGL}_{\Lambda}^{\wedge (m+1)}\wedge\MGL_{\Lambda}^{\wedge(\bullet+1)})
\ar[d] \\
\s_{q}(\overline{\MGL}_{\Lambda}^{\wedge m})
\ar[r] \ar[d] &
\underset{\Delta}{\holim}\,
\s_{q}(\overline{\MGL}_{\Lambda}^{\wedge m}\wedge\MGL_{\Lambda}^{\wedge(\bullet+1)})
\ar[d] \\
\s_{q}(\overline{\MGL}_{\Lambda}^{\wedge m}\wedge\MGL_{\Lambda})
\ar[r] & 
\underset{\Delta}{\holim}\,
\s_{q}(\overline{\MGL}_{\Lambda}^{\wedge m}\wedge\MGL_{\Lambda}\wedge\MGL_{\Lambda}^{\wedge(\bullet+1)})  }
\end{equation}
In effect, 
the model categorical considerations concerning functorial (co)fibrant replacements in \cite[\S 3.1]{grso} combined with \eqref{equation:tyuhbvo} show there is a strict diagram 
\begin{equation}
\Delta
\to
\Spt_{\PP^1}^{\Sigma}\MS;
\;
n
\mapsto
\s_{q}(\overline{\MGL}_{\Lambda}^{\wedge m}\wedge\MGL_{\Lambda}^{\wedge n}).
\end{equation}
The slice functor $\s_{q}$ is by definition a composition of functorial colocalization and localization functors which are Quillen functors between stable model structures on 
$\Spt_{\PP^1}^{\Sigma}\MS$ \cite[\S 3.1]{grso}.
It follows that $\s_{q}$ is a triangulated functor of $\SH$.
This verifies our claim concerning \eqref{equation:pjfygf}.

Lemma \ref{lemma:hfgcjke} given below shows the lower horizontal map in (\ref{equation:pjfygf}) is an isomorphism.
By the induction hypothesis there is an isomorphism
\begin{equation*}
\s_{q}(\overline{\MGL}_{\Lambda}^{\wedge (m+1)})
\overset{\cong}{\to}
\underset{\Delta}{\holim}\,
\s_{q}(\overline{\MGL}_{\Lambda}^{\wedge (m+1)}\wedge\MGL_{\Lambda}^{\wedge(\bullet+1)}).
\end{equation*}
It follows that the middle horizontal map in (\ref{equation:pjfygf}) is also an isomorphism.
The desired isomorphism in \eqref{equation:hdgwgv} is the special case of \eqref{equation:fdghfg} for $m=0$. 
\end{proof}

\begin{lemma}
\label{lemma:hfgcjke}
For every $\MGL_{\Lambda}$-module $\mathscr{M}$ there are weak equivalences 
\[ 
\mathscr{M}
\xrightarrow{\sim}
\underset{\Delta}{\holim}\,\mathscr{M}\wedge\MGL_{\Lambda}^{\wedge(\bullet+1)}
\]
and
\[
\s_{q}(\mathscr{M})
\xrightarrow{\sim}
\underset{\Delta}{\holim}\, \s_{q}(\mathscr{M}\wedge\MGL_{\Lambda}^{\wedge(\bullet+1)}).
\]
\end{lemma}
\begin{proof}
Let $\widetilde{\Delta}$ denote the split augmented simplicial category with naturally ordered objects $[-1]_{+}=\{+\}$, $[n]_{+}=\{+,0,\dots,n\}$ for $n\geq 0$. 
Maps in $\widetilde{\Delta}$ are monotone maps preserving $+$.
Note that $\widetilde{\Delta}$ has an initial object.
The evident functor from the simplicial category $\Delta$ to $\widetilde{\Delta}$ is homotopy left cofinal in the sense of \cite[Definition 19.6.1]{hirschhorn}.
Thus by \cite[Theorem 19.6.7(2)]{hirschhorn}, 
for any split coaugmented cosimplicial diagram $\mathsf{X}$ there is a weak equivalence 
\[ 
\mathsf{X}_{-1}
\cong
\underset{\widetilde{\Delta}}{\holim}\,\mathsf{X}
\xrightarrow{\sim}
\underset{\Delta}{\holim}\,\mathsf{X}.
\]
The standard cosimplicial $\MGL_{\Lambda}$-resolution of $\mathscr{M}$ is a split coaugmented cosimplicial diagram.
Applying the $q$th slice functor preserves this structure.
\end{proof}

For any $q\geq 0$ we consider $\pi_{2q}\MU^{\wedge(\bullet+1)}$ as a cosimplicial abelian group. 
According to \cite[Conjecture 6]{voevodsky.open} the slices of the cosimplicial motivic spectrum $\MGL_{\Lambda}^{\wedge(\bullet+1)}$ are given by 
\begin{equation}
\label{equation:sgfdgd}
\s_{q}(\MGL_{\Lambda}^{\wedge(\bullet+1)}) 
\cong
\Sigma^{2q,q}\M \Lambda\otimes\pi_{2q}\MU^{\wedge(\bullet+1)}.
\end{equation}
In cosimplicial degree zero, 
(\ref{equation:sgfdgd}) holds for compatible pairs by (\ref{equation:slicesofMGL}).
Inductively, 
this implies \eqref{equation:sgfdgd} in every cosimplicial degree on account of the smash product decompositions 
\begin{equation*}
\MGL_{\Lambda} \wedge \MGL_{\Lambda} = \MGL_{\Lambda}[b_1,b_2,b_3,\ldots], \,
\MU\wedge \MU = \MU[b_1,b_2,b_3,\ldots], 
\end{equation*}
from \cite[Lemma 6.4(i)]{NSO}. Here $b_j$ is the standard choice of generator
in the dual Landweber-Novikov algebra \cite{novikov}, \cite[Theorem 4.1.11]{ravenel.green}, 
with the standard degree $\lvert b_j\rvert =2j$.

In the following we refine \eqref{equation:sgfdgd} to a weak equivalence of cosimplicial $\M \Lambda$-modules for a fixed compatible pair $(S,\Lambda)$, 
where $S$ is connected.
Note that $\s_{q}(\MGL_{\Lambda}^{\wedge(\bullet+1)})$ is a cosimplicial $\M \Lambda$-module by Theorem \ref{theorem:s0isomorphism}.
The category of $\M \Lambda$-modules $\M \Lambda-\Mod$ is enriched over chain complexes of $\Lambda$-modules $\Cpx(\Lambda-\Mod)$ by viewing $\M \Lambda$ as an $E_{\infty}$ object
in the motivic symmetric spectrum category $\Spt_{\PP^1}^{\Sigma}\Pre\Cpx(\Lambda-\Mod)$ for presheaves $\Pre\Cpx(\Lambda-\Mod)$ of chain complexes of $\Lambda$-modules on $\Sm_{S}$.
The corresponding hom objects $\uHom_{\M \Lambda-\Mod}$ will be implicitly derived by taking (co)fibrant replacements in the stable model structure on 
$\M \Lambda-\Mod$ \cite[Proposition 38]{RO2}.

The cosimplicial chain complex of abelian groups  
\begin{equation*}
\uHom_{\M \Lambda-\Mod}(\Sigma^{2q,q}\M \Lambda,\s_{q}(\MGL_{\Lambda}^{\wedge(\bullet+1)}))
\end{equation*}
yields by adjunction a map of cosimplicial $\M \Lambda$-modules
\begin{equation}
\label{equation:cosimplicialMRweakequivalence}
\Sigma^{2q,q}\M \Lambda\otimes\uHom_{\M \Lambda-\Mod}(\Sigma^{2q,q}\M \Lambda,\s_{q}(\MGL_{\Lambda}^{\wedge(\bullet+1)}))
\xrightarrow{\sim} 
\s_{q}(\MGL_{\Lambda}^{\wedge(\bullet+1)}).
\end{equation}
This is a weak equivalence on account of the $\M \Lambda$-module isomorphism in \eqref{equation:slicesofMGL} and the identification of the hom object of motives $\uHom_{\M \Lambda-\Mod}(\M \Lambda,\M \Lambda)$ 
with the chain complex comprised of a single copy of $\Lambda$ in degree zero,
cf.~Definition \ref{def:compatible}(3). We claim there is an isomorphism of cosimplicial objects in $\Cpx(\Lambda-\Mod)$ 
\begin{equation}
\label{equation:cosimplicalisomorphism}
\uHom_{\M \Lambda-\Mod}(\Sigma^{2q,q}\M \Lambda,\s_{q}(\MGL_{\Lambda}^{\wedge(\bullet+1)}))
\cong
\pi_{2q}\MU^{\wedge(\bullet+1)} \otimes \Lambda.
\end{equation}
Here 
$\pi_{2q}\MU^{\wedge(\bullet+1)}$ in \eqref{equation:cosimplicalisomorphism} is concentrated in degree zero.
For connective chain complexes $\Cpx_{\ge 0}(\Lambda-\Mod)$ of $\Lambda$-modules there are adjunctions 
\begin{equation}
\label{equation:chaincomplexadjunctions} 
\xymatrix{
\Lambda-\Mod \ar@<-1pt>[r] & \ar@<-4pt>[l]
\Cpx_{\ge 0}(\Lambda-\Mod) \ar@<4pt>[r] & \ar@<1pt>[l] 
\Cpx(\Lambda-\Mod). 
}
\end{equation}
Here a $\Lambda$-module maps to the corresponding chain complex concentrated in degree zero via the inclusion of $\Cpx_{\ge 0}(\Lambda-\Mod)$ into $\Cpx(\Lambda-\Mod)$.
In the opposite direction, 
we use the good truncation functor and the zeroth homology functor. 
For a cosimplicial object of $\Cpx(\Lambda-\Mod)$ concentrated in degree zero, 
such as $\pi_{2q}\MU^{\wedge(\bullet+1)}$ considered above,
the units and counits of the adjunctions in \eqref{equation:chaincomplexadjunctions} furnish canonical isomorphisms producing isomorphic objects in $\Cpx(\Lambda-\Mod)$,
hence the isomorphism in \eqref{equation:cosimplicalisomorphism}. To summarize, 
\eqref{equation:cosimplicialMRweakequivalence} and \eqref{equation:cosimplicalisomorphism} imply the weak equivalence of cosimplicial $\M \Lambda$-modules
\begin{equation}
\label{equation:summaryweakequivalence}
\Sigma^{2q,q}\M \Lambda\otimes \pi_{2q}\MU^{\wedge(\bullet+1)}
\xrightarrow{\sim}
\s_{q}(\MGL_{\Lambda}^{\wedge(\bullet+1)}).
\end{equation}

\begin{lemma}
\label{lemma:pidtot}
Suppose $\Lambda$ is a principal ideal domain and $M^\bullet$ is a levelwise free cosimplicial $\Lambda$-module such that $\Tot(M^\bullet)$ is a 
perfect chain complex of $\Lambda$-modules, 
i.e., 
its homology is finitely generated and trivial in almost all degrees. 
For $C \in \Cpx(\Lambda-\Mod)$, 
let $C \otimes M^\bullet$ be the cosimplicial object in $\Cpx(\Lambda-\Mod)$ given degreewise by $[n] \mapsto C \otimes M^n$.
Then the natural map $C \otimes^{\mathbb{L}}\Tot(M^\bullet)\to\Tot(C \otimes M^\bullet)$ is a quasi-isomorphism.
\end{lemma}
\begin{proof}
Note that the tensor $-\otimes M^\bullet$ and the derived tensor $-\otimes^{\mathbb{L}} M^\bullet$ coincide because $M^\bullet$ is levelwise free.
Without loss of generality we may assume that $C$ is cofibrant in the projective model structure on $\Cpx(\Lambda-\Mod)$, 
see e.g., 
\cite[Theorem 2.3.11]{hovey.modelcategories}.
Let $N \in \Cpx(\Lambda-\Mod)$ be the degreewise free (non-normalized) chain complex associated to $M^\bullet$. 
Since $\Lambda$ is a PID it follows that $N$ is projectively cofibrant (as a sum of length two chain complexes with free entries).
Let $C \boxtimes N$ be the double complex obtained by taking degreewise tensor products of entries from $C$ and $N$, 
i.e.,
$(C \boxtimes N)^{p,q}= C^{p}\otimes_{\Lambda} N^{q}$ for all integers $p$, $q$.
Then $\Tot(C \otimes M^\bullet) \simeq \Tot^{\prod}(C \boxtimes N)$, 
where $\Tot^{\prod}$ denotes the total complex associated to a double complex given degreewise as $\Tot^{\prod}(C \boxtimes N)^{n}={\prod}(C \boxtimes N)^{i,j}$ for all $i$, 
$j$ such that $i+j=n$.
By perfectness of $\Tot(M^\bullet)$ there exists a quasi-isomorphism $N'\to N$, 
where $N'\in \Cpx(\Lambda-\Mod)$ is comprised of finitely generated free $\Lambda$-modules and $N'^{p}=0$ for almost all $p$.
Since $N'$ and $N$ are both (co)fibrant in the projective model structure,  
cf.~\cite[p.~44]{hovey.modelcategories}, 
this map is in fact a homotopy equivalence. 
It follows that $\Tot^{\prod}(C \boxtimes N') \to \Tot^{\prod}(C \boxtimes N)$ is also a homotopy equivalence.
Since $N'$ is nonzero in only finitely many degrees we conclude there is a canonical isomorphism $C \otimes N' \cong \Tot^{\prod}(C \boxtimes N')$.
\end{proof}

We are ready to finish our identification of the slices of the motivic sphere spectrum. 
\begin{theorem}
\label{theorem:sphereslices}
Suppose $(S,\Lambda)$ is a compatible pair and $A$ is a $\Lambda$-module. 
Then there are isomorphisms of $\M \Lambda$-modules
\begin{equation*}
\s_{q}(\unit_{A})
\cong
\Sigma^{2q,q}\mathbf{M}A\otimes\Tot(\pi_{2q}\MU^{\wedge(\bullet+1)})
\iso
\bigvee_{p\geq 0} \Sigma^{2q,q}\mathbf{M}A \otimes\Ext_{\MU_\ast\MU}^{p,2q}(\MU_\ast,\MU_\ast)[-p].
\end{equation*}
In particular, 
the slices of the $\Lambda$-local sphere spectrum are
\[ 
\s_{q}(\unit_{\Lambda})
\cong
\bigvee_{p\geq 0} \Sigma^{2q-p,q}\mathbf{M}(\Lambda \otimes\Ext_{\MU_\ast\MU}^{p,2q}(\MU_\ast,\MU_\ast)).
\]
\end{theorem}
\begin{proof}
By Proposition~\ref{prop:adams-resol} and \eqref{equation:summaryweakequivalence} we have
\[ \s_{q}(\unit_{\Lambda})
\overset{\cong}{\to}
\underset{\Delta}{\holim}\,
\s_{q}(\MGL_{\Lambda}^{\wedge\bullet}) \xleftarrow{\sim}
\underset{\Delta}{\holim}\,
\Sigma^{2q,q}\M \Lambda\otimes \pi_{2q}\MU^{\wedge(\bullet+1)}.
\]
Using the identification
$\Tot(\pi_{2q}\MU^{\wedge(\bullet+1)})
\cong 
\underset{\Delta}{\holim}\, 
\pi_{2q}\MU^{\wedge(\bullet+1)}$, 
we claim there is a weak equivalence of $\M \Lambda$-modules
\begin{equation*}
\xi\colon \Sigma^{2q,q}\M \Lambda\otimes\Tot(\pi_{2q}\MU^{\wedge(\bullet+1)})
\xrightarrow{}
\underset{\Delta}{\holim}\, 
\Sigma^{2q,q}\M \Lambda\otimes\pi_{2q}\MU^{\wedge(\bullet+1)}.
\end{equation*}
It suffices to show that,
for every generator $G$ of $\M \Lambda-\Mod$, the map 
$\uHom_{\M \Lambda-\Mod}(G,\xi)$ is an isomorphism.
Hence let $G$ be a 
free $\M \Lambda$-module on a shifted motivic symmetric suspension 
spectrum of a smooth scheme over $S$, 
cf.~\cite[Proposition 38]{RO2}. 
Setting $C= \uHom_{\M \Lambda-\Mod}(G,\Sigma^{2q,q}\M\Lambda)$, the map $\uHom_{\M \Lambda-\Mod}(G,\xi)$ takes the form
\[ 
C\otimes \Tot(\pi_{2q}\MU^{\wedge(\bullet+1)})
\to
\Tot (C\otimes \pi_{2q}\MU^{\wedge(\bullet+1)}).
\]
Since $\MU^{\wedge(\bullet+1)}$ is levelwise free and $\Tot(\pi_{2q}\MU^{\wedge(\bullet+1)})$ is a perfect chain complex of abelian groups, 
Lemma \ref{lemma:pidtot} implies our claim. 
Finally, the result for $\s_{q}(\unit_{\Lambda})$ follows since every chain complex of abelian groups is quasi-isomorphic to its homology, 
considered as a chain complex in a natural way.
\end{proof}

We refer to Appendix~\ref{sec:ext-groups-complex} for explicit computations of the $\Ext$-groups in Theorem \ref{theorem:sphereslices}. 
These form the terms of the $E^2$-page of the topological Adams-Novikov-sequence,
see e.g., 
\cite[Chapter 4, 7, Appendix A3]{ravenel.green}, and \cite[Table 2, 3]{Zahler}.
Appendix \ref{sec:ext-groups-complex} recalls the $\alpha$-family $\overline{\alpha}_{i}$ of generators of $\Ext^{1,2i}_{\BP_\ast\BP}(\BP_\ast,\BP_\ast)$ listed in  
\cite[Theorems 5.3.5(b), 5.3.6(b), 5.3.7]{ravenel.green}.
Via the direct sum decomposition 
\[ 
\Ext^{1,2i}_{\MU_\ast\MU}(\MU_\ast,\MU_\ast)
=
\bigoplus
\Ext^{1,2i}_{\BP_\ast\BP}(\BP_\ast,\BP_\ast), 
\]
indexed over all prime numbers in \cite[\S 11]{novikov},
the $\Ext$-group elements for $\BP$ correspond to generators of the corresponding direct summands of $\s_{q}(\unit_{\Lambda})$ given in Theorem \ref{theorem:sphereslices}. 

\begin{corollary}
\label{corollary:small-slices}
The suspensions $\Sigma^{q,q}\M \Lambda/2$ generated by $\alpha_{1}^{q}$ and $\Sigma^{q+2,q}\M \Lambda/2$ generated by $\alpha_{1}^{q-3}\alpha_{3}$ 
are direct summands of $\s_{q}(\unit_{\Lambda})$ for $q>0$ respectively $q>2$. 
The first seven positive slices of the $\Lambda$-local sphere spectrum $\unit_{\Lambda}$ are:
\begin{align*} 
\s_1(\unit_{\Lambda}) & \cong \Sigma^{1,1} \M \Lambda/2 \\
\s_2(\unit_{\Lambda}) & \cong \Sigma^{2,2} \M \Lambda/2 \vee \Sigma^{3,2} \M \Lambda/12 \\
\s_3(\unit_{\Lambda}) & \cong \Sigma^{3,3} \M \Lambda/2 \vee \Sigma^{5,3} \M \Lambda/2  \\
\s_4(\unit_{\Lambda}) & \cong \Sigma^{4,4} \M \Lambda/2 \vee \Sigma^{6,4} \M (\Lambda/2)^2 \vee \Sigma^{7,4}\M \Lambda/240 \\
\s_5(\unit_{\Lambda}) & \cong \Sigma^{5,5} \M \Lambda/2 \vee \Sigma^{7,5} \M \Lambda/2 \vee \Sigma^{8,5} \M(\Lambda/2)^2\vee \Sigma^{9,5}\M \Lambda/2 \\
\s_6(\unit_{\Lambda}) & \cong \Sigma^{6,6} \M \Lambda/2 \vee \Sigma^{8,6} \M \Lambda/2 \vee \Sigma^{9,6} \M(\Lambda/2)^2\vee \Sigma^{10,6}\M \Lambda/6 \vee \Sigma^{11,6}\M \Lambda/504 \\
\s_7(\unit_{\Lambda}) & \cong \Sigma^{7,7} \M \Lambda/2 \vee \Sigma^{9,7} \M \Lambda/2 \vee \Sigma^{10,7}\M \Lambda/2\vee \Sigma^{11,7} \M \Lambda/2 \vee \Sigma^{12,7}\M \Lambda/2 \vee \Sigma^{13,7} \M \Lambda/2.
\end{align*}
Here the direct summand $\Sigma^{3,2} \M \Lambda/4$ of $\s_{2}(\unit_{\Lambda})$ is generated by $\alpha_{2/2}$ and the direct summand $\Sigma^{6,4} \mathbf{M}(\Lambda/2)^{2}$ of $\s_4(\unit_{\Lambda})$ 
is generated by $\alpha_{1}\alpha_{3}$ and $\beta_{2/2}$.
The direct summand $\Sigma^{4q+1,2q+1}\M \Lambda/2$ of $\s_{2q+1}(\unit_{\Lambda})$ is generated by $\alpha_{2q+1}$.
\end{corollary}

\subsection{Cellularity and slices of $K$-theory spectra}
Recall that the subcategory of cellular spectra $\SH_{\cell,\Lambda}$ in $\SH_{\Lambda}$ is the smallest full 
localizing subcategory that contains all bigraded suspensions of $\unit_{\Lambda}$,
cf.~\cite[\S2.8]{dugger-isaksen.cell}, 
\cite[\S4]{voevodsky.open}.

\begin{theorem}
\label{theorem:KQcellular}
Suppose $(S,\Lambda)$ is a compatible pair and the base scheme $S$ contains no 
point whose residue characteristic equals two. 
Then hermitian $K$-theory $\KQ_{\Lambda}$ is a cellular spectrum in $\SH_{\Lambda}$ and compatible with pullbacks.
\end{theorem}
\begin{proof}
The proof of cellularity given in \cite{RSO} carries over to the setting of compatible pairs.
\end{proof}

Next we extend the computation of the slices of hermitian $K$-theory in \cite{roendigs-oestvaer.hermitian}.
This will be used to analyze the map between slices induced by the unit map $\unit_{\Lambda}\to\KQ_{\Lambda}$.
\begin{theorem}
\label{theorem:slices-kq}
Suppose $(S,\Lambda)$ is a compatible pair and the base scheme $S$ contains no point whose residue characteristic equals two. 
Then the slices of $\KQ_{\Lambda}$ are given by
\[  
\s_{q}(\KQ_{\Lambda})
\cong
\begin{cases}
\Sigma^{2q,q} \M \Lambda \vee \bigvee_{i<\frac{q}{2}} \Sigma^{2i+q,q} \M \Lambda/2 & q\equiv 0 \bmod 2 \\
\bigvee_{i<\frac{q+1}{2}} \Sigma^{2i+q,q} \M \Lambda/2 & q\equiv 1 \bmod 2.
\end{cases}
\]
\end{theorem}
\begin{proof}
This follows from Corollaries \ref{cor:cell-slicewise-cell}, 
\ref{cor:cell-slice-basechange}, 
Remark \ref{rem:cell-filt-basechange} applied to every connected component of the base scheme, 
and the computation of $\s_{q}(\KQ)$ over fields of characteristic unequal to two in \cite[Theorem 4.18]{roendigs-oestvaer.hermitian}.
\end{proof}

Recall that a motivic spectrum $\mathsf{E} \in \SH_{\Lambda}$ is called {\em slice-wise cellular} if $\s_{q}(\mathsf{E})$ is contained in the full 
localizing triangulated subcategory of $\SH_{\Lambda}$ generated by the $q$th suspension $\Sigma^{2q,q}\M \Lambda$ \cite[Definition 4.1]{voevodsky.open}.
Let $\mathbf{D}_{\M \Lambda}$ denote the homotopy category of $\M \Lambda-\Mod$. 
Replacing $\SH_{\Lambda}$ by $\mathbf{D}_{\M \Lambda}$ gives an equivalent definition of slice-wise cellular spectra. 

\begin{corollary}
\label{cor:cell-slicewise-cell}
Every cellular spectrum in $\SH_{\Lambda}$ is slice-wise cellular.
\end{corollary}
\begin{proof}
This follows from Theorem \ref{theorem:sphereslices}.
\end{proof}

In the next result, base schemes are included in the notation for clarity. 

\begin{corollary} 
\label{cor:cell-slice-basechange}
Let $\phi\colon T\to S$ be a map between connected base schemes and suppose $(S,\Lambda)$ and $(T,\Lambda)$ are compatible pairs. 
Then the pullback functor $\phi^* \colon \SH(S)_{\cell,\Lambda} \to \SH(T)_{\cell,\Lambda}$ commutes with slices.
\end{corollary}
\begin{proof}
This holds for the sphere spectrum $\unit_{\Lambda}$ by Theorem \ref{theorem:sphereslices} because $\phi^*(\M \Lambda_{S})\cong \M \Lambda_{T}$ according to \cite[Theorem 9.16, \S10]{spitzweckmz}. 
The general case follows by attaching cells and noting that slices commute with homotopy colimits, 
see \cite[Corollary 4.5]{spitzweck.relations}, \cite[Lemma 4.2]{voevodsky.open}.
\end{proof}

\begin{remark}
\label{rem:cell-filt-basechange}
Let $\mathbf{D}_{\Lambda}$ denote the derived category of the ring $\Lambda$.
The functor $\mathbf{D}_{\Lambda} \to \mathbf{D}_{\M \Lambda}$ sending the tensor unit $\Lambda$ in degree zero to $\Sigma^{0,q}\M \Lambda$ is a full embedding for every connected component of $S$. 
(This follows from Definition \ref{def:compatible}(3).)
Hence for every $\E\in\SH_{\cell,\Lambda}$ and integer $q$, 
there exists a unique $C \in \mathbf{D}_{\Lambda}$ such that $\s_{q}(\mathsf{E})=\Sigma^{0,q}\M(\Lambda \otimes C)$.
In this way we may view slices as objects of $\mathbf{D}_{\Lambda}$.
For $\phi\colon\Spec(F)\to S$, 
where $F$ is a field such that $(F,\Lambda)$ is compatible, 
the slices of $\mathsf{E} \in \SH(S)_{\cell,\Lambda}$ are determined by the slices of $\phi^*(\mathsf{E})$ over every connected component of $S$.
In particular, 
this applies to $\KQ$ by Theorem \ref{theorem:KQcellular}.
\end{remark}

Algebraic $K$-theory $\KGL_{\Lambda}$ is a cellular motivic spectrum by \cite[Theorem 6.2]{dugger-isaksen.cell}.
Its slices are known for perfect fields by the works of Levine \cite[\S6.4, 11.3]{levine.coniveau} and Voevodsky \cite{Voevodsky:motivicss}, 
\cite{voevodsky.zero}.
By base change the slices of $\KGL_{\Lambda}$ are known for all fields;
in fact, 
every field is essentially smooth over a perfect field by \cite[Lemma 2.9]{hko.positivecharacteristic}, 
and \cite[Lemma 2.7(1)]{hko.positivecharacteristic} verifies the hypothesis of \cite[Theorem 2.12]{pp:functoriality} for an essentially smooth map.
Hence, 
as in Theorem \ref{theorem:slices-kq}, 
we obtain:
\begin{theorem}
\label{theorem:slices-kgl}
Suppose $(S,\Lambda)$ is a compatible pair.
Then the slices of $\KGL_{\Lambda}$ are given by  
\[  
\s_{q}(\KGL_{\Lambda})
\cong
\Sigma^{2q,q} \M \Lambda.
\]
\end{theorem}

\subsection{Multiplicative structure on sphere slices}
Next we discuss the multiplicative properties of the isomorphisms in Theorem \ref{theorem:sphereslices}. 
Note that $[n] \mapsto \pi_{2\ast}\MU_{\Lambda}^{\wedge (n+1)}$ is a cosimplicial graded $E_\infty$ algebra of $\Lambda$-modules. 
Its total object $\Tot(\pi_{\ast}\MU_{\Lambda}^{\wedge(\bullet+1)})$ is a graded $E_\infty$ algebra of $\Lambda$-modules, 
i.e., 
in $E_\infty(\Cpx(\Lambda-\Mod)^\mathbb{Z})$.
To form the totalization as in \cite[Definition 19.8.1(2)]{hirschhorn} we employ the model structure stated in \cite[Corollary, p.~286]{mandell}.
Its homology groups, 
the $E_2$-terms of the Adams-Novikov spectral sequence with $\Lambda$-coefficients,
form an Adams graded object in the category of graded $\Lambda$-modules \cite[\S3]{NSO}.
The Adams grading of $\pi_{2\ast}$ is $2\ast$ while the algebras $\pi_{2\ast}\MU_{\Lambda}^{\wedge (n+1)}$ are strictly commutative in graded $\Lambda$-modules of degree $0$. 
For $M \in \Cpx(\Lambda-\Mod)^\mathbb{Z}$ we denote by $M\{k\} \in \Cpx(\Lambda-\Mod)$ the Adams degree $2k\in \mathbb{Z}$ part of $M$. 

The multiplication in $\Tot(\pi_{\ast}\MU^{\wedge(\bullet+1)})$ furnishes maps in the derived category $\mathbf{D}_{\Lambda}$ 
\begin{equation}\label{eq:mu-mult}
m_{k,l} \colon \Tot(\pi_{\ast}\MU^{\wedge(\bullet+1)}_{\Lambda})\{k\} \otimes^\mathbb{L}_{\Lambda} \Tot(\pi_{\ast}\MU^{\wedge(\bullet+1)}_{\Lambda})\{l\} \to \Tot(\pi_{\ast}\MU^{\wedge(\bullet+1)}_{\Lambda})\{k+l\}.
\end{equation}

\begin{theorem}
\label{theorem:mult-adams-resol}
Let $\Tot$ be short for $\Tot(\pi_{\ast}\MU^{\wedge(\bullet+1)}_{\Lambda})$ and $k$, $l$ be nonnegative integers.
There is a commutative diagram:
\[ 
\xymatrix{
\s_k(\unit_{\Lambda}) \wedge_{\s_0(\unit_{\Lambda})} \s_l(\unit_{\Lambda}) \ar[r] \ar[d]_\cong & \s_{k+l}(\unit_{\Lambda}) \ar[dd]^\cong \\
(\Sigma^{2k,k} \M \Lambda \otimes \Tot\{k\}) \wedge_{\s_0(\unit_{\Lambda})} (\Sigma^{2l,l} \M \Lambda \otimes \Tot\{l\}) \ar[d]_\cong & \\
\Sigma^{2(k+l),k+l} \M \Lambda \otimes (\Tot\{k\} \otimes_{\Lambda}^\mathbb{L} \Tot\{l\}) \ar[r]^-{\mathrm{id} \otimes m_{k,l}} & 
\Sigma^{2(k+l),k+l} \M \Lambda \otimes \Tot\{k+l\} 
}
\]
\end{theorem}
\begin{proof}
We may assume the base scheme is connected.
Let $\HHH \Lambda$ denote the Eilenberg-MacLane spectrum in symmetric spectra. 
It affords a symmetric monoidal functor $\HHH \Lambda-\Mod\to\M \Lambda-\Mod$. 
We denote the associated enriched hom objects by $\uHom_{\M \Lambda-\Mod}$.

Let $\mathscr{F}$ be the free associative algebra over $\M \Lambda$ on a cofibrant model of $\PP^{1}\simeq S^{2,1}$ viewed as a graded associative algebra,
i.e., 
an algebra in $\M \Lambda-\Mod^\mathbb{Z}$. 
We note the multiplication maps $\mathscr{F}\{k\} \wedge_{\M \Lambda} \mathscr{F}\{l\} \to \mathscr{F}\{k+l\}$ are isomorphisms.
Thus for every associative algebra $\mathscr{A} \in \M \Lambda-\Mod^\mathbb{Z}$, 
the collection $(\uHom_{\M \Lambda-\Mod}(\mathscr{F}\{k\},\mathscr{A}\{k\}))_k$ has the natural structure of an associative algebra in $\HHH \Lambda-\Mod^\mathbb{Z}$, 
which we denote by $\mathscr{A}^\mathscr{F}$.

The total slice functor $\s_{\ast}$ sends $E_\infty$ motivic spectra to graded $E_\infty$ $\M \Lambda$-module spectra in a functorial way \cite[Theorem 5.1]{grso}, 
\cite[Theorems 3.8, 3.13]{grso2}.
Thus the assignment $[n]\mapsto \s_{\ast}(\MGL_{\Lambda}^{\wedge (n+1)})$ defines a cosimplicial graded $E_\infty$ $\M \Lambda$-module spectrum.
Now let $\mathscr{A}^\bullet$ be a fibrant replacement of the latter as a cosimplicial graded associative algebra. 
This uses the semi model structure on associative algebras for the $\Sigma$-cofibrant operad in \cite[Theorem 4.7]{spitzweck.thesis}.
Then $(\mathscr{A}^\bullet)^\mathscr{F}$ is a cosimplicial object of associative algebras in $\HHH \Lambda-\Mod^\mathbb{Z}$.
It is equivalent to $\HHH(\pi_{2*}\MU_{\Lambda}^{\wedge (\bullet+1)})$ using the truncation methods in the proof of \eqref{equation:summaryweakequivalence}.
By adjointness there is a canonical weak equivalence of cosimplicical graded associative $\M \Lambda$-algebra spectra
$$
\mathscr{F} \wedge^d \HHH(\pi_{2*}\MU_{\Lambda}^{\wedge (\bullet+1)}) 
\xrightarrow{\sim} 
\mathscr{A}^\bullet.
$$ 
Here $\wedge^d$ ($d$ is short for ``diagonal'' in order to distinguish it from the convolution product, for which there is no such map) 
stands for the levelwise smash product with respect to the grading. 
This follows as in the proof of Theorem \ref{theorem:sphereslices} by considering the full slice functor $\s_*$.
Likewise, 
there is a canonical weak equivalence of graded associative $\M \Lambda$-algebras
$$
\mathscr{F} \wedge^d \Tot(\HHH(\pi_{2*}\MU_{\Lambda}^{\wedge (\bullet+1)})) 
\xrightarrow{\sim} 
\Tot(\mathscr{A}^\bullet).
$$
Finally, 
applying the weak equivalence of graded associative $\M \Lambda$-algebras
\begin{equation}
\label{equation:sstarunit}
\s_*(\unit_{\Lambda}) 
\xrightarrow{\sim} 
\Tot(\s_*(\MGL_{\Lambda}^{\wedge (\bullet+1)}))
\xrightarrow{\sim} 
\Tot(\mathscr{A}^\bullet)
\end{equation}
yields the desired commutative diagram.
The proof of \eqref{equation:sstarunit} is the same as for a fixed slice degree, 
cf.~Proposition \ref{prop:adams-resol}.
\end{proof}

\begin{remark}
We expect that $\mathscr{F} \wedge^d \Tot(H(\pi_{2*}\MU_{\Lambda}^{\wedge (\bullet+1)}))\to \Tot(\mathscr{A}^\bullet)$ is an $E_\infty$ map.
These objects acquire natural $E_\infty$ structures; 
the former because $\M \Lambda$ is strongly periodizable, 
see \cite[Theorem 8.2]{spitzweckmz}.

We use Theorem \ref{theorem:mult-adams-resol} to show the motivic Hopf map $\eta$ is non-nilpotent \cite[Theorem 4.7]{MorelICM2006}.
The topological Hopf map $\eta_{\Top}$ is detected by $\alpha_{1}$ in the Adams-Novikov spectral sequence. 
Theorem \ref{theorem:sphereslices} shows that $\eta\in\pi_{1,1}\unit_{\Lambda}$ is detected by $1\in \pi_{1,1}\s_1(\unit_{\Lambda})\cong h^{0,0}$ on the 
$E^1$-page of the first slice spectral sequence of $\unit_{\Lambda}$. 
Since all differentials entering or exiting this group are trivial, 
cf.~Theorem \ref{theorem:sphereslices} and Lemma \ref{lem:first-diff-unit-1},
this class survives to the $E^\infty$-page. 
Theorem \ref{theorem:mult-adams-resol} implies that its $n$th power is the nontrivial element of $\pi_{n,n}\s_n(\unit_{\Lambda})\cong h^{0,0}$ in the $n$th slice spectral sequence; 
it is also a permanent cycle by Theorem \ref{theorem:sphereslices} and Lemma \ref{lem:first-diff-unit-1}.
\end{remark}

\begin{corollary}
\label{corollary:etanonnilpotent}
The $n$-fold multiplication map by the generator $\alpha_{1}$ yields the canonical projection
\begin{equation*}
\label{equation:sstarunit-2}
\s_1(\unit_{\Lambda}) \wedge_{\s_0(\unit_{\Lambda})}  \cdots \wedge_{\s_0(\unit_{\Lambda})} \s_1(\unit_{\Lambda})
\xrightarrow{} 
\Sigma^{n,n}\M \Lambda/2
\end{equation*}
on the direct summand $\Sigma^{n,n}\M \Lambda/2$ of 
$\s_n(\unit_{\Lambda})$.
Iterating the first motivic Hopf map $\eta$ yields a nontrivial element $\eta^{n}\neq 0\in\pi_{n,n}\unit_{\Lambda}$ for $n\geq 2$.
\end{corollary}

Recall that $\Ext_{\MU_\ast\MU}^{1,2q}(\MU_\ast,\MU_\ast)$ is a finite 
cyclic group of order $a_{q}$, where $a_q=2$ for 
$q$ odd, and $a_{2q}$ is equal to the denominator of $\tfrac{B_q}{4q}$
\cite{novikov}, \cite[\S2]{Zahler}, \cite[Theorem 5.3.11]{ravenel.green}. 
These groups describe first order phenomena in the Adams-Novikov spectral sequence related to the $J$-homomorphism \cite[Theorem 5.3.7]{ravenel.green}.

\begin{lemma}
\label{lem:mult-slices-sphere}
Multiplication with the zero slice of $\unit_\Lambda$ yields the canonical module map 
\[
\s_0(\unit_\Lambda) \smash \s_{q} (\unit_\Lambda) \to \s_{q} (\unit_\Lambda).
\]
Multiplication in the first slice $\s_{1}(\unit_{\Lambda}) \smash_{\s_{0}(\unit_{\Lambda})} \s_{1}(\unit_{\Lambda})\rightarrow \s_{2}(\unit_{\Lambda})$ takes the form
\begin{align*}
\Sigma^{2,2}\M \Lambda/2 \vee \Sigma^{3,2}\M \Lambda/2 
& 
\scriptsize{
\xrightarrow{
\begin{pmatrix} 
\id & 0 \\
0 & \inc 
\end{pmatrix} }
} 
\Sigma^{2,2}\M \Lambda/2 \vee \Sigma^{3,2}\M \Lambda/12 
\end{align*}
via the identification of $\MZ/2\smash_{\MZ} \MZ/2$ in~(\ref{eq:dual-steenrod-alg}).
The map $\inc\colon\Sigma^{3,2}\M \Lambda/2\to\Sigma^{3,2}\M \Lambda/12$ is induced by the canonical inclusion $\Lambda/2\to\Lambda/12$. 
More generally, 
the multiplication map $\s_1(\unit_\Lambda)\smash_{\s_{0}(\unit_{\Lambda})} \s_{2q-1}(\unit_\Lambda)\to \s_{2q}(\unit_\Lambda)$, 
restricted to the second summand of
\[ \Sigma^{1,1}\M\Lambda/2\{\alpha_1\}\smash_{\M\Lambda} \Sigma^{4q-3,2q-1}\M\Lambda/2\{\alpha_{2q-1}\} \cong \Sigma^{4q-2,2q}\M \Lambda/2\vee \Sigma^{4q-1,2q}\M \Lambda/2 \]
is the canonical inclusion 
$\Sigma^{4q-1,2q}\M \Lambda/2\to\Sigma^{4q-1,2q}\M \Lambda/a_{2q}\{\alpha_{2q/n}\}$
for every $q>0$.
\end{lemma}

\begin{proof}
The first multiplication is given by the canonical module structure since the unit map $\unit_{\Lambda} \to \M \Lambda$ induces an isomorphism on zero slices 
\cite[\S10]{levine.coniveau}, \cite[Theorem 6.6]{voevodsky.zero}.
With reference to Theorem~\ref{theorem:mult-adams-resol} we compute the map 
$\Tot\{1\}\otimes_{\Lambda}^{\mathbb{L}} \Tot\{1\}\to \Tot\{2\}$
of complexes of length two, which provides the matrix given above.
The final statement follows, because the map \eqref{eq:mu-mult} sends 
the generator $\tfrac{1}{2}d(x_1^{2q-1}\tensor x_1)$ of the degree $1$ homology of $\Tot\{2q-1\}\otimes_{\Lambda}^{\mathbb{L}} \Tot\{1\}$ 
to an element that is not in the image of $d\colon \Tot\{2q\}_0\to \Tot\{2q\}_1$. 

Alternatively,
the second map is determined by Lemmas \ref{lem:slices-unit-kq-1} and \ref{lem:slices-unit-kq-2} using that the multiplication
\[ 
\Sigma^{1,1}\M \Lambda/2\smash \Sigma^{1,1} \M \Lambda/2 \hookrightarrow
\s_1(\KQ_{\Lambda})\smash \s_1(\KQ_{\Lambda}) \to \s_2 (\KQ_{\Lambda}) \iso \Sigma^{4,2}  \M \Lambda \vee \Sigma^{2,2}  \M \Lambda/2
\vee \dotsm 
\]
hits the first two direct summands of $\s_2 (\KQ_{\Lambda})$ nontrivially \cite[Theorem 3.3]{ro.mult-slices-hermitian}.
(The proof carries over verbatim to compatible pairs by Theorem \ref{theorem:slices-kq}.)
Similarly, the multiplication on $\s_{\ast}(\KQ_{\Lambda})$ implies the last statement.
\end{proof} 

The multiplicative structure of the sphere spectrum induces a multiplication on the slice spectral sequence,
cf.~\cite[Theorem 3.6.16]{Pelaez} for the zeroth slice spectral sequence. 
We shall make use of the following Leibniz rule. 

\begin{proposition}
\label{prop:leibniz}
The slice spectral sequence for $\unit$ is multiplicative. 
The $r$th differential satisfies the Leibniz rule
\[ 
d^r(a\cdot b) 
= 
d^r(a)\cdot b + (-1)^{p} a\cdot d^r(b) 
\]
for $a\in E^{r}_{p,q,n}$ and $b\in E^{r}_{p^\prime,q^\prime,n^\prime}$.
\end{proposition}
\begin{proof}
We show a more general result for pairings $\D\smash \E\to \mathsf{F}$ of motivic spectra by checking the condition $\mu$ from \cite[\S3, 8]{massey}. 
To verify $\mu_0$ we use commutativity of the diagram
\[ \xymatrix{
    \s_p\D\smash \s_q\E \ar[r]^-{\mathrm{diag}} \ar[d]_-{\mathrm{mult}_{p,q}^{\s}} & 
    \bigl(\s_p\D\!\smash \!\s_q\E\bigr)\! \vee\! \bigl(\s_p\D\!\smash\! \s_q\E\bigr) \ar[r]&    
    \bigl(\Sigma^{1,0}\f_{p+1}\D\!\smash \!\s_q\E\bigr) \!\vee \!\bigl(\s_p\D\!\smash \!\Sigma^{1,0}\f_{q+1}\E\bigr)\ar[d]\\
    \s_{p+q}\D\smash \E \ar[d] & &
    \bigl(\Sigma^{1,0}\s_{p+1}\D\!\smash \!\s_q\E\bigr)\! \vee\! \bigl(\s_p\D\!\smash\! \Sigma^{1,0}\s_{q+1}\E\bigr)\ar[d]^-{\iso}\\
    \Sigma^{1,0}\f_{p+q+1} \D\smash \E \ar[d] & &
    \Sigma^{1,0}\bigl(\s_{p+1}\D\!\smash\! \s_q\E \vee \s_p\D\!\smash \!\s_{q+1}\E\bigr)\ar[d]^-{\mathrm{mult}_{p+1,q}^{\s}\vee \mathrm{mult}_{p,q+1}^{\s}}\\
    \Sigma^{1,0}\s_{p+q+1} \D\smash \E  & &
    \Sigma^{1,0}\bigl(\s_{p+1+q}\D\!\smash\!\E \vee\s_{p+q+1}\D\!\smash\!\E\bigr)\ar[ll]_-{\mathrm{fold}}} 
\]
in $\SH$, which follows from the commutative diagram implied by \cite[Theorems 3.4.5, 3.4.12]{Pelaez} (see \cite[\S5]{grso} for a strictification):
\[ 
\xymatrix{
\f_p\D\smash \f_q\E \ar[rr]^-{\mathrm{mult}_{p,q}^{\f}} \ar[d] & &\f_{p+q} \D\smash \E \ar[d]\\
\s_p\D\smash \s_q\E \ar[rr]^-{\mathrm{mult}_{p,q}^{\s}}  & &\s_{p+q} \D\smash \E} 
\]
More precisely, 
the latter implies the two compositions in the first diagram agree up to a homotopy on $\f_p\D\smash \s_q\E$ and on $\s_p\D\smash \f_q\E$. 
Moreover,
both compositions restrict to a constant map on $\f_p\D\smash \f_q\E$, 
so the respective homotopies can be chosen to be constant on $\f_p\D\smash \f_q\E$. 
The homotopy cofiber sequence
\[ 
\f_p\D\smash \s_q\E \cup_{\f_p\D\smash \f_q\E} \s_p\D\smash \f_q\E
\to \s_p\D\smash \s_q\E \to \Sigma^{1,0} \f_{p+1}\D\smash \Sigma^{1,0}
\f_{q+1} \E 
\]
shows that both compositions agree up to homotopy on $\s_p\D\smash \s_q\E$, 
since every map from the $p+q+2$-effective motivic spectrum $\Sigma^{1,0} \f_{p+1}\D\smash \Sigma^{1,0}\f_{q+1} \E$ to $\Sigma^{1,0}\s_{p+q+1}\D\smash \E$ is trivial.
  
Suppose $\mu_i$ holds for $0\leq i\leq m$.
To prove $\mu_{m+1}$ we choose commutative diagrams 
\[ 
\xymatrix{
& &  \Sigma^{1,0} \f_{q+m+2}\D \ar[d] &      & &  \Sigma^{1,0} \f_{q^\prime+m+2}\E \ar[d] \\
\Sigma^{p,n}X \ar[r]^-{a} \ar[rru]^-{x} &\s_q\D \ar[r] & 
\Sigma^{1,0} \f_{q+1}\D & \Sigma^{p^\prime,n^\prime}Y \ar[r]^-{b} \ar[rru]^-{y} &\s_{q^\prime}\E \ar[r] & \Sigma^{1,0} \f_{q^\prime +1}\E
}
\]
in $\SH$. 
Since $\mu_0,\dotsc,\mu_m$ are satisfied, 
the product $a\cdot b$ maps to an element in the group $[\Sigma^{p+p^\prime,n+n^\prime} X\smash Y,\Sigma^{1,0}\f_{q+q^\prime +1} \D\smash \E]$, 
which lifts to a map
\[ 
z\colon 
\Sigma^{p+p^\prime,n+n^\prime} X\smash Y 
\to 
\Sigma^{1,0}\f_{q+q^\prime+m+2}\D\smash \E.
\]
It remains to check the second equation in \cite[(3.1)]{massey}, 
which follows by induction and a similar consideration as for $\mu_0$.
\end{proof}

\subsection{The first motivic Hopf map and slices}
\label{sec:first-motivic-hopf}

For the Brown-Peterson spectrum $\BP$ at a prime $p$, 
the group $\Ext_{\BP_\ast\BP}^{s,t}(\BP_\ast,\BP_\ast)$ is trivial for $2s(p-1)>t$, 
cf.~Appendix \ref{sec:ext-groups-complex}.  
We use this to show that every iteration of the Hopf map $\eta$ acts trivially on odd order direct summands of $\s_{q}(\unit_{\Lambda})$.

\begin{lemma}
\label{lem:hopf-slices-global}
Let $A$ be a finite $\Lambda$-module of odd order.
If $\Sigma^{r,q}\MA$ is a direct summand of $\s_{q}(\unit_{\Lambda})$, 
there exists a natural number $k$ such that $\eta^k$ restricts to the trivial map on $\Sigma^{r,q}\MA$.
\end{lemma}
\begin{proof}
We may assume $A$ is cyclic of odd prime power order $p^{m}>1$.
For
\[
k>\frac{q-(2q-r-1)(p-1)}{p-2},
\]
the inequality $(2q-r+k-1)\cdot 2(p-1)>2q+2k$ implies the vanishing of the groups    
$$
\Ext^{2q-r+k-1,2q+2k}_{\BP_\ast\BP}(\BP_\ast,\BP_\ast)
=
\Ext^{2q-r+k,2q+2k}_{\BP_\ast\BP}(\BP_\ast,\BP_\ast) 
= 
0.
$$
These $\Ext$-groups give trivial contributions to $\s_{q+k}(\unit_{\Lambda})$,
cf.~Theorem \ref{theorem:sphereslices}.  
Note that $\eta^k$ restricts to an $\M \Lambda$-module map on the direct summand $\Sigma^{r,q}\MA$ of $\s_{q}(\unit_{\Lambda})$.
Since its target is trivial by the above, 
the assertion follows.
\end{proof}

\begin{definition}
The $\eta$-inverted $\Lambda$-local sphere spectrum $\unit_{\Lambda}[\eta^{-1}]$ is the homotopy colimit of the diagram
\[
\unit_{\Lambda}
\xrightarrow{\Sigma^{-1,-1}\eta}
\Sigma^{-1,-1}\unit_{\Lambda}
\xrightarrow{\Sigma^{-2,-2}\eta}
\Sigma^{-2,-2}\unit_{\Lambda}
\xrightarrow{\Sigma^{-3,-3}\eta}
\cdots
.
\]
\end{definition}

Since slices commute with homotopy colimits, 
see \cite[Corollary 4.5]{spitzweck.relations}, \cite[Lemma 4.2]{voevodsky.open}, 
Lemma \ref{lem:hopf-slices-global} implies:
\begin{corollary}
\label{cor:hopf-slices-global}
The zero slice of $\unit_{\Lambda}[\eta^{-1}]$ is an Eilenberg-MacLane spectrum associated with a connective chain complex of $\Lambda$-modules whose homology are $2$-primary torsion modules.
\end{corollary}

Next we analyze the unit map for hermitian $K$-theory.
\begin{lemma}
\label{lem:slices-unit-kq-1}
Let $u\colon \unit_{\Lambda}\to \KQ_{\Lambda}$ be the unit map.
Then $\s_{q}(u)$ is an inclusion on the bottom summand of $\s_{q}(\unit_{\Lambda})$; 
that is, 
$\M \Lambda\{1\}$ for $q=0$ and $\Sigma^{q,q}\M \Lambda/2\{\alpha_{1}^{q}\}$ for $q>0$. 
\end{lemma}
\begin{proof}
Note that $\s_0(\unit_{\Lambda})$ is a direct summand of $\s_0(\KQ_{\Lambda})$ by the proof of Theorem \ref{theorem:slices-kq} given for fields in~\cite[Theorem 4.18]{roendigs-oestvaer.hermitian}.
The Hopf map $\eta$ relates algebraic and hermitian $K$-theory via the motivic Wood cofiber sequence \cite[Theorem 3.4]{roendigs-oestvaer.hermitian}
\begin{equation}
\label{equation:motivicWood}
\Sigma^{1,1}\KQ_{\Lambda} 
\xrightarrow{\eta}
\KQ_{\Lambda} 
\to
\KGL_{\Lambda}.
\end{equation}
Theorems \ref{theorem:slices-kq} and \ref{theorem:slices-kgl} identify the slices of $\KQ_{\Lambda}$ and $\KGL_{\Lambda}$, 
respectively.
It follows that $\s_1(\unit_{\Lambda})$ is a direct summand of $\s_1(\KQ_{\Lambda})$ if and only if $\eta$ induces the canonical map 
\begin{equation}
\label{equation:s1directsummandmap}
\Sigma^{1,1} \s_0(\unit_{\Lambda}) 
\to  
\s_1(\unit_{\Lambda})
\end{equation}
obtained from the coefficient reduction map $\Lambda\to \Lambda/2$.
Assuming the map in (\ref{equation:s1directsummandmap}) is trivial leads to a contradiction since by \cite[Lemma 4.12]{roendigs-oestvaer.hermitian} the composition    
\[ 
\Sigma^{1,1}  \s_0(\unit_{\Lambda}) 
\to 
\Sigma^{1,1}  \s_0(\KQ_{\Lambda})
\xrightarrow{\eta} 
\s_1(\KQ_{\Lambda}) 
\]
is the canonical map $\Sigma^{1,1}\M \Lambda \to \Sigma^{1,1}\M \Lambda/2$ composed with the inclusion into the infinite sum
$\bigvee_{i<1} \Sigma^{2i+1,1} \M \Lambda/2$.
Proceeding by induction on $q>0$, 
$\eta$ induces the trivial map
\[ 
\Sigma^{q+1,q+1}\M \Lambda/2 
\hookrightarrow 
\Sigma^{1,1} \s_{q}(\unit_{\Lambda}) 
\to 
\s_{q+1}(\unit_{\Lambda}), 
\]
precisely if 
\[
\s_{q+1}(\mathrm{cone}(\eta)) 
\to 
\Sigma^{2,1}\s_{q} (\unit_{\Lambda}) 
\to
\Sigma^{2,1}\s_{q}(\KQ_{\Lambda}) 
\]
induces the inclusion of the direct summand $\Sigma^{q+2,q+1}\M \Lambda/2$.
(Here $\Sigma^{q+1,q+1}\M \Lambda/2$ can only map nontrivially to the corresponding direct summand of the same bidegree in $\s_{q+1}(\unit_{\Lambda})$.)
However, 
the latter coincides with the composite
\[
\s_{q+1}(\mathrm{cone}(\eta)) 
\to 
\s_{q+1}(\KGL_{\Lambda}) 
\to
\Sigma^{2,1}\s_{q}(\KQ_{\Lambda}),
\]
which is trivial on any direct summand of the form $\Sigma^{q+2,q+1}\M \Lambda/2$
by Lemma~\ref{lem:steenrod-alg-weight-zero}. 
\end{proof}

Next we determine the images under the unit map $u\colon \unit_{\Lambda}\to \KQ_{\Lambda}$ on the slice summands corresponding to the finite cyclic groups 
$\Ext_{\MU_\ast\MU}^{1,2q}(\MU_\ast,\MU_\ast)$ of order $a_{q}$.

\begin{lemma}
\label{lem:slices-unit-kq-2}
For $q\geq 1$, 
the map $\s_{2q-1} (u)\colon \s_{2q-1}(\unit_{\Lambda}) \to \s_{2q-1}(\KQ_{\Lambda})$ restricts to the canonical inclusion of the direct summand $\Sigma^{4q-3,2q-1}\M \Lambda/2$ 
generated by $\alpha_{2q-1}$. 
Moreover, 
$\s_{2q}(u)\colon \s_{2q}(\unit_{\Lambda}) \to \s_{2q}(\KQ_{\Lambda})$ maps the direct summand $\Sigma^{4q-1,2q}\M \Lambda/a_{2q}$ generated by $\overline{\alpha}_{2q}$ to 
$\Sigma^{4q,2q}\M \Lambda$ such that 
the composite
\[ 
\Sigma^{4q-1,2q}\M \Lambda/a_{2q} 
\to 
\Sigma^{4q,2q}\M \Lambda 
\xrightarrow{\pr} 
\Sigma^{4q,2q}\M \Lambda/2 
\]
is the unique nontrivial map $\partial^{a_{2q}}_{2}$ from $\Sigma^{4q-1,2q}\M \Lambda/a_{2q}$ to $\Sigma^{4q,2q}\M \Lambda/2$, 
cf.~Lemma \ref{lem:steenrod-alg-weight-zero}.
\end{lemma}
\begin{proof}
We argue by induction on $q$.
Lemma~\ref{lem:slices-unit-kq-1} shows the claim for $q=1$.
Suppose the direct summand $\Sigma^{4q-3,2q-1}\M \Lambda/2$ of $\s_{2q-1}(\unit_{\Lambda})$ corresponding to $\Ext_{\MU_\ast\MU}^{1,4q-2}(\MU_\ast,\MU_\ast)$ maps identically 
via $\s_{2q-1}(u)$ to the top summand of $\s_{2q-1}(\KQ_{\Lambda})$, 
i.e., 
the direct summand of largest simplicial suspension degree.
By \cite[Theorem 5.5]{roendigs-oestvaer.hermitian} the first slice differential $\dd^{\KQ_{\Lambda}}_{1}$ maps this direct summand via the cohomology operation $\delta_{2}\Sq^2\Sq^1$ 
to the top summand $\Sigma^{4q+1,2q}\M \Lambda$ of $\Sigma^{1,0}\s_{2q}(\KQ_{\Lambda})$;
here, 
$\delta_{n}\colon\M \Lambda/n\to\Sigma^{1,0}\M \Lambda$ is the canonical connecting map.
The commutative diagram with horizontal slice differentials
\[ 
\xymatrix{
\s_{2q-1}(\unit_{\Lambda}) \ar[r]^{\dd^{\unit_{\Lambda}}_{1}} \ar[d]_{\s_{2q-1}(u)} 
& \Sigma^{1,0}  \s_{2q}(\unit_{\Lambda})\ar[d]^{\Sigma^{1,0}  \s_{2q}(u)} \\
\s_{2q-1}(\KQ_{\Lambda}) \ar[r]^{\dd^{\KQ_{\Lambda}}_{1}} 
& \Sigma^{1,0}  \s_{2q}(\KQ_{\Lambda})
}
\]
implies the map $\s_{2q}(u)$ is nontrivial --- since $\dd^{\KQ_{\Lambda}}_{1}$ is nontrivial --- by Lemma~\ref{lem:slices-unit-kq-1}. 
More precisely, 
Lemma~\ref{lem:steenrod-alg-weight-zero} and Theorem~\ref{theorem:steenrod-algebra} supply an element $x_{2q}\in \Lambda/a_{2q}$ such that $\dd^{\unit_{\Lambda}}_{1}$ maps the 
direct summand $\Sigma^{4q-3,2q -1}\M \Lambda/2$ via $\inc^2_{a_{2q}}\Sq^2\Sq^1+x_{2q}\partial^2_{a_{2q}}\Sq^2$ to the direct summand $\Sigma^{4q,2q}\M \Lambda/a_{2q}$ of 
$\Sigma^{1,0}\s_{2q}(\unit_{\Lambda})$ obtained from $\Ext_{\MU_\ast\MU}^{1,4q}(\MU_\ast,\MU_\ast)$. 
Here $\inc^2_{a_{2q}}$ is induced by the canonical map $\Lambda/2\to \Lambda/a_{2q}$ and $\partial^2_{a_{2q}}\colon\Sigma^{4q-1,2q}\M \Lambda/2\to \Sigma^{4q,2q}\M \Lambda/a_{2q}$ is defined in 
Example~\ref{ex:steenrod-alg-weight-zero}.
Moreover, 
Lemma~\ref{lem:steenrod-alg-weight-zero} implies there exists an odd integer $y_{2q}$ such that $\s_{2q}(u)$ maps the direct summand $\Sigma^{4q-1,2q}\M \Lambda/a_{2q}$ of 
$\s_{2q}(\unit_{\Lambda})$ via $y_{2q}\delta_{a_{2q}}$ to the direct summand $\Sigma^{4q,2q}\M \Lambda$ of $\s_{2q}(\KQ_{\Lambda})$.
Composing with the canonical map to $\Sigma^{4q,2q} \M \Lambda/2$ yields the unique nontrivial map $\partial^{a_{2q}}_{2}$ in the statement of the lemma.
 
The Hopf map $\eta$ and the unit $u:\unit_{\Lambda}\to\KQ_{\Lambda}$ induces the commutative diagram:
\[ 
\xymatrix{
\Sigma^{1,1}  \s_{2q}(\unit_{\Lambda}) \ar[r]^{\eta} \ar[d]_{\Sigma^{1,1}  \s_{2q}(u)} 
& \s_{2q+1}(\unit_{\Lambda})\ar[d]^{\s_{2q+1 }(u)} \\
\Sigma^{1,1}  \s_{2q}(\KQ_{\Lambda}) \ar[r]^{\eta}  
& \s_{2q+1}(\KQ_{\Lambda})\\
}
\]
The direct summand $\Sigma^{1,1}  \Sigma^{4q-1,2q}  \M \Lambda/a_{2q}$ maps via the unique nontrivial map $\partial^{a_{2q}}_{2}$ to the top summand $\Sigma^{4q +1,2q+1}\M \Lambda/2$ of 
$\s_{2q +1}(\KQ_{\Lambda})$.
Thus $\s_{2q +1}(u)$ maps the direct summand $\Sigma^{4q +1,2q+1}  \M \Lambda/2$ of $\s_{2q+1}(\unit_{\Lambda})$ generated by 
$\alpha_{2q+1}\in\Ext_{\BP_\ast\BP}^{1,4q+2}(\BP_\ast,\BP_\ast)$ nontrivially, 
and hence identically, 
to the top summand in $\s_{2q +1}(\KQ_{\Lambda})$.
This completes the induction step.
\end{proof}

\begin{lemma}
\label{lem:slices-unit-hopf}
The first motivic Hopf map $\eta$ induces the projection map 
\[ 
\s_{1}(\eta)\colon
\Sigma^{1,1}\s_0(\unit_{\Lambda})
\iso 
\Sigma^{1,1}\M\Lambda\{1\}
\to 
\Sigma^{1,1}\M\Lambda/2\{\alpha_{1}\}\iso\s_1(\unit_{\Lambda}). 
\]
Moreover, for $q\geq 1$,
$\s_{2q+1}(\eta)$ restricts to the unique nontrivial map 
\[
\Sigma^{1,1}\partial^{a_{2q}}_{2}
\colon
\Sigma^{1,1}\Sigma^{4q-1,2q}\M\Lambda/a_{2q}\{\alpha_{2q/n}\} 
\to
\Sigma^{4q+1,2q+1}\M\Lambda/2\{\alpha_{2q+1}\}, 
\]
while for $q\geq 0$, $i\geq 1$, $\s_{2q+i+1}(\eta)$ restricts to the identity map 
\[
\Sigma^{1,1}\Sigma^{4q+i,2q+i}\M\Lambda/2\{\alpha_{1}^{i-1}\alpha_{2q+1}\} 
\xrightarrow{\id}  
\Sigma^{4q+i+1,2q+i+1}\M\Lambda/2\{\alpha_{1}^{i}\alpha_{2q+1}\}. 
\]
\end{lemma}  
\begin{proof}
In $\s_{\ast}(\KQ_{\Lambda})$, 
$\eta$ induces the identity on every direct summand $\Sigma^{2i+q,q}\M \Lambda/2$ and the canonical projection on $\Sigma^{4q,2q}\M \Lambda$ \cite[\S4.3]{roendigs-oestvaer.hermitian}. 
The first motivic Hopf map $\eta$ corresponds to derived multiplication with $\alpha_1\in\Ext_{\MU_\ast\MU}^{1,2}(\MU_\ast,\MU_\ast)$ via the commutative diagram:
\[ 
\xymatrix{ 
\Sigma^{1,1}\unit_{\Lambda}\smash \s_{q}(\unit_{\Lambda}) 
\ar[r]^-\cong \ar[d] & 
\s_{q+1}(\Sigma^{1,1} \unit_{\Lambda}) \ar[d]^-{\s_{q+1}(\eta)} \\
\s_{1}(\unit_{\Lambda})\smash_{\s_0(\unit_{\Lambda})} \s_{q}(\unit_{\Lambda})
\ar[r] & 
\s_{q+1}(\unit_{\Lambda})}
\]
Applying Lemmas~\ref{lem:mult-slices-sphere}, \ref{lem:slices-unit-kq-1}, and \ref{lem:slices-unit-kq-2} finishes the proof.
\end{proof}

We proceed by analyzing the action of $\eta$ on the slices of $\unit_{\Lambda}$.
Lemma~\ref{lem:slices-unit-hopf} implies that $\s_0(\unit_{\Lambda}[\eta^{-1}])$ contains $\M \Lambda/2\vee\Sigma^{2,0}\M \Lambda/2\vee\Sigma^{4,0}\M \Lambda/2\vee\Sigma^{6,0}\M \Lambda/2$ as a direct summand.
For $q\neq 0$ we have $\s_{q}(\unit_{\Lambda}[\eta^{-1}]) \iso \Sigma^{q,q} \s_0(\unit_{\Lambda}[\eta^{-1}])$ as observed in \cite[Example 2.3]{roendigs-oestvaer.hermitian}.
By expanding on the proof of Lemma~\ref{lem:slices-unit-hopf} it follows that $\Sigma^{q+n,q}\M \Lambda/2$ is a retract of $\s_{q}(\unit_{\Lambda}[\eta^{-1}])$ for every even natural number $n$.
A complete computation can be extracted from the following recent result due to Andrews-Miller \cite[(1.1.1)]{andrews-miller},
cf.~\cite[p.~500]{Zahler}. 

\begin{theorem}
\label{theorem:zahlers-hope}
At the prime $2$ and for $t\neq 1$ there exists a natural number $N(t)$ such that $\alpha_{1}^{s-1}\overline{\alpha}_{t+1}$ is the unique nontrivial element in 
$\Ext^{s,2(s+t)}_{\BP_\ast\BP}(\BP_\ast,\BP_\ast)$ for all $s\geq N(t)$.
\end{theorem} 

Recall that slices commute with homotopy colimits \cite[Corollary 4.5]{spitzweck.relations}, \cite[Lemma 4.2]{voevodsky.open}.
Thus Theorem \ref{theorem:mult-adams-resol}, 
Corollary \ref{cor:hopf-slices-global}, 
and Theorem \ref{theorem:zahlers-hope} allow us to compute the slices of the $\eta$-inverted $\Lambda$-local sphere spectrum.
(Note the multiplicative identity $\alpha_{1}\alpha_{2/2}=0$ explains the condition $n\neq 1$ in the wedge product decomposition below,
cf.~Appendix \ref{sec:ext-groups-complex}.)
\begin{theorem}
\label{theorem:slices-eta-inv-sphere}
The slices of the $\eta$-inverted $\Lambda$-local sphere spectrum are given by 
\[ 
\s_{q}(\unit_{\Lambda}[\eta^{-1}])
\cong
\Sigma^{q,q}\M \Lambda/2\{\alpha_{1}^{q}\}
\vee
\bigvee_{n\geq 2} \Sigma^{n+q,q}\M \Lambda/2\{\alpha_{1}^{q-n-1}\overline{\alpha}_{n+1}\}.
\]
\end{theorem}

\begin{remark}  
The canonical map $\M \Lambda\cong \s_0(\unit_{\Lambda}) \to \s_0(\unit_{\Lambda}[\eta^{-1}])$ does not induce a map of ring spectra $\M \Lambda/2\to \s_0(\unit_{\Lambda}[\eta^{-1}])$. 
This follows since the multiplicative structure on $\s_0(\unit_{\Lambda}[\eta^{-1}])$ involves at least as many relations as for $\s_0(\KT)$, 
see \cite[Theorem 3.6]{ro.mult-slices-hermitian}.
\end{remark}

For completeness we give slice spectral sequence interpretations of the zeroth, second, and third motivic Hopf maps.
\begin{example}
\label{ex:Hopf-maps}
The zeroth motivic Hopf map $1-\epsilon\in\pi_{0,0}\unit_{\Lambda}$ is detected by $2\in\pi_{0,0}\M \Lambda$.
Further, 
$\nu\in \pi_{3,2}\unit_{\Lambda}$ is detected by $1\in\pi_{3,2}\Sigma^{3,2}\M \Lambda/12\cong h^{0,0}_{12}$ for the direct summand $\Sigma^{3,2}\M \Lambda/12$ of $\s_2(\unit_{\Lambda})$ 
corresponding to the Adams-Novikov spectral sequence representative $\alpha_{2/2}$ of $\nu_{\Top}$, 
cf.~Corollary \ref{corollary:small-slices}.
Finally, 
$\sigma\in\pi_{7,4}\unit_{\Lambda}$ is detected by $1\in\pi_{7,4}\Sigma^{7,4}\M \Lambda/240\cong h^{0,0}_{240}$ for the direct summand $\Sigma^{7,4}\M \Lambda/240$
of $\s_4(\unit_{\Lambda})$ corresponding to the Adams-Novikov spectral sequence representative $\alpha_{4/4}$ of $\sigma_{\Top}$,
cf.~Corollary \ref{corollary:small-slices}.
Theorem \ref{theorem:sphereslices} and Lemma \ref{lem:first-diff-unit-1} imply that all differentials entering or exiting these groups are trivial.
\end{example}

\section{Convergence of the slice spectral sequence}
\label{sec:convergence}

In this section we answer Voevodsky's convergence problem for the slice spectral sequence \cite[\S 7]{voevodsky.open}.
This clarifies the role of the first motivic Hopf map $\eta$ in the stable theory, 
and recovers Levine's convergence result for the slice filtration over fields of finite cohomological dimension in \cite[Theorem 4]{levine.sliceconvergence}.

\subsection{Preliminaries}
\label{sec:preliminaries}
For every integer $q$ the inclusion $i_{q}\colon\Sigma^{2q,q}\SH^{\eff}\subset\SH$ affords a right adjoint $r_{q}$ by work of Neeman \cite{neeman}.
Define $\f_{q}$ as the composite $i_{q}\circ r_{q}$.
The derived counit for the adjunction yields for every motivic spectrum $\E$ a distinguished triangle
\[ 
\f_{q}(\E)
\to 
\E 
\to 
\f^{q-1}(\E) 
\to 
\Sigma^{1,0} \f_{q}(\E). 
\]
In the special case where $\f_{q-1}(\E)=\E$, one has $\f^{q-1}(\E)=\s_{q-1}(\E)$. 
The natural transformation $\f_{q+1}\to \f_{q}$ induces a natural transformation $\f^{q}\to \f^{q-1}$. 
By \cite{grso} and \cite{Pelaez} the slice filtration can be modeled on the level of model categories; 
whence the following is well-defined.

\begin{definition}\label{def:slicecompletion}
Let $\slicecomp$ denote the slice completion endofunctor $\holim_{q} \f^{q-1}$ of $\SH$.
\end{definition}

By construction, 
the slice spectral sequence for $\E$ \cite[\S2]{roendigs-oestvaer.hermitian}, \cite[\S 7]{voevodsky.open} is an upper half-plane spectral sequence with entering 
differentials which converges conditionally to the homotopy groups of $\slicecomp(\E)$ in the sense of Boardman \cite[Definition 5.10, \S7]{Boardman}.
\begin{remark}
\label{remark:scaugmentedtrivial}
The slice completion functor is augmented by the natural transformation $\Id_{\SH}\to\slicecomp$.
If each slice of $\E$ is trivial then so is $\slicecomp(\E)$.
\end{remark}
\begin{example}
\label{example:convergence}
The canonical map $\KT\to \slicecomp(\KT)$ is not a weak equivalence over the real numbers $F=\mathbb{R}$.
The $\lim^1$-short exact sequence and \cite[Lemma 3.13, Corollary 3.16]{roendigs-oestvaer.hermitian} show
\[ 
\pi_{s,t}\slicecomp(\KT) 
\iso 
\begin{cases} \mathbb{Z}^\wedge_2 & 4\mid s-t \\
\mathbb{Z}^\wedge_2/\mathbb{Z} & 4\mid s-1-t \\
0 & \mathrm{otherwise.}
\end{cases}
\]
On the other hand, 
$\pi_{s,t}\KT$ is trivial if $4\nmid s-t$ and infinite cyclic in all other degrees.
\end{example}  

\begin{remark}
In \cite[Theorem 6.12]{roendigs-oestvaer.hermitian} it is shown that the slice spectral sequence for $\KT$ converges to the filtration of the Witt ring 
by its fundamental ideal of even dimensional forms.
Example \ref{example:convergence} shows the slice spectral sequence for $\KT$ does not converge conditionally to the homotopy groups of $\KT$.
If $p$ is an odd prime, 
$\s_{\ast}(\KT/p)=\ast$ and $\f_{q}(\KT/p)=\KT/p$ for all $q\in\Z$.
\end{remark}

Throughout the following, 
we let $(S,\Lambda)$ be a compatible pair, 
cf.~Definition \ref{def:compatible}.

\begin{lemma}
\label{lem:slice-eta-inv}
If $\eta\colon\Sigma^{1,1}\E\to \E$ is a weak equivalence, 
then the multiplication by $2$ map is trivial on all slices of $\E$.
\end{lemma}
\begin{proof}
Note that $\E$ is an $\unit[\eta^{-1}]$-module.
By \cite[\S6 (iv),(v)]{grso} and \cite[Theorem 3.6.13(6)]{Pelaez} it follows that $\s_0(\E)$ is an $\s_0(\unit[\eta^{-1}])$-module.
Multiplication by $2$ is trivial on $\s_0(\unit[\eta^{-1}])$ by Corollary \ref{cor:hopf-slices-global}, 
cf.~Theorem \ref{theorem:zahlers-hope}, 
so the same holds for any $\s_0(\unit[\eta^{-1}])$-module,
e.g., 
$\s_{q}(\E)$.
\end{proof}

Let $\epsilon\colon \unit\to \unit$ be the endomorphism induced by the commutativity automorphism on $S^{1,1}\smash S^{1,1}$, 
or equivalently by the inverse map for the multiplicative group scheme $S^{1,1}=\G=\mathbf{A}^1\minus\{0\}$.
The motivic Hopf map $\eta$ is induced by the canonical map $\mathbf{A}^{2}\minus \{0\}\to \PP^{1}$, 
or equivalently, 
see \cite[p.~73 in \S3.3 and Example 7.26]{morel.field},
by the Hopf construction applied to the multiplicative group $\G$.
The first Hopf relation $\eta\epsilon=\eta$ follows, 
cf.~\cite[Theorem 1.4, Lemma 4.8]{dugger-isaksen.hopf}.
Moreover, 
the rational point $[-1]\colon S^{0,0}\to S^{1,1}$ satisfies the relation $1+\epsilon=-\eta[-1]$ from \cite[p.~51]{morel.field}.

In $\SH[\frac{1}{2}]$ there are orthogonal idempotents $\epsilon_{\pm}=(1\mp\epsilon)/2$.
We obtain a decomposition of triangulated categories $\SH[\frac{1}{2}]=\SH[\frac{1}{2}]^{+}\times\SH[\frac{1}{2}]^{-}$, 
where the factors correspond to inverting $\epsilon_{+}$ and $\epsilon_{-}$, respectively.
The element $\epsilon$ acts as $-1$ on the plus part and as the identity on the minus part.
It follows that $\eta$ acts trivially on $\SH[\frac{1}{2}]^{+}$ while $\eta$ acts invertibly on $\SH[\frac{1}{2}]^{-}$.

\begin{lemma}
\label{lem:slices-2-inv}
Suppose $2\colon \E\to \E$ is a weak equivalence. 
Then the canonical map
\[
\E
\to 
\E[\tfrac{1}{1-\epsilon}]
\] 
induces a weak equivalence
\[ 
\s_{q}(\E)
\xrightarrow{\sim} 
\s_{q}(\E[\tfrac{1}{1-\epsilon}]) 
\]
for all $q\in \ZZ$. 
Hence there is a canonical weak equivalence between slice completions
\[ 
\slicecomp(\E) 
\xrightarrow{\sim} 
\slicecomp (\E[\tfrac{1}{1-\epsilon}]). 
\]
\end{lemma}
\begin{proof}
By assumption $\E$ affords a splitting 
\[ 
\E 
\simeq 
\E[\tfrac{1}{1-\epsilon}] \times \E[\tfrac{1}{1+\epsilon}]. 
\]
Since $\eta\epsilon=\eta$, 
the map $\eta$ is trivial on $\E[\tfrac{1}{1-\epsilon}]$. 
The equality $1+\epsilon=-\eta[-1]$ implies that $\eta$ is invertible on $\E[\tfrac{1}{1+\epsilon}]$. 
Lemma~\ref{lem:slice-eta-inv} shows the map $2\colon \s_{q}(\E[\tfrac{1}{1+\epsilon}])\to \s_{q}(\E[\tfrac{1}{1+\epsilon}])$ is trivial.
By the assumption on $\E$ this is also a weak equivalence. 
Hence $\s_{q}(\E[\tfrac{1}{1+\epsilon}])\simeq \ast$ for all $q\in \ZZ$. 
Applying the slice completion functor yields 
\[ 
\slicecomp (\E) 
\simeq
\slicecomp \bigl(\E[\tfrac{1}{1-\epsilon}]\bigr) \times \slicecomp \bigl(\E[\tfrac{1}{1+\epsilon}]\bigr) 
\simeq 
\slicecomp \bigl(\E[\tfrac{1}{1-\epsilon}]\bigr),
\]
see Remark \ref{remark:scaugmentedtrivial}.
\end{proof}

\begin{corollary}
\label{corollary:rationaliso}
There is an isomorphism $\slicecomp (\unit_{\QQ}^+)\cong \slicecomp(\unit_{\QQ})\cong \MQ$ in $\SH_{\QQ}$.
\end{corollary}

\begin{proof}
The first identification is a special case of Lemma~\ref{lem:slices-2-inv}. The second follows from the identification $\s_\ast(\unit_{\QQ})=\MQ$ in Theorem~\ref{theorem:sphereslices}.
\end{proof}

It follows that $\unit_{\QQ}$ is an effective cellular motivic spectrum which need not coincide with its slice completion, e.g.,~over ordered fields \cite[Proposition 32.23]{ekm}. See also Corollary~\ref{corollary:1QMQ}. 

Following \cite[\S3]{roendigs-oestvaer.rigidity} every element $\alpha\in \pi_{s,t}\unit$ defines an $\alpha$-completion functor sending a motivic spectrum $\E$ to the homotopy limit of the canonically induced sequential diagram 
\[ 
\dotsm 
\to 
\E/\alpha^{m+1} 
\to 
\E/\alpha^{m}
\to 
\dotsm 
\to 
\E/\alpha. 
\]
There is a naturally induced augmentation map 
\[ 
\E
\to 
\E^{\wedge}_\alpha:=\holim_m \E/\alpha^m.
\]

\begin{definition}
A motivic spectrum $\E$ is $\alpha$-{\em complete\/} if the canonical map  $\E\to \E^\wedge_\alpha$ is a weak equivalence.
\end{definition}

Inverting $\alpha$ on $\E$ amounts to forming the homotopy colimit $\E[\alpha^{-1}]=\E[\tfrac{1}{\alpha}]$ of the diagram
\[ 
\E 
\xrightarrow{\alpha\smash \E} 
\Sigma^{-s,-t} \E
\xrightarrow{\alpha\smash \Sigma^{-2s,-2t} \E} 
\dotsm.
\]
The main examples of interest to us are integers and $\eta\in \pi_{1,1}\unit$. 
Completion and inversion with respect to a homotopy class are related by an arithmetic square.

\begin{lemma}
\label{lem:arithmetic-square}
For every motivic spectrum $\E$ and $\alpha\in \pi_{s,t}\unit$ there is a homotopy pullback
\[ 
\xymatrix{ 
\E \ar[r] \ar[d] & \E[\tfrac{1}{\alpha}] \ar[d]\\
\E^\wedge_\alpha \ar[r]  & \E^\wedge_\alpha[\tfrac{1}{\alpha}].  
}
\]
\end{lemma}
\begin{proof}
It suffices to show $\alpha$ is invertible in the homotopy fiber of the canonical map $\E\to \E^\wedge_\alpha$,
i.e., 
in the homotopy limit of the diagram
\[ 
\dotsm 
\xrightarrow{\alpha\smash \Sigma^{2s,2t} \E} \Sigma^{2s,2t} \E 
\xrightarrow{\alpha\smash \Sigma^{s,t} \E}  \Sigma^{s,t}\E
\xrightarrow{\alpha\smash \E} \E. 
\]
This follows since multiplication by $\alpha$ is an isomorphism on the corresponding homotopy colimit in the opposite category of $\SH$.
\end{proof}

\begin{lemma}
\label{lemma:slicesinverting2}
The map $\E\to \E^\wedge_\eta$ induces a weak equivalence on slices after inverting $2$.
\end{lemma}
\begin{proof}
The arithmetic square for $\eta$
\[ 
\xymatrix{ 
\E \ar[r] \ar[d] & \E[\tfrac{1}{\eta}] \ar[d]\\
\E^\wedge_\eta \ar[r]  &\E^\wedge_\eta[\tfrac{1}{\eta}]
}
\]
is a homotopy pullback square by Lemma~\ref{lem:arithmetic-square}. 
Inverting $2$ yields a homotopy pullback square in which the motivic spectra on the right hand side have contractible slices,
see Lemma~\ref{lem:slice-eta-inv}.
Thus the map on the left hand side induces a weak equivalence on all slices,
as desired. 
\end{proof}

\begin{lemma}
\label{lemma2}
Let $\alpha\in \pi_{s,t}\unit_{\Lambda}$ and suppose $(\alpha\smash \E)^{N}$ is the trivial map for some $N\geq 1$.
Then $\E$ is $\alpha$-complete.
\end{lemma}

\begin{lemma}
\label{lemma1}
The slice completion $\slicecomp(\E)$ of an effective motivic spectrum $\E$ is $\eta$-complete.
\end{lemma}
\begin{proof}
By the assumption on $\E$ it follows that $\f^{-1}(\E)=\s_{0}(\E)$.
The distinguished triangle $\s_{q}(\E)\to\f^{q}(\E)\to\f^{q-1}(\E)\to\Sigma^{1,0}\s_{q}(\E)$ shows the effective co-covers are finite extensions with slices as associated graded pieces. 
Since $\eta$ acts trivially on slices, 
these are $\eta$-complete by Lemma \ref{lemma2}.
The result follows by commuting homotopy limits.
\end{proof}

\begin{lemma}
\label{lemma3}
Suppose $\E\in\SH^{\eff}$ and $\E/\eta$ is slice complete. 
Then the $\eta$-completion of $\E$ maps via a natural weak equivalence to the slice completion of $\E$.
\end{lemma}
\begin{proof}
There is a naturally induced commutative diagram:
\begin{equation}
\label{equation:jhbdcjhs}
\xymatrix{
\E
\ar[r] \ar[d] &
\E^{\wedge}_{\eta}
\ar[d] \\
\holim_{q}\f^{q-1}(\E)
\ar[r] &
(\holim_{q}\f^{q-1}(\E))^{\wedge}_{\eta} 
}
\end{equation}
Lemma \ref{lemma1} shows the lower horizontal map is an isomorphism in $\SH$.
We shall prove the right vertical map is an isomorphism.
It suffices to show there is an isomorphism for all $n\geq 1$
\[ 
\E/\eta^{n}
\to 
\holim_{q} \f^{q-1}(\E)/\eta^{n}.
\]
This follows from our assumption on $\E/\eta$, induction on $n$, and the canonical identification
\[ 
\holim_{q} \f^{q-1}(\E)/\eta^{n}
\cong
\holim_{q} \f^{q-1}(\E/\eta^{n})
\]
of iterated homotopy limits.
\end{proof}

\begin{remark}
\label{ex:kt-sc-eta}
Lemmas~\ref{lemma1} and~\ref{lemma3} hold for $\E$ provided $\f_q(\E)=\E$ for some $q\in \ZZ$.
The said lemmas do not hold for $\KT$, since $\KT^\wedge_\eta$ is contractible, while $\slicecomp(\KT)$ is not.
\end{remark}

\begin{example}
\label{ex:kq-sc-eta}
Hermitian $K$-theory has the property that $\KQ/\eta\cong \KGL$ is slice complete by the motivic Wood cofiber sequence (\ref{equation:motivicWood})
and \cite{Voevodsky:motivicss}, 
and likewise for its effective cover.
In particular, 
we have $\slicecomp(\f_0(\KQ))\cong \f_0(\KQ)^\wedge_\eta$.
\end{example}

\subsection{Connectivity}
\label{sec:connectivity}
To study slice completeness of $\unit_{\Lambda}$ we employ a notion of connectivity based on presheaves of homotopy groups.

\begin{definition}
\label{def:connectivity}
A motivic spectrum $\E$ is $k$-{\em connected\/} if for every triple $(s,t,d)$ of integers with $s-t+d<k$ and every $X\in\Sm_{S}$ of dimension $\leq d$, 
the group $[ \Sigma^{s,t} X_+,\E]$ is trivial. 
A map of motivic spectra is $k$-{\em connected\/} if its homotopy fiber is $k$-connected.
\end{definition}

A related notion of connectivity is defined in \cite[Section~1.1]{hoyois}:
$\E$ is said to be $k$-connective if it is contained in the full localizing subcategory generated by the shifted suspension spectra $\{\Sigma^{s,t}X_+\}_{s-t\geq k,X\in \Sm_S}$.
If $S$ is a field, 
$\E$ is $k$-connective if and only if its Nisnevich sheaves of homotopy groups $\mathbf{\pi}_{s,t} \E=0$ for $s-t<k$,
see~\cite[Theorem 6.1.8]{morel.connectivity} for the case of perfect fields
and \cite[Theorem 2.3]{hoyois} for all fields.

\begin{lemma}
\label{lem:connectivities}
Suppose $S$ is a field. 
If $\E$ is $k$-connective, 
then it is $k$-connected.
\end{lemma}
\begin{proof}
If $\E$ is a $k$-connective motivic spectrum and $t\in \ZZ$, the associated
$S^1$-spectrum $\Omega_{\G}^\infty \Sigma^{0,-t}\E$ is $k+t$-connective, 
and by adjunction determines the values of the presheaf of homotopy groups
of $\E$ with weight $t$. For any motivic $S^1$-spectrum $E$ which is $k+t$-connective with respect to the homotopy $t$-structure, 
we have $[\Sigma^s X_+,E]=0$ for every $X\in\Sm_{S}$ of dimension $\leq d$ and $s<-d+k+t$: 
By induction the analogous statement holds for any finite Postnikov truncation of $E$, 
so the claim follows from a limit argument (the $\lim^1$-term vanishes because the corresponding system is eventually constant).
See also the proof of \cite[Corollary 2.4]{hoyois}.
\end{proof}

\begin{example}
\label{example:mz-connectivity}
If $S$ is a field or a Dedekind domain,
the motivic cohomology spectrum $\MA$ over $S$ is $0$-connected.
In fact, 
$[\Sigma^{s,t}X_+,\MA]\iso H^{-s,-t}(X;A)=0$ if $-s>-t+d$, 
where $d$ is the dimension of $X$, 
by \cite[Corollary 4.4]{geisser.dedekind}. 
\end{example}

\begin{definition}
\label{def:connected-pair}
A {\em connected pair\/} is a compatible pair $(S,\Lambda)$ such that the $\Lambda$-local sphere spectrum $\unit_{\Lambda}$ and the motivic Eilenberg-MacLane spectrum 
$\M\Lambda$ are $0$-connected.
\end{definition}  

Unless the contrary is stated we assume throughout that $(S,\Lambda)$ is a connected pair.

\begin{example}
\label{example:fields}
If $S$ is a field, 
every compatible pair is a connected pair by Lemma~\ref{lem:connectivities} and Example~\ref{example:mz-connectivity}.
\end{example}

\begin{example}
\label{example:connectivity-thom}
If $V\to Y$ be a vector bundle of rank $r$ over $Y\in\Sm_{S}$, 
then the Thom space $\Thom(V)$ is $r$-connected. 
More generally, 
if $a\in K^0(Y)$ is a virtual vector bundle of rank $r\in \ZZ$, 
then $\Thom(a)$ is $r$-connected. 
This follows from \cite[Lemma 3.1]{hoyois} and the assumption on $(S,\Lambda)$.
In particular, 
$\Sigma^{a,b}\unit_{\Lambda}$ is $(a-b)$-connected. 
\end{example}

\begin{lemma}
\label{lem:cellular-connectivity}
A cellular motivic spectrum $\E$ is $k$-connected if and only if $\pi_{s,t}\E = 0$ for all integers $s$, $t$ with $s-t<k$.
\end{lemma}
\begin{proof}
Suppose $\E$ is a cellular motivic spectrum satisfying the stated vanishing.
Inductively one constructs a map of cellular motivic spectra $\D\to\E$ inducing a $\pi_{p,q}$-isomorphism for all $p,q\in\ZZ$, 
where $\D$ is obtained by attaching $(s,t)$-cells for $s-t\geq k$. 
Since this is a weak equivalence \cite[Corollary 7.2]{dugger-isaksen.cell} the desired conclusion follows using connectivity of spheres as in Example \ref{example:connectivity-thom}.
The converse implication is immediate.
\end{proof}

The connectivity of $\f_{q}(\unit_{\Lambda})$ does not increase with $q$, 
mainly because of the direct summand $\Sigma^{q,q}\M\Lambda/2\{\alpha^{q}\}$ of $\s_{q}(\unit_{\Lambda})$, 
cf.~Corollary \ref{corollary:small-slices}.
By coning off $\eta$ the situation changes drastically. 
To make this precise, 
we compare with the algebraic cobordism spectrum $\MGL$ \cite[\S6.3]{VoevodskyICM1998}.

\begin{lemma}
\label{lem:slice-convergence-mgl}
The effective cover $\f_{q}(\MGL_\Lambda)$ is $q$-connected.
\end{lemma}
\begin{proof}
This follows from the proof of~\cite[Theorem 4.7]{spitzweck.relations} and Lemma~\ref{lem:connectivities}.
\end{proof}

\begin{lemma}
\label{lem:cone-eta-mgl}
The unit map of $\MGL$ factors through a map $\unit/\eta\to \MGL$ whose cone is contained in $\Sigma^{4,2}\SH^{\eff}$.
\end{lemma}
\begin{proof}
The composite of $\unit \to \MGL$ with $\eta$ is trivial, 
so the factorization exists abstractly.
As for an explicit construction, 
$\eta$ identifies up to a $S^{2,1}$-suspension with the canonical map from the complement of the zero section of the tautological line bundle $\gamma_1\to \PP^1$ to $\PP^1$; 
hence its mapping cone is the Thom space $\Thom_{\PP^1}(\gamma_1)$.
Thus we can model $\unit/\eta$ by the motivic spectrum with evident structure maps given by  
\[ 
(\ast,\Thom_{\PP^1}(\gamma_1),\Thom_{\PP^1}(\A^1\oplus \gamma_1),\dotsc).
\]
Recall that $\MGL$ is the motivic spectrum
\[ 
\bigl(S^0 = \Thom_{S}(\gamma_0),\Thom_{\Gr_1=\PP^\infty}(\gamma_1),\Thom_{\Gr_2}(\gamma_2),\dotsc), 
\] 
where $\gamma_{n}\to \Gr_{n}$ is the tautological vector bundle over the $n$th Grassmannian \cite[\S6.3]{VoevodskyICM1998}.
The resulting map $\unit/\eta\to \MGL$ is determined by the inclusion $\PP^1\hookrightarrow \PP^\infty$ sending $(x:y)$ to $(x:y:0:\dotsm)$, 
which in turn induces $\Thom_{\PP^1}(\gamma_1)\to\Thom_{\PP^\infty}(\gamma_1)$ in spectrum level one.

For every $m\geq 2$, 
the inclusion $\PP^1\hookrightarrow \PP^m$ is the zero section of the vector bundle $\PP^{m} \smallsetminus \PP^{m-2} \to \PP^1$ projecting onto the first two coordinates.
By homotopy purity for vector bundles, 
the cofiber of $\PP^1\hookrightarrow \PP^{m}$ is weakly equivalent to the Thom space of a rank $2$ bundle over $\PP^{m-2}$. 
Taking the colimit produces the Thom space of a rank $2$ bundle over $\PP^{\infty}$. 
Since the cofiber of $\Thom_{\PP^1}(\gamma_1)\to\Thom_{\PP^\infty}(\gamma_1)$ involves the map induced on Thom spaces of a line bundle, 
it is weakly equivalent to the Thom space of a rank $3$ bundle.
Thus the homotopy cofiber of 
\[ 
\bigl(\ast,\Thom_{\PP^1}(\gamma_1),\Thom_{\PP^1}(\A^1\oplus \gamma_1),\dotsc\bigr) 
\to 
\bigl(S^0 = \Thom_{S}(\gamma_0),\Thom_{\Gr_1=\PP^\infty}(\gamma_1),\Thom_{\Gr_1}(\A^1\oplus\gamma_1),\dotsc\bigr) 
\]
is contained in $\Sigma^{4,2}\SH^{\eff}$. 
Since the same holds true for the canonical map 
\[ 
\bigl(S^0=\Thom_{S}(\gamma_0),\Thom_{\Gr_1=\PP^\infty}(\gamma_1),\Thom_{\Gr_1}(\A^1\oplus\gamma_1),\dotsc\bigr) 
\to 
\MGL, 
\]
the result follows.
\end{proof}

We refer to the map constructed in the proof of Lemma~\ref{lem:cone-eta-mgl} as the canonical one.

\begin{lemma}
\label{lem:connectivity-cone-eta-mgl}
Let $S$ be a base scheme such that $\unit$ is $0$-connected.
Then the canonical map $\unit/\eta \to \MGL$ is $1$-connected, 
and likewise for $\f_q(\unit/\eta) \to \f_q(\MGL)$ when $q\in\{0,1,2\}$. 
\end{lemma}
\begin{proof}
For the first statement see also \cite[Theorem 2.8]{hoyois}.
The connectivity of Thom spaces given in Example \ref{example:connectivity-thom} implies together with Lemma~\ref{lem:cone-eta-mgl} that $\unit/\eta \to \MGL$ is $1$-connected. 
Moreover, 
it induces a weak equivalence on $\s_0$ and $\s_1$ by Lemma~\ref{lem:cone-eta-mgl}.
An application of the five lemma implies the assertions for $\f_1$ and $\f_2$.
\end{proof}

\begin{lemma}
\label{lem:cone-eta-conn}
Let $(S,\Lambda)$ be a compatible pair where $S$ is a field.
The map $\f_{q}(\unit_\Lambda/\eta)\to \s_{q}(\unit_\Lambda/\eta)$ is $1$-connected for all $q$.
\end{lemma}
\begin{proof}
Lemmas \ref{lem:slice-convergence-mgl} and \ref{lem:connectivity-cone-eta-mgl} imply that $\f_{q}(\unit_\Lambda/\eta)\to \s_{q}(\unit_\Lambda/\eta)$ is $1$-connected for $q\in \{0,1\}$. 
Moreover, 
$[\Sigma^{s,t} X_+,\f_2(\unit_\Lambda/\eta)\to \s_2(\unit_\Lambda/\eta)]$ is an isomorphism for all $X\in\Sm_{S}$ of dimension $d$ and integers $s,t$ satisfying $s-t+d<1$. 
In order to conclude for $\f_{2}(\unit_\Lambda/\eta)\to\s_{2}(\unit_\Lambda/\eta)$, 
it remains to show the canonical map
\[ 
[\Sigma^{s,t}X_+,\f_2(\unit_{\Lambda}/\eta)
\to 
\s_2(\unit_{\Lambda}/\eta)] 
\]
is surjective for $s-t+d=1$. 
We may assume $S$ is connected.
Since $\f_{2}(\unit_\Lambda/\eta)\to\s_{2}(\unit_\Lambda/\eta)$ is a map between cellular motivic spectra, 
we may assume $X$ is the base scheme by Lemma~\ref{lem:cellular-connectivity}.
Lemma~\ref{lem:slices-unit-hopf} implies the isomorphism $\s_2(\unit_\Lambda/\eta)\iso\Sigma^{3,2}\M\Lambda/12$.
Hence the target of
\[ 
[S^{3,2},\f_2(\unit_{\Lambda}/\eta)
\to 
\s_2(\unit_{\Lambda}/\eta)] 
\]
is the group $\Lambda/12$ (over arbitrary connected base schemes).
By construction of the slice filtration the Hopf map $\nu \colon S^{3,2}\to \unit$ (which exists over $\ZZ$, cf.~\cite[Remark 4.14]{dugger-isaksen.hopf}, and hence over $S$) 
factors through $\f_2(\unit_{\Lambda})$. 
If $S$ maps to $\Spec(\CC)$, 
the complex realization map
\[ 
[S^{3,2},\f_2(\unit/\eta)] \to \pi_3 (\unit_{\Top}/\eta) \iso \ZZ/12
\]
is surjective, 
and hence so is the map 
\[ 
[S^{3,2},\f_2(\unit/\eta)\to \s_2(\unit/\eta)]. 
\]
For an arbitrary field of characteristic zero one can reduce to the previous case since all the constructions commute with filtered colimits.
Suppose $S$ has characteristic $p>0$ and set $(T,\Lambda)=\bigl(\Spec(\ZZ_{(p)}),\ZZ[\tfrac{1}{p}]\bigr)$.
Since the canonical maps 
\[ 
H^{0,0}(S;\Lambda/12)
\leftarrow 
H^{0,0}(T;\Lambda/12) 
\to 
H^{0,0}(\QQ;\Lambda/12)
\]
are isomorphisms,
also 
\[ 
[S^{3,2},\f_2(\unit_\Lambda/\eta)\to \s_2(\unit_\Lambda/\eta)] 
\]
is surjective over $S$.
For any field and integers $s,t$ with $s-t=1$,
surjectivity of
\[ 
[S^{s,t},\f_2(\unit_\Lambda/\eta)\to \s_2(\unit_\Lambda/\eta)] 
\]
follows from the $\bigoplus_{n}\pi_{n,n}\unit$-module structure, 
which induces a surjection
\[ 
\pi_{3,2}\s_2(\unit_\Lambda/\eta)\times \pi_{n,n}\unit_\Lambda 
\to 
\pi_{3+n,2+n}\s_2(\unit_\Lambda/\eta) 
\]
for all $n$. 
(Note that $\unit_\Lambda =\f_0(\unit_\Lambda)\to \f_0(\unit_\Lambda/\eta)\to \s_0(\unit_\Lambda/\eta)=\s_0(\unit_\Lambda)$ is $0$-connected, 
which implies the claimed surjectivity.)
Hence $\f_2(\unit_\Lambda/\eta)\to \s_2(\unit_\Lambda/\eta)$ is $1$-connected.
Since $\s_q(\unit_\Lambda/\eta)$ is at least $2$-connected for every $q>2$ by Lemma~\ref{lem:slices-unit-hopf}, 
it follows that $\f_{q}(\unit_\Lambda/\eta)\to \s_{q}(\unit_\Lambda/\eta)$ is $1$-connected for every integer $q$.
\end{proof}

\begin{lemma}\label{lem:cone-eta-two}
Let $(S,\Lambda)$ be a compatible pair where $S$ is a field.
The canonical map $\unit_\Lambda/\eta\to \slicecomp(\unit_\Lambda/\eta)$ is $1$-connected.
\end{lemma}
\begin{proof}
We show that 
$\holim_{q}\f_{q}(\unit_\Lambda/\eta)$ is $1$-connected. 
In fact, 
since $\f_{q+1}(\unit_\Lambda/\eta) \to \f_{q}(\unit_\Lambda/\eta)$ is $1$-connected for $q\geq 3$ by Lemma~\ref{lem:slices-unit-hopf},
it suffices to prove that $\f_{3}(\unit_\Lambda/\eta)$ is $1$-connected.
This follows from the case $q=2$ of Lemma~\ref{lem:cone-eta-conn}.
\end{proof}

\subsection{Cellularity}
\label{sec:cellularity}

The following definitions modeled on \cite[\S2,8]{dugger-isaksen.cell}, \cite[\S2]{hu-kriz-ormsby} are tailored for our discussion of convergence for the slice filtration.

{\em Attaching a cell\/} to $\E$ refers to the process of forming the pushout of some diagram 
\begin{equation}
\label{eq:cell} 
D^{s+1,t}\unit \hookleftarrow \Sigma^{s,t} \unit \xrightarrow{\alpha} \E. 
\end{equation}
The leftmost map in (\ref{eq:cell}) denotes the canonical inclusion into the simplicial mapping cylinder $D^{s+1,t}\unit$ of $\Sigma^{s,t}\unit \to \ast$. 
Thus the pushout $\D$ consists of $\E$ together with a cell of {\em dimension\/} $(s+1,t)$ and {\em weight\/} $t$. 
More generally, 
one attaches a collection of cells indexed by some set $I$ by forming the pushout of some diagram
\[ \bigvee_{i\in I}D^{s_i,t_i} \unit
\hookleftarrow
\bigvee_{i\in I}\Sigma^{s_i,t_i} \unit 
\xrightarrow{\vee\alpha_{i}} 
\E. 
\]

A {\em cell presentation\/} of $f\colon \D\to \E$ consists of a sequence of motivic spectra
\begin{equation}
\label{eq:cell-presentation} 
\D
=
\D_{-1}
\xrightarrow{d_0} 
\D_0
\xrightarrow{\sim} 
\D^{\prime}_0
\xrightarrow{d_1} 
\D_1
\xrightarrow{\sim} 
\D^{\prime}_1
\to 
\dotsm 
\xrightarrow{d_{n}} 
\D_{n} 
\xrightarrow{\sim} 
\D^{\prime}_n
\xrightarrow{d_{n+1}} 
\dotsm 
\end{equation}
along with the canonical map to the colimit $c\colon \D\to \D_\infty$, 
and attaching maps
\[ 
\bigvee_{i\in I_n}\Sigma^{s_i,t_i} \unit 
\xrightarrow{\alpha_n} 
\D^{\prime}_{n-1} 
\]
such that $\D_{n}$ is obtained by attaching cells to $\D^{\prime}_{n-1}$ along $\alpha_n$, 
and a weak equivalence $w\colon \D_\infty\xrightarrow{\sim} \E$ satisfying $w\circ c =f$. 
Maps labeled with $\sim$ in \eqref{eq:cell-presentation} are acyclic cofibrations.

A cell presentation $\D\to \D_\infty \xrightarrow{\sim} \E$ of a map $\D\to \E$ is of {\em finite type\/} if the following hold:
\begin{enumerate}
\item 
There exists an integer $k$ such that $\D\to \D_\infty$ contains no cells in dimension $(s,t)$ for $s-t<k$. 
\item 
For every integer $n$, 
$\D\to \D_\infty$ contains at most finitely many cells of dimension $(s,t)$ for $s-t=n$. 
\end{enumerate}
\begin{remark}
\label{rem:relative-cell-pres}
A cell presentation of $\ast\to \E$ is simply called a cell presentation of 
$\E$ -- the most important (absolute) case. 
The distinction between a motivic spectrum and a cell presentation thereof will often be suppressed.
\end{remark}

\begin{remark}
\label{rem:R-finite-type}
The exact same definitions apply to $\Lambda$-local motivic spectra.
However, a $\Lambda$-cell presentation of finite type need not be a cell presentation of finite type.
\end{remark}

\begin{example}
\label{ex:bounded}
The suspension spectrum $\Sigma^\infty \PP^\infty$ admits a cell presentation of finite type with one cell of dimension $(2n,n)$ for every natural number $n$.
If $\E$ admits a cell presentation of finite type, 
then so does $\Sigma^{m,n}\E$.
The $(s,t)$-cells of $\E$ become $(s+m,t+n)$-cells of $\Sigma^{m,n}\E$.
\end{example}

\begin{proposition}
\label{prop:MZcellular}
The motivic spectrum $\MGL$ admits a cell presentation of finite type.
If $(S,\Lambda)$ is a compatible pair,
then $\M \Lambda$ admits a $\Lambda$-cell presentation of finite type.
\end{proposition}
\begin{proof}
Example \ref{ex:bounded} shows the suspension spectrum of $\MGL_n$ admits a cell presentation of finite type if $n=1$; 
the general case follows similarly. 
We conclude for $\MGL$ using Lemma~\ref{lemma:equivalentcellularfinitetype} below.
The result for $\M \Lambda$ was proven in \cite[Proposition 8.1]{hoyois} for base schemes that are essentially smooth over a field, 
and more generally in \cite[Corollary 11.4]{spitzweckmz}.
Here we use Definition~\ref{def:compatible} (1) of a compatible pair.
\end{proof}

Next we turn to a few auxiliary results on cell presentations.

\begin{lemma}
\label{lem:cellular-presentation}
A motivic spectrum is cellular if and only if it has a cell presentation.
\end{lemma}
\begin{proof}
Let $\E$ be a fibrant cellular motivic spectrum.
To construct a cell presentation,
let 
\[ 
\D_0:= 
\bigvee_{\alpha\in \pi_{\ast,\ast}\E} \Sigma^{\ast,\ast} \unit 
\xrightarrow{\vee\alpha} 
\E. 
\]
Then $\D_0\to \E$ factors through an acyclic cofibration $\D_0\xrightarrow{\sim} \D_0^\prime$ and a fibration $\D_0^\prime \to \E$.
By construction, 
$\pi_{\ast,\ast}(\D_0^\prime)\to\pi_{\ast,\ast}(\E)$ is surjective.
We choose lifts of generators of its kernel, 
and use these to attach cells to  $\D_0^\prime$.
The resulting map $\D_1\to \E$ factors again through an acyclic cofibration and a fibration $\D_1^\prime \to \E$, 
which is also a $\pi_{\ast,\ast}$-surjection.
Iterating this procedure leads to an $\pi_{\ast,\ast}$-isomorphism $\D_\infty=\colim_n{\D_n}\to\E$.
Here injectivity uses compactness of $\Sigma^{s,t}\unit$ so that any element in the kernel lifts to a finite stage, 
and is therefore annihilated at the next stage.
Now since $\D_\infty$ and $\E$ are cellular, 
the map $\D_\infty\to\E$ is a weak equivalence \cite[Corollary 7.2]{dugger-isaksen.cell}.
The converse implication follows directly from the definitions.
\end{proof}

\begin{lemma}
\label{lem:rearrange}
The cells in any $\Lambda$-cell presentation of finite type can be rearranged according to their connectivity (in the sense of Definition \ref{def:connectivity}).
\end{lemma}
\begin{proof}
Without loss of generality we may assume the cell presentation $\D\to \E_\infty \xrightarrow{\sim}\E$ is one of a map $\D\to \E$ of fibrant motivic spectra such that 
each step $\D_{n}^{\prime}\xrightarrow{d_{n+1}} \D_{n+1}$ consists of attaching exactly one cell of bidegree $(s_n,t_n)$. 
We may further assume that $\D_n^{\prime}$ is fibrant.
Now suppose there exist natural numbers $m<n$ such that $s_m-t_m>s_n-t_n$. 
Let $k$ be the smallest natural number such that the sequence
$
(s_0-t_0,s_1-t_1,\dotsc,s_k-t_k) 
$
is not increasing, 
so that $s_{k-1}-t_{k-1}>s_k-t_k$. 
Then the canonical map $\D_{k-1}^{\prime} \to \D_k$ is $s_{k-1}-t_{k-1}-1$-connected by Example~\ref{example:connectivity-thom}.
Thus the attaching map for the $k$th cell
$ 
\alpha_{k+1}
\colon 
\Sigma^{s_k-1,t_k}\unit_\Lambda 
\to 
\D_{k}^{\prime} 
$,
considered as an element in $\pi_{s_k-1,t_k} \D_{k}^{\prime}$, lifts to $\pi_{s_k-1,t_k} \D_{k-1}^{\prime}$. 
Choose a representative $\beta_{k+1}\colon \Sigma^{s_k-1,t_k}\unit_\Lambda \to \D_{k-1}^{\prime}$ and a homotopy 
$H\colon \Sigma^{s_k-1,t_k}\unit_\Lambda\smash \Delta^1_+ \to \D_{k}^{\prime}$ from 
$\Sigma^{s_k-1,t_k}\unit_\Lambda\xrightarrow{\beta_{k+1}} \D_{k-1}^{\prime}\xrightarrow{d_{k}}\D_{k}\xrightarrow{\sim} \D_{k}^{\prime}$
to the attaching map
$ 
\alpha_{k+1}\colon \Sigma^{s_k-1,t_k}\unit_\Lambda \to \D_{k}^{\prime} 
$
for the $k$th cell. 
Factor the canonical map $\D_{k-1}^\prime \to \D_{k}^\prime$ as an acyclic cofibration $\D_{k-1}^\prime \xrightarrow{\sim} \C_{k-1}^\prime$ followed by a fibration 
$q\colon \C_{k-1}^\prime \to \D_{k}^{\prime}$.
Lifting $H$ in the commutative diagram
\[ 
\xymatrix{\Sigma^{s_k-1,t_k}\unit_\Lambda\smash \Delta^0_+ \ar[r]^-{\beta_{k+1}}\ar[d] &
\C_{k-1}^{\prime} \ar[d]^-q \\ 
\Sigma^{s_k-1,t_k}\unit_\Lambda\smash \Delta^1_+ \ar[r]^-H & \D_{k}^{\prime}}
\]
produces a homotopy whose restriction to the other inclusion is an attaching map 
\[ 
\gamma_{k}\colon \Sigma^{s_k-1,t_k}\unit_\Lambda \to \C_{k-1}^{\prime}. 
\]
Define $\C_k$ by attaching a cell with respect to $\gamma_k$ and $\C_{k+1}$ by attaching a cell with respect to 
\[ 
\gamma_{k+1}\colon \Sigma^{s_{k-1}-1,t_{k-1}}\unit_\Lambda 
\to \D_{k-1}^{\prime} 
\xrightarrow{\sim} \C_{k-1}^{\prime} \to \C_k. 
\]
The universal property of pushouts --- applied twice --- produces a map $v_{k+1}\colon \C_{k+1} \to \D_{k+1}$ and a commutative diagram of homotopy cofiber sequences:
\[ \xymatrix{
\D_{k-1}^{\prime} \ar[r] & \D_{k+1} \ar[r] & \D_{k+1}/\D_{k-1}^{\prime} \\
\D_{k-1}^{\prime} \ar[r]\ar[d]_-\sim \ar[u]^-{=} & \C_{k+1} \ar[d]^-{=} \ar[r] \ar[u]_-{v_{k+1}} & \C_{k+1}/\D_{k-1}^{\prime} \ar[d]^-\sim \ar[u]_-w \\
\C_{k-1}^{\prime} \ar[r] & \C_{k+1}  \ar[r] & \C_{k+1}/\C_{k-1}^{\prime}}
\]
The map $w$ is a weak equivalence by construction; 
its source and target are weakly equivalent to $\Sigma^{s_{k-1},t_{k-1}}\unit_\Lambda\vee \Sigma^{s_{k},t_{k}}\unit_\Lambda$. 
Hence $v_{k+1}$ is a weak equivalence, 
producing a cell presentation $\C_\infty\xrightarrow{\sim}\D_\infty = \E^\prime$ of higher connectivity.
Iterating this finitely many times yields a cell presentation of $\E$ (with cells of the same degrees) in which the first $k+1$ cells are attached in their order of connectivity.
Proceeding by induction on the cell filtration of $\E$ provides the desired statement.
\end{proof}

The next result is essentially a modification of \cite[Lemma 7]{hu-kriz-ormsby} to our setup.

\begin{lemma}
\label{lem:presentation-connected}
Let $k$ be an integer, 
and let $f\colon \D\to \E$ be a $k$-connected map of motivic spectra.
If $f$ has a $\Lambda$-cell presentation of finite type, 
then it admits a $\Lambda$-cell presentation of finite type whose cells are at least $k$-connected.
\end{lemma}
\begin{proof}
We may assume $f$ is a map between bifibrant motivic spectra. 
Let 
\[ 
\D=
\D_{-1}
\to 
\D_0
\xrightarrow{\sim} 
\D_0^\prime
\to 
\D_1
\xrightarrow{\sim} 
\D_1^\prime
\to 
\dotsc 
\to 
\D_\infty 
\xrightarrow{\sim}
\E
\]
be a $\Lambda$-cell presentation of finite type, 
where $\D_n^\prime$ is fibrant for every $n$. 
By Lemma~\ref{lem:rearrange} we may assume the $\Lambda$-cells are attached according to increasing connectivity.
Consider the first attaching $\Lambda$-cell 
and $\D_0=D^{s,t}\unit_\Lambda\cup_{\Sigma^{s-1,t}\unit_\Lambda}\D_{-1}$. 
If $s-t\geq k$, 
there is nothing to prove. 
If $s-t<k$, 
the canonical map $(\Sigma^{s-1,t}\unit_\Lambda,D^{s,t}\unit_\Lambda)\to (\D,\E)$ lifts to $(\D,\D)$ up to homotopy since $f\colon \D\to \E$ is $k$-connected. 
In other words the attaching map $\Sigma^{s-1,t}\unit_\Lambda\to \D$ is homotopic to a constant map. 
Choose a homotopy $H\colon \Sigma^{s-1,t}\unit_\Lambda\smash \Delta^1_+\to \D$ from the given attaching map to the constant map, 
and let $\C_0$ be the pushout of the diagram
\[
\xymatrix{
D^{s-1,t}\unit_\Lambda\smash \Delta^1_+ & \Sigma^{s-1,t}\unit_\Lambda\smash \Delta^1_+\ar[r] \ar[l] &  \D. }
\]
There are canonical acyclic cofibrations $\D_0\xrightarrow{\sim} \C_0 \xleftarrow{\sim}\D\vee \Sigma^{s,t}\unit_\Lambda$.
Attaching an $(s+1,t)$-$\Lambda$-cell via the map $\Sigma^{s,t}\unit_\Lambda\to \D\vee \Sigma^{s,t}\unit_\Lambda\to \C_0$ to $\C_0$ produces a motivic spectrum $\B_0$ receiving an
acyclic cofibration from $\D$. 
Let $\B_0^\prime$ be the pushout of the diagram
\[
\xymatrix{
\B_0 & \D_0\ar[r]^-\sim \ar[l] &  \D_0^\prime. }
\]
The canonical map $\D\to \B_0^\prime$ is an acyclic cofibration. 
Choose a retraction $\B_0^\prime\xrightarrow{\sim} \D$. 
By construction $\D_0^\prime\to \E$ has a $\Lambda$-cell presentation with precisely one $\Lambda$-cell of degree $(s,t)$ less than in the $\Lambda$-cell presentation of $f\colon \D\to \E$.
Cobase change along $\D_0^\prime\to \B_0^\prime\xrightarrow{\sim} \D$ produces a $\Lambda$-cell presentation for the induced map $\D \to \D_\infty\cup_{\D_0^\prime}\D$.
The target has the homotopy type of $\E\vee \Sigma^{s+1,t}\unit_\Lambda$ by construction. 
Thus attaching a $\Lambda$-cell of degree $(s+2,t)$ to $\D_\infty\cup_{\D_0^\prime}\D$ via the canonical map produces a motivic spectrum mapping via a weak equivalence to $\E$. 
Hence $\E$ admits a $\Lambda$-cell presentation with precisely one $\Lambda$-cell of degree $(s,t)$ less than $\D\to \D_\infty$ 
and precisely one $\Lambda$-cell of degree $(s+2,t)$ more than $\D\to \D_\infty$.
Iterating this procedure finitely many times provides the desired statement.
\end{proof}

Any map $\D\to \E$ of motivic spectra such that $\D$ and $E$ admit $\Lambda$-cell presentations of finite type acquires a $\Lambda$-cell presentation of finite type, 
as one deduces from applying the simplicial mapping cylinder.
Hence Lemma~\ref{lem:presentation-connected} applies to any map of motivic spectra admitting $\Lambda$-cell presentations of finite type.

Also \cite[Lemma 8]{hu-kriz-ormsby} requires a modification in our context.

\begin{lemma}
\label{lemma:equivalentcellularfinitetype}
Let $(\phi(n))_{n}$ be an increasing sequence of integers diverging to $+\infty$.
Suppose $\E$ is a $\Lambda$-cellular motivic spectrum.
If for each integer $n$ there exists a motivic spectrum $\E_{n}$ with a $\Lambda$-cell presentation of finite type and a $\phi(n)$-connected map $\E_{n}\to\E$, 
then $\E$ has a $\Lambda$-cell presentation of finite type.
\end{lemma}
\begin{proof}
The assumption on the connectivity of $\E_{n}\to \E$ implies there exists a $\phi(n)$-connected map $f_{n}\colon \E_{n}\to \E_{n+1}$ and a homotopy commutative diagram:
\[ 
\xymatrix{
\E_{n} \ar[rr] \ar[rd] & & \E_{n+1} \ar[ld] \\
& \E & }
\]
Lemma~\ref{lem:presentation-connected} implies $f_{n}$ admits a $\Lambda$-cell presentation $w_{n}\circ c_{n}$ of finite type comprised of $\Lambda$-cells of connectivity at least $\phi(n)$. 
The colimit of the maps $c_{n}$ yields the desired $\Lambda$-cell presentation of finite type.
\end{proof}

The following statement adapts \cite[Comment after Lemma 8, p.~581]{hu-kriz-ormsby} to our notion of connectivity.

\begin{lemma}
\label{lem:conn-cell-reverse}
Let $g\colon \E\to \C$ be a $k$-connected map of motivic spectra for $\C$ cellular. 
Choose a cell presentation $c\colon \C^\prime \xrightarrow{\sim} \C$ and a fibrant replacement $e\colon \E\xrightarrow{\sim} \E^\dagger$.
Then there exists a $k-1$-connected cellular inclusion $i\colon \D\hookrightarrow \C^\prime$ and a $k-1$-connected map $h\colon \D\to \E^\dagger$ such that $g\circ e^{-1} \circ h$ equals 
$c\circ i$ in the homotopy category. 
In particular, 
if $\C^\prime$ is of finite type or has no cells below a fixed weight, 
then the same holds for $\D$.
\end{lemma}
\begin{proof}
Choose a diagram $\ast = \C_0^\prime\to \C_1^\prime\to \dotsm $, 
with $\C_{n+1}^\prime$ obtained by attaching cells to $\C_{n}^\prime$, 
whose colimit is $\C^\prime$.
Let $\D$ be the motivic spectrum comprised of all cells in $\C^\prime$ of dimension $(s,t)$ with $s-t\leq k$. 
It is constructed inductively by restricting the corresponding attaching map. 
By Example~\ref{example:connectivity-thom} the restriction of the $n$th attaching map to cells of dimension $(s,t)$ with $s-t\leq k$ to $\C_{n}^\prime$ factors through 
$\D_{n}\to \C_{n}^\prime$ up to homotopy. 
This allows us to form the motivic spectrum $\D_{n+1}$.
  
The reverse map $\D\to \E$ in $\SH$ is constructed inductively, starting with the point. 
Suppose $h_{n}\colon \D_{n}\to \E$ has been constructed with $g\circ h_{n} = i_{n}$ in $\SH$, 
where $i_{n}\colon \D_{n}\to \C_{n}^\prime\to \C^\prime$ is the inclusion described above. 
Since $g\colon \E\to \C$ is $k$-connected, 
every map $\Sigma^{s,t}\unit \to \D_{n}\to \C_{n}^\prime\to \C$ with $s-t<k$ lifts to $\E$.
This yields a map $h_{n+1}\colon \D_{n+1} \to \E$ such that $g\circ h_{n+1}=i_{n+1}$ in $\SH$. 
The inclusion $\D\hookrightarrow \C^\prime$ is $k-1$-connected by Example~\ref{example:connectivity-thom}, 
hence so is $h$.
\end{proof}

\begin{lemma}
\label{lem:connectivity-slices}
If $\E$ is a $k$-connected cellular motivic spectrum, 
then $\f_{n}(\E)$ and $\s_{n}(\E)$ are $k$-connected for all $n\in \ZZ$ and $\slicecomp(\E)$ is at least $k-1$-connected.
\end{lemma}
\begin{proof}
The result holds for the motivic sphere spectrum $\unit$ by Theorem~\ref{theorem:sphereslices},
and hence for all spheres.
To conclude recall that $\f_{n}$ and $\s_{n}$ commute with homotopy colimits, 
and a homotopy limit of $k$-connected motivic spectra is at least $k-1$-connected.
\end{proof}

\subsection{Convergence}
\label{sec:convergence-sub}
We are interested in finding conditions on a motivic spectrum such that its slice completion admits a $\Lambda$-cell presentation of finite type.

\begin{definition}
\label{defn:slice-connected}
A motivic spectrum $\E$ is called {\em slice-connected\/} if for every $n\in \ZZ$ there exists a natural number $\phi(n)$ with the properties:
\begin{enumerate}
\item The $n$th slice $\s_{n}(\E)$ is $\phi(n)$-connected.
\item The sequence $(\phi(n))$  is non-decreasing and diverges to $+\infty$.
\end{enumerate}
\end{definition}

\begin{definition}
\label{def:slice-finitary}
A motivic spectrum $\E$ is called $\Lambda$-{\em slice-finitary\/} if for every $n\in \ZZ$ there exist natural numbers $\phi(n)\leq \psi(n)$ 
and a finite collection $\{C_{n,0},C_{n,1},\dotsc,C_{n,\psi(n)}\}$ of finitely presented $\Lambda$-modules with the properties:
\begin{enumerate}
\item The $n$th slice $\s_{n}(\E)$ is weakly equivalent to the (finite) sum of motivic Eilenberg-MacLane spectra 
$\Sigma^{n+\phi(n),n}\MC_{n,0}\vee\Sigma^{n+\phi(n)+1,n}\MC_{n,1}\vee\dotsm\vee\Sigma^{n+\phi(n)+\psi(n),n}\MC_{n,\psi(n)}$.
\item The sequence $(\phi(n))$ is non-decreasing and diverges to $+\infty$.
\item There exists an integer $e$ such that $\f_e(\E)=\E$.
\end{enumerate}
A $\ZZ$-slice-finitary motivic spectrum is simply called slice-finitary.
\end{definition}

\begin{example}
\label{example:slice-connected}
Algebraic cobordism is $\Lambda$-slice-finitary. 
The Moore spectrum $\unit_\Lambda/p^m$ is $\Lambda$-slice-finitary for every odd prime $p$ and natural number $m$. 
The motivic spectrum $\unit_\Lambda[\tfrac{1}{2}]$ is slice-connected,
but not $\Lambda$-slice-finitary (although it is $\Lambda[\tfrac{1}{2}]$-slice-finitary).
In fact, 
\[ 
\phi(n) 
= 
\begin{cases} n & n \mathrm{\ odd} \\ 
n-1 & n \mathrm{\ even}
\end{cases}
\] 
is a sequence of lower bounds for the connectivity of the slices of $\unit_\Lambda[\tfrac{1}{2}]$.
However, 
$\f_1(\unit_\Lambda[\tfrac{1}{2}])$ is $\Lambda$-slice-finitary.
Theorem~\ref{theorem:zahlers-hope} implies $\unit_\Lambda/\eta^m$ is $\Lambda$-slice-finitary for all $m\geq 1$. 
\end{example}

\begin{lemma}
\label{lem:slice-ctd-fin}
Any $\Lambda$-slice-finitary motivic spectrum is slice-connected.
\end{lemma}
\begin{proof}
This follows from the isomorphisms for $X\in \Sm_S$
\begin{align*}
[\Sigma^{s,t}  X_+,\s_{n}\E] & \iso [\Sigma^{s,t}  X_+,\Sigma^{n+\phi(n),n}  \MC_{n,0} \vee \dotsm \vee \Sigma^{n+\psi(n),n}  \MC_{n,\psi(n)}] \\
& \iso H^{n+\phi(n)-s,n-t}(X;C_{n,0}) \directsum\dotsm \directsum  H^{n+\psi(n)-s,n-t}(X;C_{n,\psi(n)}) 
\end{align*}
and the assumption on the motivic Eilenberg-MacLane spectrum in Definition~\ref{def:connected-pair}.
\end{proof}

\begin{lemma}
\label{lem:slice-ctd-fin-triangulated}
The full subcategories of slice-connected and $\Lambda$-slice-finitary motivic spectra of $\SH$ are triangulated.
\end{lemma}
\begin{proof}
For slice-connected motivic spectra this follows since $\s_{n}$ is a triangulated functor. 
Lemma \ref{lem:steenrod-alg-weight-zero} allows us to conclude for $\Lambda$-slice-finitary motivic spectra.
\end{proof}

\begin{lemma}
\label{lem:slice-connected}
Suppose $\E$ is a slice-connected motivic spectrum for the sequence $(\phi(n))$.
Then for every $n\in \ZZ$, 
the maps $\slicecomp(\E) \to \f^{n-1}(\E)$ and $\holim_{q}\f_{q}(\E) \to \f_{n}(\E)$ are $\phi(n)$-connected.
\end{lemma}
\begin{proof}
If $X\in\Sm_{S}$ is of dimension at most $d$, 
the homotopy cofiber sequence
\[ 
\s_{n}(\E)
\to 
\f^{n+1}(\E) 
\to 
\f^n(\E) 
\to 
\Sigma^{1,0}\s_{n}(\E)
\]
implies that $[\Sigma^{s,t}X_+,\f^{n+1}(\E)\to \f^n(\E)]$ is an isomorphism for $s-t+d<\phi(n)$. 
Since $(\phi(n))$ is non-decreasing and diverges to $+\infty$, 
the map  $\f^{n+m}(\E)\to \f^n(\E)$ is $\phi(n)$-connected for every natural number $m$. 
Thus the $\lim^1$-term in the $\lim$-$\lim^1$-short exact sequence for 
\[ 
[\Sigma^{s,t}X_+,\slicecomp(\E)=\holim_{q\to \infty}\f^{q-1}(\E)]
\]
vanishes, 
and for all $s,t$ and $X\in \Sm_{S}$ there is an isomorphism
\[ 
[\Sigma^{s,t}X_+,\slicecomp(\E)]
\iso 
\lim_{q\to \infty}[\Sigma^{s,t}X_+\f^{q-1}(\E)].
\]
The asserted connectivity follows from the analogous statement for $\holim_{q} \f_{q}(\E)$.
\end{proof}

\begin{lemma}
\label{lem:hocolim-slice-ctd}
Let $\alpha\mapsto \E_\alpha$ be an $I$-indexed diagram of slice-connected motivic spectra. 
Suppose there exists a non-decreasing sequence $(\phi(n))$ with $\lim_{n\to \infty}\phi(n)=+\infty$ such that $\s_{n}(\E_\alpha)$ is at least $\phi(n)$-connected for every $\alpha$. 
Then the homotopy colimit 
\[ 
\hocolim_{\alpha\in I} \E_\alpha 
\]
is slice-connected, 
and the $n$th slice $\s_{n} (\hocolim_{\alpha\in I} \E_\alpha)$ is at least $\phi(n)$-connected. 
Moreover, 
there is a natural weak equivalence
\[  
\hocolim_{\alpha\in I} \slicecomp(\E_\alpha)
\xrightarrow{\sim} 
\slicecomp\bigl( \hocolim_{\alpha\in I} \E_\alpha\bigr). 
\]
\end{lemma}
\begin{proof}
Use that $\s_{n}$ and $\f^n$ commute with homotopy colimits and connectivity is preserved under homotopy colimits.
The second part follows from Lemma~\ref{lem:slice-connected}.
\end{proof}

\begin{proposition}
\label{prop:slicecomp-cellular}
If $\E$ is a $\Lambda$-slice-finitary motivic spectrum then the slice completion $\slicecomp(\E)$ has a $\Lambda$-cell presentation of finite type.
\end{proposition}
\begin{proof}
Lemma~\ref{lem:slice-connected} implies that $\slicecomp(\E) \to \f^n(\E)$ is $\phi(n)$-connected.
The assumption on $\E$ implies that all slices have a $\Lambda$-cell presentation of finite type by Proposition~\ref{prop:MZcellular}. 
Since $\E$ is $e$-effective for some integer $e$, 
$\f^n(\E)$ is $e$-effective and has a $\Lambda$-cell presentations of finite type. 
Lemma~\ref{lem:conn-cell-reverse} allows us to reverse the canonical maps in order to produce a $\phi(n)-1$-connected map $\D_n\to \slicecomp(\E)$ for every $n\geq e$.
We may arrange that the reversion process produces a $\phi(n)-1$-connected map $\D_n\to \D_{n+1}$ for every $n\geq e$, 
which commutes up to homotopy with the reversed maps. 
This produces an $\infty$-connected map $\hocolim_n \D_n \to \slicecomp(\E)$.
In particular, 
$\slicecomp(\E)$ has a $\Lambda$-cell presentation of finite type by Lemma~\ref{lemma:equivalentcellularfinitetype}.
\end{proof}

\begin{proposition}
\label{prop:holim-eff-covers}
Let $\E$ be a $\Lambda$-slice-finitary motivic spectrum with a $\Lambda$-cell presentation of finite type.
Then $\holim_{q} \f_{q}(\E)$ admits a $\Lambda$-cell presentation of finite type, and its slices are trivial.
\end{proposition}
\begin{proof}
Lemma~\ref{lem:slice-connected} implies the canonical map $\holim_{q} \f_{q}(\E) \to \f_{n}(\E)$ is $\phi(n)$-connected. 
Since $\E$ is slice-connected and the condition on the orders appearing in the slices holds, 
the slices of $\E$ are cellular of finite type by Proposition~\ref{prop:MZcellular}.
Since $\E$ is of finite type, 
in particular effective, 
$\f_{q}(\E)$ is of finite type.
Lemma~\ref{lem:conn-cell-reverse} reverses the canonical map $\holim_{q} \f_{q}(\E) \to \f_{n}(\E)$ to produce an $\phi(n)-1$-connected map $\D_{n} \to \holim_{q} \f_{q}(\E)$, 
with $\D_{n}\hookrightarrow \f_{n}(\E)$ a cellular inclusion. 
In particular, 
$\D_{n}$ is of finite type, 
with no cells of weight strictly less than $n$. 
As in Proposition~\ref{prop:slicecomp-cellular} 
we obtain a weak equivalence
\begin{equation}
\label{equation:inftyconnected}
\hocolim_{n} \D_{n} 
\xrightarrow{\sim}
\holim_{q} \f_{q}(\E). 
\end{equation}
Since slices commute with homotopy colimits and $\D_{n}$ is $n$-effective, 
$\s_{\ast}(\holim_{q}\f_{q}(\E))=\ast$.
Moreover,
$\hocolim_{n} \D_{n}$ is of finite type.
\end{proof}

\begin{lemma}
\label{lem:connectivity-slicecomp-cell-ft}
Let $\Phi$ be an endofunctor of the category of motivic spectra which commutes with homotopy colimits,
commutes with $\Sigma^{a,b}$ up to weak equivalence,
and preserves cellularity and connectivity in the sense that $\conn(\Phi(\E))$-$\conn(\E)$ is constant.
If the canonical map
\[ 
\Phi(\unit)
\to 
\slicecomp(\Phi(\unit)) 
\]
is $1$-connected, 
then for every cellular motivic spectrum $\E$ of finite type the canonical map
\[ 
\Phi(\E)
\to 
\slicecomp(\Phi(\E)) 
\]
is $1+\conn(\E)$-connected.
\end{lemma}
\begin{proof}
Induction on the number of cells implies the result for motivic spectra with finitely many cells.
If $\E$ is cellular of finite type, 
consider the homotopy cofiber sequence
\[ 
\Phi(\E_{(n)}) 
\to
\Phi(\E) 
\to 
\Phi(\E/\E_{(n)}) 
\to 
\Sigma^{1,0}  
\Phi(\E_{(n)}), 
\]
where $\E_{(n)}$ is the subspectrum of $\E$ comprising the first $n$ cells. 
Since the connectivity of $\E/\E_{(n)}$ increases with $n$, 
and hence by assumption on $\Phi$ and Lemma~\ref{lem:connectivity-slices} also the connectivity of the slice completion $\slicecomp(\Phi(\E/\E_{(n)}))$,
the result follows.
\end{proof}

\begin{corollary}
\label{cor:slice-comp-cell-finite-type}
Let $(S,\Lambda)$ be a compatible pair where $S$ is a field.
Suppose $\E$ has a $\Lambda$-cell presentation of finite type.
Then the canonical map  
\[ 
\E/\eta
\to 
\slicecomp(\E/\eta)
\]
is $1+\conn(\E)$-connected.
\end{corollary}
\begin{proof}
This follows from Lemmas \ref{lem:cone-eta-two} and \ref{lem:connectivity-slicecomp-cell-ft}.
\end{proof}

\begin{proposition}
\label{prop:slice-convergence}
Let $(S,\Lambda)$ be a compatible pair where $S$ is a field.
Suppose $\E$ has a $\Lambda$-cell presentation of finite type.
Then the canonical map $\E/\eta\to \slicecomp(\E/\eta)$ is a weak equivalence.
\end{proposition}
\begin{proof}
There is a commutative diagram of homotopy cofiber sequences:
\begin{equation}
\label{equation:etadiagram} 
\xymatrix{
\holim_{q} \f_{q} (\E/\eta) \ar[r] \ar[d] & \E/\eta \ar[r]\ar[d] & \slicecomp{(\E/\eta )} \ar[d] \\
\slicecomp\bigl(\holim_{q} \f_{q} (\E/\eta )\bigr) \ar[r]  & \slicecomp(\E/\eta )\ar[r] & \slicecomp\bigl(\slicecomp{(\E/\eta )}\bigr) 
}
\end{equation}
Example~\ref{example:slice-connected} implies $\E/\eta $ is slice-finitary.
By Proposition~\ref{prop:holim-eff-covers} the right vertical map in (\ref{equation:etadiagram}) is a weak equivalence. 
Note that $\slicecomp(\holim_{q} \f_{q} (\E/\eta ))$ is contractible, 
see Proposition~\ref{prop:holim-eff-covers}. 
The homotopy limit $\holim_{q} \f_{q} (\E/\eta )$ is $1+\conn(\E)$-connected by Corollary~\ref{cor:slice-comp-cell-finite-type}, 
and cellular of finite type by Proposition~\ref{prop:holim-eff-covers}.
Applying Corollary~\ref{cor:slice-comp-cell-finite-type} to $\holim_{q}\f_{q}(\E/\eta )$ shows the left vertical map in (\ref{equation:etadiagram}) is $2+\conn(\E)$-connected. 
By the five lemma the middle vertical map in (\ref{equation:etadiagram}) is $2+\conn(\E)$-connected.
Thus the left vertical map in (\ref{equation:etadiagram}) is $3+\conn(\E)$-connected. 
Iterating this argument finishes the proof.
\end{proof}

We can now give a criterion for when the slice completion coincides with the $\eta$-completion.
\begin{theorem}
\label{theorem:eta-slice-completion}
Let $(S,\Lambda)$ be a compatible pair where $S$ is a field.
Suppose $\E$ has a $\Lambda$-cell presentation of finite type.
Then there is a canonical zig-zag of weak equivalences between the slice completion $\slicecomp(\E)$ and the $\eta$-completion $\E^{\wedge}_{\eta}$. 
\end{theorem}
\begin{proof}
This follows from Lemma~\ref{lemma3} and Proposition~\ref{prop:slice-convergence}.
\end{proof}

Theorems \ref{theorem:sphereslices} and \ref{theorem:eta-slice-completion} identify the $+$-part of the rational motivic sphere spectrum as first announced in \cite{morelsplitting}, 
see \cite[Theorem 16.2.13]{cisinski-deglise} for an alternate formulation and proof.
\begin{corollary}
\label{corollary:1QplusMQ}
For $S$ a field, 
the unit map for the motivic Eilenberg-MacLane spectrum induces an isomorphism $\unit_{\Q}^{+}\cong\MQ$ in $\SH_{\Q}$.
\end{corollary}
\begin{proof}
Theorem \ref{theorem:eta-slice-completion} applies to $\unit_{\Q}$ and $\MQ$ (note that $\MQ$ has a rational cell presentation of finite type by Proposition \ref{prop:MZcellular}).
The Hopf relation $\epsilon\eta=\eta$ implies the $\eta$-completion of $\unit_{\Q}$ is the $+$-part $\unit_{\Q}^{+}$ of the rational motivic sphere spectrum.
Thus there are conditionally convergent slice spectral sequences
\begin{equation}
\label{equation:splicespectralsequence1}
\pi_{p,n}\s_{q}(\unit_{\Q})
\Longrightarrow
\pi_{p,n}(\unit_{\Q}^{+}),
\end{equation}
\begin{equation}
\label{equation:splicespectralsequence2}
\pi_{p,n}\s_{q}(\MQ)
\Longrightarrow
\pi_{p,n}(\MQ).
\end{equation}
Theorem \ref{theorem:sphereslices} shows the unit map induces an isomorphism of slices 
\[ 
\s_{q}(\unit_{\Q})
\overset{\cong}{\to}
\s_{q}(\MQ)
\iso 
\begin{cases} 
\MQ & q=0 \\
\ast & q\neq 0,
\end{cases}
\]
and hence an isomorphism between \eqref{equation:splicespectralsequence1} and \eqref{equation:splicespectralsequence2}.
The filtered target groups are isomorphic by \cite[Theorem 7.2]{Boardman}.
Cellularity of $\unit_{\Q}^{+}$ and $\MQ$ finish the proof.
\end{proof}

\begin{corollary}
\label{corollary:1QMQ}
Suppose $S$ is a field with finite mod 2 \'etale cohomological dimension or positive characteristic.
Then there is an isomorphism $\unit_{\Q}\cong\MQ$ in $\SH_{\Q}$.
\end{corollary}
\begin{proof}
Combine \cite[Lemma 6.8]{levine.sliceconvergence} showing $\unit_{\Q}^{-}=\ast$ and Corollary \ref{corollary:1QplusMQ}.
\end{proof}

\begin{corollary}
\label{corollary:1Qff}
If $S$ is a finite field there are isomorphisms
\[  
\pi_{p,q}(\unit)\otimes\Q
\cong
\begin{cases}
\Q & (p,q)=(0,0) \\
0  & \mathrm{otherwise}.
\end{cases}
\]
If $S$ is a totally imaginary number field with $s$ pairs of complex embeddings,
there is an isomorphism $\pi_{-1,-1}(\unit)\otimes\Q\cong\Q^\infty$ and in all other degrees
\[  
\pi_{p,q}(\unit)\otimes\Q
\cong
\begin{cases}
\Q & (p,q)=(0,0) \\
\Q^s  & (p,q)=(-1,\mathrm{odd}) \\
0  & \mathrm{otherwise}.
\end{cases}
\]
\end{corollary}
\begin{proof}
If $S$ is a finite field, 
Quillen's computation of the algebraic $K$-groups of a finite field \cite{quillen.kfq} shows
$K_{n}(S)$ is a finite group for all $n>0$.
If $S$ is a totally imaginary number field,
Borel's computation \cite{borel} provides the identification of 
$K_n(S)\otimes \Q$. 
Using the rational degeneration of the slice spectral sequence for $\KGL$, 
the identifications follow from Corollary \ref{corollary:1QMQ}.
\end{proof}

\section{Differentials for the motivic sphere spectrum} 
\label{sec:slice-spectr-sequ}
We start this section by describing components of the first slice differential for $\unit_{\Lambda}$, where $(S,\Lambda)$ is a compatible pair and $S$ is connected.
This leads to a discussion of higher differentials which eventually allows us to compute the first motivic stable stems.

\subsection{The first slice differential}
\label{sec:first-differential}
We recall that $\Ext^{1,2i}_{\MU_{\ast}\MU}(\MU_{\ast},\MU_{\ast})$ is a finite cyclic group of order $a_{i}$ introduced prior to Lemma \ref{lem:mult-slices-sphere};
see also Appendix \ref{sec:ext-groups-complex}.
Let $\overline{\alpha}_{i}$ be the generator of its $p$-primary component $\Ext^{1,2i}_{\BP_{\ast}\BP}(\BP_{\ast},\BP_{\ast})$.
We denote motivic Steenrod operations by $\Sq^{i}$ \cite[\S9]{Voevodsky.reduced}. 
Let $\tau$ be the generator of $h^{0,1}\cong\mu_{2}(\mathcal{O}_{S})$ when $S$ has no points of residue field characteristic two, 
and let $\rho$ denote the class of $-1$ in $h^{1,1}\cong \mathcal{O}^{\times}_{S}/2$.
Recall from Lemma \ref{lem:slices-unit-kq-2} that $\partial^{a_{2q}}_{2}$ denotes the unique nontrivial map from $\Sigma^{4q-1,2q}\M \Lambda/a_{2q}$ to $\Sigma^{4q,2q}\M \Lambda/2$.
We refer to Theorem \ref{theorem:steenrod-algebra} for properties of the naturally induced maps $\inc^2_{a_{2q}}\colon\Sigma^{4q,2q}\M \Lambda/2 \to \Sigma^{4q,2q}\M \Lambda/a_{2q}$ 
and $\partial^2_{a_{2q}}\colon\Sigma^{4q-1,2q}\M \Lambda/2\to \Sigma^{4q,2q}\M \Lambda/a_{2q}$. 
Recall the slices of $\unit_{\Lambda}$ from Theorem \ref{theorem:sphereslices} and Corollary \ref{corollary:small-slices}. 

\begin{lemma}
\label{lem:first-diff-unit-1}
For $0\leq q\leq 2$ the slice differential $\dd^{\unit_{\Lambda}}_{1}(q)\colon \s_{q}(\unit_{\Lambda})\to \Sigma^{1,0}\s_{q+1}(\unit_{\Lambda})$ is given by:
\begin{align*}
\dd^{\unit_{\Lambda}}_{1}(0)
& 
= \Sq^2 \mathrm{pr}\colon \M \Lambda\to \M \Lambda/2\to \Sigma^{2,1}  \M \Lambda/2 \\
\dd^{\unit_{\Lambda}}_{1}(1)
& 
= 
\begin{pmatrix} 
\Sq^2 \\ 
\inc^{2}_{12} \Sq^2\Sq^1
\end{pmatrix}
\colon 
\Sigma^{1,1}  \M \Lambda/2\to \Sigma^{3,2}  \M \Lambda/2 \vee \Sigma^{4,2}  \M \Lambda/12 \\
\dd^{\unit_{\Lambda}}_{1}(2) 
& 
= 
\begin{pmatrix} 
\Sq^2 &   \tau \partial^{12}_{2} \\ 
\Sq^3\Sq^1   &  \Sq^2 \partial^{12}_{2} 
\end{pmatrix}
\colon 
\Sigma^{2,2}  \M \Lambda/2 \vee \Sigma^{3,2}  \M \Lambda/12 
\to
\Sigma^{4,3}  \M \Lambda/2 \vee \Sigma^{6,3}  \M \Lambda/2 \\
\end{align*}
\vspace{-0.4in}

\noindent 
For $q\geq 3$,
$\dd^{\unit_{\Lambda}}_{1}(q)$ restricts to the direct summand $\Sigma^{q,q}\M \Lambda/2\vee\Sigma^{q+2,q}\M \Lambda/2$ of $\s_{q}(\unit_{\Lambda})$ by 
\begin{align*}
\begin{pmatrix} 
\Sq^2 & \tau  \\ 
\Sq^3\Sq^1 & \Sq^2+\rho\Sq^1  
\end{pmatrix}
\colon
\Sigma^{q,q}\M \Lambda/2\vee\Sigma^{q+2,q}\M \Lambda/2\to\Sigma^{q+2,q+1}  \M \Lambda/2 \vee  \Sigma^{q+4,q+1} \M \Lambda/2.
\end{align*}
Here $\Sigma^{q+2,q}\M \Lambda/2$ is generated by $\alpha_{1}^{q-3}\alpha_{3}\in\Ext^{q-2,2q}_{\BP_{\ast}\BP}(\BP_{\ast},\BP_{\ast})$.
Moreover, 
$\dd^{\unit_{\Lambda}}_{1}(q)$ restricts as follows on top direct summands of $\s_{\ast}(\unit_{\Lambda})$:
\begin{align*}
\inc^2_{a_{2q}}\Sq^2\Sq^1
& 
\colon \Sigma^{4q-3,2q-1}   \M \Lambda/2
\to 
\Sigma^{4q,2q}  \M \Lambda/{a_{2q}} \\
\Sq^2 \partial^{a_{2q}}_{2} 
& 
\colon \Sigma^{4q-1,2q}  \M \Lambda/{a_{2q}} 
\to 
\Sigma^{4q+2,2q+1}   \M \Lambda/2\\
\tau \partial^{a_{2q}}_{2} 
& 
\colon \Sigma^{4q-1,2q}  \M \Lambda/{a_{2q}} 
\to 
\Sigma^{4q,2q+1} \M \Lambda/2 
& 
q \mathrm{\ odd} \\
0 
& 
\colon \Sigma^{4q-1,2q}  \M \Lambda/{a_{2q}} 
\to 
\Sigma^{4q,2q+1} \M \Lambda/2 
& 
q \mathrm{\ even} 
\end{align*}
Here $\Sigma^{4q-3,2q-1}\M \Lambda/2$ and $\Sigma^{4q-1,2q}\M \Lambda/{a_{2q}}$ are generated by $\alpha_{2q-1}$ and $\alpha_{2q/n}$, 
respectively.
\end{lemma}

\begin{proof}
The unit map $\unit_{\Lambda} \to \KGL_{\Lambda}$ and mod-$2$ reduction map $\KGL_{\Lambda}\to \KGL_{\Lambda}/2$ for the algebraic $K$-theory spectrum induce a commutative diagram:
\[ 
\xymatrix{
\s_0(\unit_{\Lambda}) \ar[r]^-{\iso} \ar[d]_-{\dd^{\unit_{\Lambda}}_{1}(0)} 
& \s_0(\KGL_{\Lambda}) \ar[r]^-{\pr} \ar[d]_-{\dd^{\KGL_{\Lambda}}_{1}(0)} 
& \s_0(\KGL_{\Lambda}/2) \ar[d]_-{\dd^{\KGL_{\Lambda}/2}_{1}(0)} \\ 
\Sigma^{1,0} \s_1(\unit_{\Lambda}) \ar[r] 
& \Sigma^{1,0}\s_1(\KGL_{\Lambda}) \ar[r]^-{\pr} 
& \Sigma^{1,0}\s_1(\KGL_{\Lambda}/2)}
\]
The map $\s_0(\unit_{\Lambda})\to \s_0(\KGL_{\Lambda})$ is an isomorphism 
by \cite{levine.coniveau} (see also \cite[Theorem 4.1]{roendigs-oestvaer.hermitian}).
Over $\Spec(\ZZ[\tfrac{1}{2}])$, 
we have $\dd^{\KGL_{\Lambda}/2}_{1}=\Sq^2\Sq^1+\Sq^1\Sq^2$ \cite[Lemma 5.1]{roendigs-oestvaer.hermitian}.
Hence the map $\s_0(\unit_{\Lambda})\to \Sigma^{1,0}\s_1(\KGL_{\Lambda}/2)$ equals $\Sq^1\Sq^2 \pr$ (recall $\Sq^1\pr=0$ by \cite[Lemma A.3]{roendigs-oestvaer.hermitian}).
It follows that $\dd^{\unit_{\Lambda}}_{1}(0)\colon \s_0(\unit_{\Lambda})\to \Sigma^{1,0}\s_1(\unit_{\Lambda})$ and $\s_1(\unit_{\Lambda})\to\s_1(\KGL_{\Lambda}/2)$ are nontrivial. 
This implies $\s_1(\unit_{\Lambda})\to \s_1(\KGL_{\Lambda}/2)$ is given by $\Sq^1$, 
and $\dd^{\unit_{\Lambda}}_{1}(0)=\Sq^2 \pr$ (see e.g., \cite[Lemmas A.1, A.4]{roendigs-oestvaer.hermitian}).

By \cite[Lemma 2.1]{roendigs-oestvaer.hermitian} the Hopf map $\eta\colon \Sigma^{1,1}\unit_{\Lambda}\to \unit_{\Lambda}$ induces a commutative diagram:
\begin{equation}
\label{eq:hopf-diff-unit} 
\xymatrix{
\Sigma^{1,1}\s_{q-1}(\unit_{\Lambda})\ar[r]^-{\iso} \ar[d]_-{\Sigma^{1,1}\dd^{\unit_{\Lambda}}_{1}(q-1)} 
& \s_q(\Sigma^{1,1}\unit_{\Lambda}) \ar[rr]^-{\s_q(\eta)} \ar[d]_-{\dd^{\Sigma^{1,1}\unit_{\Lambda}}_{1}(q)} 
&& \s_{q}(\unit_{\Lambda}) \ar[d]^-{\dd^{\unit_{\Lambda}}_{1}(q)} \\ 
\Sigma^{2,1}\s_{q}(\unit_{\Lambda}) \ar[r]^-{\iso} 
& \Sigma^{1,0}\s_{q+1}(\Sigma^{1,1}\unit_{\Lambda}) \ar[rr]^-{\Sigma^{1,0}\s_{q+1}(\eta)} 
&& \Sigma^{1,0}\s_{q+1}(\unit_{\Lambda})}
\end{equation}
By Lemma~\ref{lem:slices-unit-hopf}, 
$\s_q(\eta)$ and $\Sigma^{1,0}\s_{q+1}(\eta)$ in \eqref{eq:hopf-diff-unit} are known on the summands generated by $\alpha_{1}^{q}$ and $\alpha_{1}^{q-1}\alpha_{3}$.
For $q\geq 1$,
\eqref{eq:hopf-diff-unit} and the motivic Steenrod algebra (see e.g., \cite[Lemma A.2]{roendigs-oestvaer.hermitian}) show there exists an element $\phi\in h^{1,1}$ 
--- independent of $q$ by the comparison with $\KQ$ given in Lemma~\ref{lem:slices-unit-kq-1} --- 
such that $\dd^{\unit_{\Lambda}}_{1}(q)=\Sq^2+\phi\Sq^1$ on the direct summand $\Sigma^{q,q}\M \Lambda/2$ of $\s_q(\unit_{\Lambda})$ generated by $\alpha_{1}^{q}$. 
The Adem relations $\Sq^2\Sq^2=\tau\Sq^1\Sq^2\Sq^1$ and $\Sq^1\Sq^1=0$ (see e.g., \cite[Theorem 5.1]{hko.positivecharacteristic}) imply
\[ 
(\Sq^2+\phi\Sq^1)^2
=
\tau \Sq^1\Sq^2\Sq^1+\phi(\Sq^1\Sq^2+\Sq^2\Sq^1).
\]
Since differentials square to zero, 
\cite[Lemma A.2]{roendigs-oestvaer.hermitian} implies that $\dd^{\unit_{\Lambda}}_{1}(q)$ maps $\Sigma^{q,q}\M\Lambda/2\{\alpha_{1}^{q}\}$ by $\Sq^3\Sq^1$ to 
$\Sigma^{q+4,q+1}\M\Lambda/2\{\alpha_{1}^{q-2}\alpha_{3}\}$ for $q\geq 4$.
A tedious calculation for $q>4$ using Adem relations and motivic cohomology operations of weight one yields 
\[  
\dd^{\unit_{\Lambda}}_{1}(q)_{\vert \Sigma^{q,q}\M \Lambda/2\vee\Sigma^{q+2,q}\M \Lambda/2}  
= 
\begin{pmatrix} 
\Sq^2 & \tau   \\ 
\Sq^3\Sq^1 & \Sq^2+\rho\Sq^1  
\end{pmatrix}. \]

Theorem~\ref{theorem:steenrod-algebra} shows $\dd^{\unit_{\Lambda}}_{1}(1)\colon \s_1(\unit_{\Lambda})\to \Sigma^{1,0}\s_2(\unit_{\Lambda})$ is given by 
\[ 
\begin{pmatrix} 
\Sq^2 \\ 
b_2 \inc^{2}_{12} \Sq^2\Sq^1 + c_2\partial^{2}_{12}  \Sq^2
\end{pmatrix}
\colon 
\Sigma^{1,1}  \M \Lambda/2 \rightarrow \Sigma^{3,2}  \M \Lambda/2\vee \Sigma^{4,2}  \M \Lambda/12, 
\]
where $b_2,c_2\in h^{0,0}$, 
and moreover that $\dd^{\unit_{\Lambda}}_{1}(2)$ is given by 
\[  
\begin{pmatrix} 
\Sq^2 &  x\tau \partial^{12}_{2} +\phi\pr^{12}_{2}\\ 
\Sq^3\Sq^1 & y\Sq^2\partial^{12}_{2} +z\Sq^1\Sq^2\pr^{12}_{2}
\end{pmatrix}
\colon 
\Sigma^{2,2}  \M \Lambda/2 \vee \Sigma^{3,2}  \M \Lambda/12 
\rightarrow
\Sigma^{4,3}  \M \Lambda/2 \vee \Sigma^{6,3}  \M \Lambda/2,  
\]
where $x,y,z\in h^{0,0}$ and $\phi\in h^{1,1}$.
(The first column is obtained as in the case $q>4$.)
In Example \ref{ex:mult-slices-sphere-p2} we use the multiplicative structure on the slices of $\unit_{\Lambda}$ to conclude $c_2=0$.
From $\dd^{\unit_{\Lambda}}_{1}(2)\dd^{\unit_{\Lambda}}_{1}(1)=0$ we obtain $b_2=y=1$.
Likewise, 
$\dd^{\unit_{\Lambda}}_{1}(3)\dd^{\unit_{\Lambda}}_{1}(2)=0$ where $\dd^{\unit_{\Lambda}}_{1}(3)$ maps by $\Sq^2+\tau$ to $\Sigma^{4,3}\M \Lambda/2$, 
so we obtain $\phi=z=0$ and $x=1$. 
  
To determine $\dd^{\unit_{\Lambda}}_{1}(3)$ and $\dd^{\unit_{\Lambda}}_{1}(4)$ the calculation for $q>4$ shows it remains to consider 
$\Sigma^{6,4}\M\Lambda/2\{\alpha_1\alpha_3\} \vee \Sigma^{6,4}\M \Lambda/2\{\beta_{2/2}\}$, 
cf.~Corollary \ref{corollary:small-slices} and \cite[Theorem 5.3.7, Table A3.3]{ravenel.green}. 
Since $\nu^2$ is detected to $\beta_{2/2}$ (see \cite[Table A.3.3]{ravenel.green}), 
Theorem \ref{theorem:mult-adams-resol} implies 
$
\s_4(\nu^2)\colon 
\s_4(\Sigma^{6,4}\unit_{\Lambda}) 
\to
\s_4(\unit_{\Lambda})
$ 
maps to $\Sigma^{6,4}\M \Lambda/2\{\beta_{2/2}\}$ via the coefficient reduction $\Lambda\to\Lambda/2$.
Here $\nu\colon \Sigma^{3,2}\unit_{\Lambda}\to \unit_{\Lambda}$ is the second motivic Hopf map, whose homotopy cofiber is the quaternionic projective plane.
Now $\Sigma^{6,4}\M \Lambda/2\{\beta_{2/2}\}$ maps trivially to $\s_4(\KQ_{\Lambda})$ under the unit map because $\nu\smash\KQ_{\Lambda}$ is trivial.
This follows from the quaternionic projective bundle theorem \cite[Theorem 9.2]{paninwalterBO},
or
$\pi_{3,2}\KQ
=GW^{[2]}_{-1}(S)
=KSp_{-1}(S)=0$,
see e.g., \cite[Proposition 6.3]{schlichting}.
Lemma \ref{lem:slices-unit-kq-2} combined with the computation of $\dd^{\KQ_{\Lambda}}_{1}$ in \cite[Theorem 5.5]{roendigs-oestvaer.hermitian}
identify $\dd^{\unit_{\Lambda}}_{1}(3)$ and $\dd^{\unit_{\Lambda}}_{1}(4)$ entering and exiting $\Sigma^{6,4}\M \Lambda/2\{\alpha_1\alpha_3\}$,
and also $\dd^{\unit_{\Lambda}}_{1}(4)$ exiting $\Sigma^{6,4}\M \Lambda/2\{\beta_{2/2}\}$.

Also the last four equations can be obtained by comparing with $\KQ_{\Lambda}$ using Lemma~\ref{lem:slices-unit-kq-2}. More precisely,
this comparison directly implies the last three equations, and also
that the differential restricts to a map 
\[ \inc^2_{a_{2q}}\Sq^2\Sq^1 +c_{2q}\Sq^2\partial^{2}_{a_{2q}}
\colon \Sigma^{4q-3,2q-1}   \M \Lambda/2
\to 
\Sigma^{4q,2q}  \M \Lambda/{a_{2q}} \]
for some $c_{2q}\in h^{0,0}$. 
Example~\ref{ex:mult-slices-sphere-p2} shows $c_{2q}=0$.
For degree reasons, 
any differential exiting a direct summand of $\s_q(\unit_{\Lambda})$ of simplicial degree at least $q+3$ maps trivially to the direct summand 
$\Sigma^{q+2,q+1}\M \Lambda/2$ of $\Sigma^{1,0}\s_{q+1}(\unit_{\Lambda})$. 
Since multiplication by $\tau\colon \M \Lambda/2\to \Sigma^{0,1}\M \Lambda/2$ is injective on the motivic Steenrod algebra, 
any differential exiting a direct summand of $\s_q(\unit_{\Lambda})$ of simplicial degree at least $q+3$ maps trivially to the direct summand 
$\Sigma^{q+4,q+1}\M \Lambda/2$ of $\Sigma^{1,0}\s_{q+1}(\unit_{\Lambda})$. 
\end{proof}

The motivic cohomology operation $\tau$ appears also in many other first slice differentials gotten from $\alpha_{1}$-towers in the Adams-Novikov spectral sequence.
\begin{lemma}
\label{lem:first-diff-unit-2}
At the prime $2$, 
for $i\geq 3$ and $q\geq 1$, 
restricting $\dd^{\unit_{\Lambda}}_{1}(4q+i-3)$ to the direct summand generated by $\alpha_{1}^{i-3}\overline{\alpha}_{4q+2/n}$ yields 
\[
\tau\pr
\colon
\Sigma^{8q+3,4q+2}\M\Lambda/a_{4q+2}
\to
\Sigma^{8q+3,4q+3}\M\Lambda/2 
\]
for $i=3$, 
and 
\[
\tau
\colon
\Sigma^{8q+i,4q+i-1}\M\Lambda/2
\to
\Sigma^{8q+i,4q+i}\M\Lambda/2
\] 
for $i>3$.
Moreover, 
for $i\geq 3$ and $q\geq 1$,
restricting $\dd^{\unit_{\Lambda}}_{1}(4q+i-4)$ to the direct summand generated by $\alpha_{1}^{i-3}\alpha_{4q-1}$ yields
\[
\tau
\colon
\Sigma^{8q+i-6,4q+i-4}\M\Lambda/2
\to
\Sigma^{8q+i-6,4q+i-3}\M\Lambda/2.
\] 
\end{lemma}
\begin{proof}
The Adams-Novikov $d^{3}$-differential maps $\alpha_{1}^{i-3}\overline{\alpha}_{4q+2/n}$ to $\alpha_{1}^{i}\overline{\alpha}_{4q/n}$, 
and $\alpha_{1}^{q-i}\alpha_{4q-1}$ to $\alpha_{1}^{i}\alpha_{4q-3}$ \cite{novikov}, \cite[Property~2.3, Table 2]{Zahler}.
Topological realization,
see e.g., \cite[Theorem 1]{levine.an}, 
shows the corresponding restrictions of $\dd^{\unit_{\Lambda}}_{1}(4q+i-3)$ and $\dd^{\unit_{\Lambda}}_{1}(4q+i-4)$ are nontrivial over fields of characteristic zero. 
The general case follows by base change from  $\ZZ[\tfrac{1}{2}]$.
\end{proof}

By taking into account the primes $3$ and $7$ we find that $\dd^{\unit_{\Lambda}}_{1}(6)$ restricts to 
\[
\tau \pr\colon \Sigma^{11,6}\M \Lambda/504\{\alpha_{6/3}\} \to \Sigma^{11,7}\M \Lambda/2\{\alpha_{1}^{3}\alpha_{4/4}\}.
\] 

We summarize these computations in Figure \ref{figure:slice-diff}.
An entry represents a shifted motivic Eilenberg-MacLane spectrum with coefficients determined by $\Ext^{p,2q}_{\MU_{\ast}\MU}(\MU_{\ast},\MU_{\ast})$ as in 
Theorem \ref{theorem:sphereslices}.  
The slices are labeled successively along the vertical axis. 
Horizontally, 
each direct summand of a fixed slice is labeled according to the difference between the simplicial suspension of the direct summand and the slice degree.
The $d^{1}$-differentials are labeled by colors corresponding to elements of the motivic Steenrod algebra, and ordered by simplicial degree.
An open square refers to $\Lambda$-coefficients and solid dots to $\Lambda/2$.
The asterisk in the second slice indicates $\Lambda/12$-coefficients, 
and similarly for $\Lambda/2\oplus \Lambda/2$, $\Lambda/240$, $\Lambda/6$, and $\Lambda/504$.

\begin{figure}[!ht]
\label{figure:slice-diff}
  \caption{The first slice differential for $\unit_{\Lambda}$.}
\begin{center}
  \pgfsetshortenend{3pt}
  \pgfsetshortenstart{3pt}
  \begin{tikzpicture}[scale=1.2,line width=1pt]
    \draw[help lines] (-4.3,-2.3) grid (3.3,5.3);
    {\draw[fill]     
      (-4.7,-2) circle (0pt) node[left=-1pt] {$0$}
      ;}
    \node[rectangle, draw] at (-4,-2) {};
    {\draw[fill]     
      (-4,-1) circle (2pt) 
      (-4.7,-1) circle (0pt) node[left=-1pt] {$1$}
      ;}
    {\draw[fill]     
      (-4,0) circle (2pt) 
      (-4.7,0) circle (0pt) node[left=-1pt] {$2$}
      ;}
    \node[star, star points=12, draw, fill] at (-3,0) {};
    {\draw[fill]     
      (-4,1) circle (2pt) (-2,1) circle (2pt) 
      (-4.7,1) circle (0pt) node[left=-1pt] {$3$}
      ;}
    {\draw[fill]     
      (-4,2) circle (2pt) (-2,2) circle (2pt) (-1.8,2) circle (2pt) 
      (-4.7,2) circle (0pt) node[left=-1pt] {$4$}
      ;}
    \node[star, star points=15, draw, fill] at (-1,2) {};
    {\draw[fill]     
      (-4,3) circle (2pt) (-2,3) circle (2pt) (-1,3) circle (2pt) (-0.8,3) circle (2pt) (0,3) circle (2pt) 
      (-4.7,3) circle (0pt) node[left=-1pt] {$5$}
      ;}
    {\draw[fill]     
      (-4,4) circle (2pt) (-2,4) circle (2pt) (-1,4) circle (2pt) (-0.8,4) circle (2pt) 
      (-4.7,4) circle (0pt) node[left=-1pt] {$6$}
      ;}
    \node[star, star points=21, draw, fill] at (1,4) {};
    \node[regular polygon, regular polygon sides=6, draw, fill] at (0,4)  {};
    {\draw[fill]     
      (-4,5) circle (2pt) (-2,5) circle (2pt) (-1,5) circle (2pt) (0,5) circle (2pt) (1,5) circle (2pt) (2,5) circle (2pt) 
      (-4.7,5) circle (0pt) node[left=-1pt] {$7$}
      ;}
    {\draw[fill]     
      (-6.1,2.4) circle (0pt) node[left=-1pt,rotate=90] {\sc \scriptsize Slice degree (weight)}
      ;}
    {\draw[fill]     
    (-4,-2.7) circle (0pt) node[below=-1pt] {$0$}
    (-3,-2.7) circle (0pt) node[below=-1pt] {$1$}
    (-2,-2.7) circle (0pt) node[below=-1pt] {$2$}
    (-1,-2.7) circle (0pt) node[below=-1pt] {$3$}
    (0,-2.7) circle (0pt) node[below=-1pt] {$4$} 
    (1,-2.7) circle (0pt) node[below=-1pt] {$5$}
    (2,-2.7) circle (0pt) node[below=-1pt] {$6$}
    (4.3,-2.5) circle (0pt) node[below=-1pt] {\sc \scriptsize Simplicial degree}
    (4.3,-2.8) circle (0pt) node[below=-1pt] {\sc \scriptsize minus slice degree}
    ;}
  
  {\draw[sq2prcolor,->] 
    (-4,-2) -- (-4,-1)
    ;}
  {\draw[sq2prcolor,->] 
    (0,4) -- (0,5)
    ;}

  {\draw[incsq2color,->] 
    (0,3) -- (0,4)
    ;}
  
  {\draw[incsq2sq1color,->] 
    (-4,-1) -- (-3,0)
    ;}
  {\draw[incsq2sq1color,->] 
    (-2,1) -- (-1,2)
    ;}
  {\draw[incsq2sq1color,->] 
    (0,3) -- (1,4)
    ;}
  {\draw[incsq2sq1color,->] 
    (2,5) -- (2.5,5.5)
    ;}

  {\draw[sq2partialcolor,->] 
    (-3,0) -- (-2,1)
    ;}
  {\draw[sq2partialcolor,->] 
    (-1,2) -- (0,3)
    ;}
  {\draw[sq2partialcolor,->] 
    (1,4) -- (2,5)
    ;}

  {\draw[taupartialcolor,->] 
    (-3,0) -- (-4,1)
    ;}

  {\draw[sq2color,->] 
    (-4,0) -- (-4,1)
    ;}
  {\draw[sq2color,->] 
    (-4,-1) -- (-4,0)
    ;}
  {\draw[sq2color,->] 
    (-4,1) -- (-4,2)
    ;}
  {\draw[sq2color,->] 
    (-4,2) -- (-4,3)
    ;}
  {\draw[sq2color,->] 
    (-4,3) -- (-4,4)
    ;}
  {\draw[sq2color,->] 
    (-4,4) -- (-4,5)
    ;}
  {\draw[sq2color,->] 
    (-4,5) -- (-4,5.5)
    ;}
  {\draw[sq2color,->] 
    (0,5) -- (0,5.5)
    ;}
  
  {\draw[sq3sq1color,->] 
    (-4,0) -- (-2,1)
    ;}
  {\draw[sq3sq1color,->] 
    (-4,1) -- (-2,2)
    ;}
  {\draw[sq3sq1color,->] 
    (-4,2) -- (-2,3)
    ;}
  {\draw[sq3sq1color,->] 
    (-4,3) -- (-2,4)
    ;}
  {\draw[sq3sq1color,->] 
    (-4,4) -- (-2,5)
    ;}
  {\draw[sq3sq1color,->] 
    (-4,5) -- (-3,5.5)
    ;}
  {\draw[sq3sq1color,->] 
    (-2,2) -- (0,3)
    ;}
  {\draw[sq3sq1color,->] 
    (-2,3) -- (0,4)
    ;}
  {\draw[sq3sq1color,->] 
    (-2,4) -- (0,5)
    ;}
  {\draw[sq3sq1color,->] 
    (-2,5) -- (-1,5.5)
    ;}
  {\draw[sq3sq1color,->] 
    (0,4) -- (2,5)
    ;}
  {\draw[sq3sq1color,->] 
    (0,5) -- (1,5.5)
    ;}
  
  {\draw[taucolor,->] 
    (-2,1) -- (-4,2)
    ;}
  {\draw[taucolor,->] 
    (-2,2) -- (-4,3)
    ;}
  {\draw[taucolor,->] 
    (-2,3) -- (-4,4)
    ;}
  {\draw[taucolor,->] 
    (-2,4) -- (-4,5)
    ;}
  {\draw[taucolor,->] 
    (-2,5) -- (-3,5.5)
    ;}
  {\draw[tauprcolor,->] 
    (1,4) -- (-1,5)
    ;}
  {\draw[taucolor,->] 
    (1,5) -- (0,5.5)
    ;}
  {\draw[taucolor,->] 
    (2,5) -- (1,5.5)
    ;}
  
  {\draw[sq2rhosq1color,->] 
    (-2,1) -- (-2,2)
    ;}
  {\draw[sq2rhosq1color,->] 
    (-2,2) -- (-2,3)
    ;}
  {\draw[sq2rhosq1color,->] 
    (-2,3) -- (-2,4)
    ;}
  {\draw[sq2rhosq1color,->] 
    (-2,4) -- (-2,5)
    ;}
  {\draw[sq2rhosq1color,->] 
    (-2,5) -- (-2,5.5)
    ;}
  {\draw[sq2rhosq1color,->] 
    (2,5) -- (2,5.5)
    ;}
  {\draw[taupartialcolor,->] 
    (1,4) -- (0,5)
    ;}

  {\draw[sq3sq1color] 
    (3.4,5.4) circle (0pt) node[right] {$\Sq^3\Sq^1$}
    ;}
  {\draw[incsq2sq1color,->] 
    (3.4,4.6) circle (0pt) node[right] {$\inc^2_{?} \Sq^2\Sq^1$}
    ;}
  {\draw[sq2partialcolor,->] 
    (3.4,3.8) circle (0pt) node[right] {$\Sq^2 \partial^{?}_{2}$}
    ;}
  {\draw[incsq2color,->] 
    (3.4,3) circle (0pt) node[right] {$\inc^{2}_{6} \Sq^2$}
    ;}
  {\draw[sq2prcolor] 
    (3.4,2.2) circle (0pt) node[right] {$\Sq^2 \pr^{?}_{2}$}
    ;}
  {\draw[sq2rhosq1color] 
    (3.4,1.4) circle (0pt) node[right] {$\Sq^2+\rho\Sq^1$}
    ;}
  {\draw[sq2color] 
    (3.4,0.6) circle (0pt) node[right] {$\Sq^2$}
    ;}
  {\draw[taupartialcolor] 
    (3.4,-0.2) circle (0pt) node[right] {$\tau \partial^{?}_{2}$}
    ;}
  {\draw[tauprcolor] 
    (3.4,-1) circle (0pt) node[right] {$\tau \pr$}
    ;}
  {\draw[taucolor] 
    (3.4,-1.8) circle (0pt) node[right] {$\tau$}
    ;} 
\end{tikzpicture}
\end{center}
\end{figure}

\begin{example}
\label{ex:mult-slices-sphere-p2}
Write $\mathbf{c}_1\in H^{2,1}(\PP^2) = [\Sigma^{-2,-1}  \PP^2_+,\s_0(\unit_{\Lambda})]$ for the first Chern class of the tautological line bundle on $\PP^2$.
Then $\Sq^2\pr^{\infty}_{2}(\mathbf{c}_1) = \pr^{\infty}_{2}(\mathbf{c}_1)^2\in h^{4,2}(\PP^2)$ is nonzero.
Lemma \ref{lem:mult-slices-sphere} computes the cup-product of $\tau\in h^{0,1}(\PP^2)=[\Sigma^{1,0}\PP^2_+,\s_1(\unit_{\Lambda})]$ with $\mathbf{c}_1$, 
i.e., 
\[ 
\tau\cdot\mathbf{c}_1 
= 
\tau\pr^{\infty}_{2}(\mathbf{c}_1)\in  h^{2,2}(\PP^2) 
= 
[\Sigma^{-1,-1}  \PP^2_+,\s_1(\unit_{\Lambda})].
\]
Applying the first slice differential yields 
\begin{align*}
d^{\unit_{\Lambda}}_{1}(\tau\cdot \mathbf{c}_1) 
& = \bigl(\Sq^2(\tau\pr^{\infty}_{2}(\mathbf{c}_1)),\inc^{2}_{12}\Sq^2\Sq^1(\tau\pr^{\infty}_{2}(\mathbf{c}_1))+c_2\partial^{2}_{12}\Sq^2(\tau\pr^{\infty}_{2}(\mathbf{c}_1))\bigr) \\
& = \bigl(\tau\pr^{\infty}_{2}(\mathbf{c}_1)^2,\inc^{2}_{12}\rho\pr^{\infty}_{2}(\mathbf{c}_1)^2+c_2\partial^{2}_{12}\tau\pr^{\infty}_{2}(\mathbf{c}_1)^2\bigr) \\
& = \bigl(\tau\pr^{\infty}_{2}(\mathbf{c}_1)^2,c_2\partial^{2}_{12}\tau\pr^{\infty}_{2}(\mathbf{c}_1)^2\bigr). 
\end{align*}
If $0\neq \rho\in h^{1,1}(S)$,  
then $\partial^{2}_{12}\tau\pr^{\infty}_{2}(\mathbf{c}_1)^2$ is nonzero because $\Sq^1 = \pr^{12}_{2} \partial^{2}_{12}$ and $\Sq^1\tau\pr^{\infty}_{2}(\mathbf{c}_1)^2=\rho\pr^{\infty}_{2}(\mathbf{c}_1)^2\in h^{5,3}(\PP^2)$.
The change of coefficients long exact sequence
\[ 
\dotsm \to h^{4,3}_6(\PP^2) \xrightarrow{\Sq^1} h^{5,3}(\PP^2)\xrightarrow{\inc^{2}_{12}}h^{5,3}_{12}(\PP^2) \to \dotsm
\]
implies $\inc^{2}_{12}\rho\pr^{\infty}_{2}(\mathbf{c}_1)^2=0$. 
By the Leibniz rule in Proposition \ref{prop:leibniz} we get  
\begin{align*}
d^{\unit_{\Lambda}}_{1}(\tau\cdot \mathbf{c}_1) 
& = d^{\unit_{\Lambda}}_{1}(\tau)\cdot \mathbf{c}_1 +\tau \cdot d^{\unit_{\Lambda}}_{1}(\mathbf{c}_1) \\
& = 0\cdot \mathbf{c}_1 + \tau \cdot \Sq^2\pr^{\infty}_{2}(\mathbf{c}_1) \\
& = \tau\cdot \pr^{\infty}_{2}(\mathbf{c}_1)^2 \\ 
& = ( \tau \pr^{\infty}_{2}(\mathbf{c}_1)^2,\inc^{2}_{12}\tau\Sq^1(\pr^{\infty}_{2}(\mathbf{c}_1)^2)) \\
& = ( \tau \pr^{\infty}_{2}(\mathbf{c}_1)^2,0).
\end{align*} 
Lemma \ref{lem:mult-slices-sphere} determines $\s_1(\unit_{\Lambda})\smash_{\s_0(\unit_{\Lambda})} \s_1(\unit_{\Lambda})\to \s_2(\unit_{\Lambda})$ and Corollary~\ref{cor:mult-mot-coh} identifies 
the effect on $\bigl(\tau,\pr^{\infty}_{2}(\mathbf{c}_1)^2\bigr)$. 
Hence $c_2=0$ over $\ZZ[\tfrac{1}{2}]$, 
where $\rho\neq 0$, 
and in general by base change. 
More generally, 
one may use the Leibniz rule, $\mathbf{c}_1$ as above, and 
\[ 
\tau\in h^{0,1}(\PP^2)=[\Sigma^{4q-3,2q-2}\PP^2_+,\Sigma^{4q-3,2q-1}\M\Lambda/2]
\to [\Sigma^{4q-3,2q-2}\PP^2_+,\s_{2q-1}(\unit_{\Lambda})] 
\]
to conclude $c_{2q}=0$ precisely if the multiplication map $\Tot\{2q-1\}\otimes \Tot\{1\}\to \Tot\{2q\}$ to the Adams degree $4q$ part of the standard 
cosimplicial $\BP$-resolution at the prime $2$
induces the map $\inc^2_{a_{2q}}=\inc^2_{a_{2q}}+0\cdot \partial^{2}_{a_{2q}}$
to $H_1(\Tot\{2q\})=\ZZ/a_{2q}$.
The latter is verified in Lemma~\ref{lem:mult-oddslice-1slice}.
\end{example}

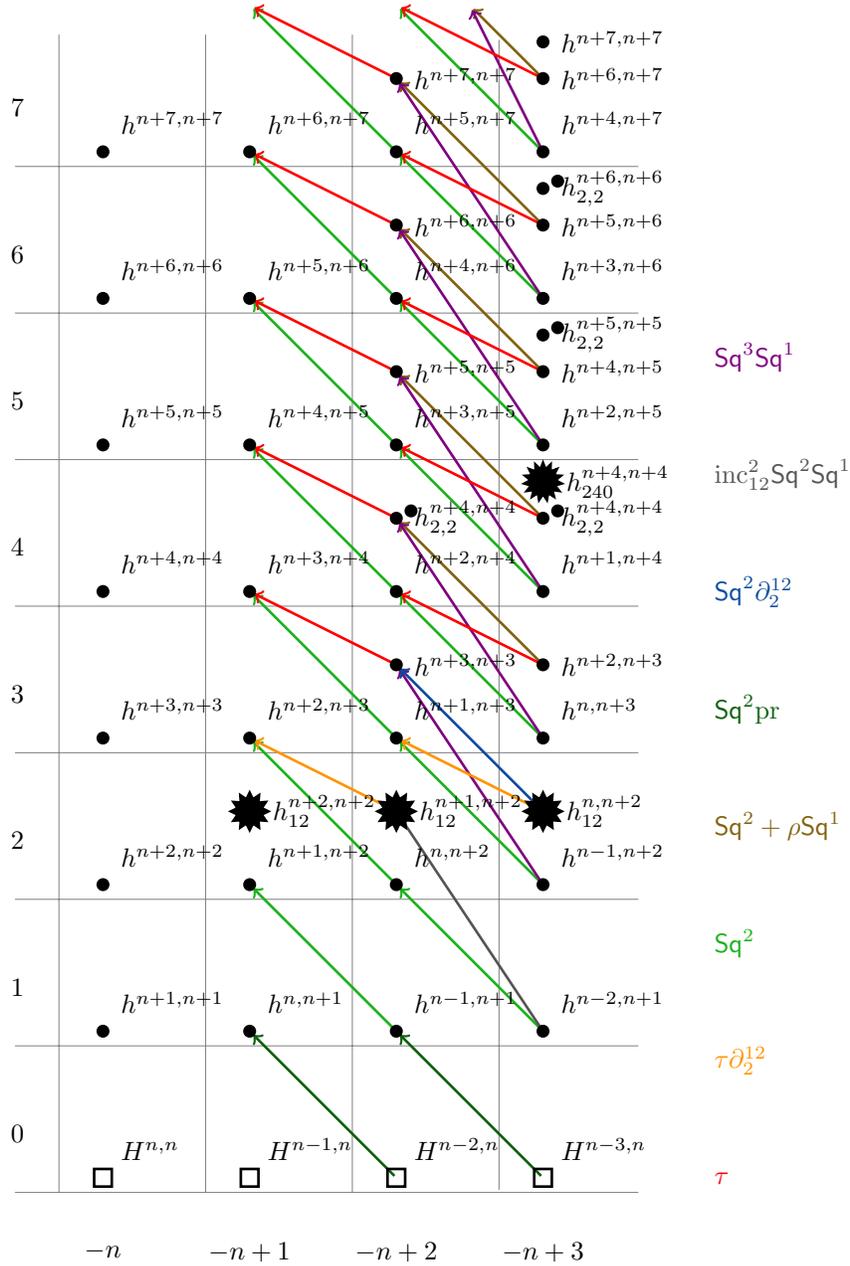
\begin{figure}[!ht]
\label{figure:E1-page}
\caption{$E^{1}$-page of the weight $-n$th slice spectral sequence for $\unit_{\Lambda}$.}
\begin{center}
  \pgfsetshortenend{2pt}
  \pgfsetshortenstart{1pt}
  \begin{tikzpicture}[scale=1.95,line width=1pt]
    \draw[help lines,shift={(-.3,-.1)}] (-2.3,0) grid (1.95,7.9);
    \foreach \i in {0,...,7} {\node[label=left:$\i$] at (-2.4,\i+.3) {};}
    {\node[label=below:$-n$] at (-2,-0.3) {};}
    {\node[label=below:$-n+1$] at (-1,-0.3) {};}
    {\node[label=below:$-n+2$] at (0,-0.3) {};}
    {\node[label=below:$-n+3$] at (1,-0.3) {};}

    \draw[->,sq2prcolor]
    (0,0) -- (-1,1);
    \draw[->,sq2color]
    (0,1) -- (-1,2);
    \draw[->,sq2color]
    (0,2) -- (-1,3);
    \draw[sq2color,->]
    (0,3) -- (-1,4);
    \draw[sq2color,->]
    (0,4) -- (-1,5);
    \draw[sq2color,->]
    (0,5) -- (-1,6);
    \draw[sq2color,->]
    (0,6) -- (-1,7);
    \draw[sq2color,->]
    (0,7) -- (-1,8);

    \draw[->,sq2prcolor]
    (1,0) -- (0,1);
    \draw[->,sq2color]
    (1,1) -- (0,2);
    \draw[->,sq2color]
    (1,2) -- (0,3);
    \draw[sq2color,->]
    (1,3) -- (0,4);
    \draw[sq2color,->]
    (1,4) -- (0,5);
    \draw[sq2color,->]
    (1,5) -- (0,6);
    \draw[sq2color,->]
    (1,6) -- (0,7);
    \draw[sq2color,->]
    (1,7) -- (0,8);

    \draw[sq2rhosq1color,->]
    (1,3.5) -- (0,4.5);
    \draw[sq2rhosq1color,->]
    (1,4.5) -- (0,5.5);
    \draw[sq2rhosq1color,->]
    (1,5.5) -- (0,6.5);
    \draw[sq2rhosq1color,->]
    (1,6.5) -- (0,7.5);
    \draw[sq2rhosq1color,->]
    (1,7.5) -- (0.5,8);

    \draw[sq3sq1color,->]
    (1,2) -- (0,3.5);
    \draw[sq3sq1color,->]
    (1,3) -- (0,4.5);
    \draw[sq3sq1color,->]
    (1,4) -- (0,5.5);
    \draw[sq3sq1color,->]
    (1,5) -- (0,6.5);
    \draw[sq3sq1color,->]
    (1,6) -- (0,7.5);
    \draw[sq3sq1color,->]
    (1,7) -- (0.5,8);
    
    \draw[taucolor,->]
    (0,3.5) -- (-1,4);
    \draw[taucolor,->]
    (0,4.5) -- (-1,5);
    \draw[taucolor,->]
    (0,5.5) -- (-1,6);
    \draw[taucolor,->]
    (0,6.5) -- (-1,7);
    \draw[taucolor,->]
    (0,7.5) -- (-1,8);

    \draw[taucolor,->]
    (1,3.5) -- (0,4);
    \draw[taucolor,->]
    (1,4.5) -- (0,5);
    \draw[taucolor,->]
    (1,5.5) -- (0,6);
    \draw[taucolor,->]
    (1,6.5) -- (0,7);
    \draw[taucolor,->]
    (1,7.5) -- (0,8);

    \draw[taupartialcolor,->]
    (0,2.5) -- (-1,3);
    \draw[taupartialcolor,->]
    (1,2.5) -- (0,3);
   
    \draw[sq2partialcolor,->]
    (1,2.5) -- (0,3.5);
 
    \draw[incsq2sq1color,->]
    (1,1) -- (0,2.5);
    
    \node at (-2,0) [shape=rectangle,draw] {};
    \node at (-2,0) [above right=3pt] {$H^{n,n}$};
    \node at (-1,0) [shape=rectangle,draw] {};
    \node at (-1,0) [above right=3pt] {$H^{n-1,n}$};
    \node at (0,0) [shape=rectangle,draw] {};
    \node at (0,0) [above right=3pt] {$H^{n-2,n}$};
    \node at (1,0) [shape=rectangle,draw] {};
    \node at (1,0) [above right=3pt] {$H^{n-3,n}$};
    
    \node[star, star points=12, draw, fill] at (-1,2.5)  {};
    \node at (-1,2.5) [right=5pt] {$h_{12}^{n+2,n+2}$};
    \node[star, star points=12, draw, fill] at (0,2.5)  {};
    \node at (0,2.5) [right=5pt] {$h_{12}^{n+1,n+2}$};
    \node[star, star points=12, draw, fill] at (1,2.5)  {};
    \node at (1,2.5) [right=5pt] {$h_{12}^{n,n+2}$};

    \node[star, star points=15, draw, fill] at (1,4.75)  {};
    \node at (1,4.75) [right=5pt] {$h_{240}^{n+4,n+4}$};

    {\draw[fill]     
      (-2,1) circle (1pt) node[above right=3pt] {{$h^{n+1,n+1}$}}
      (-2,2) circle (1pt) node[above right=3pt] {{$h^{n+2,n+2}$}}
      (-2,3) circle (1pt) node[above right=3pt] {{$h^{n+3,n+3}$}}
      (-2,4) circle (1pt) node[above right=3pt] {{$h^{n+4,n+4}$}}
      (-2,5) circle (1pt) node[above right=3pt] {{$h^{n+5,n+5}$}}
      (-2,6) circle (1pt) node[above right=3pt] {{$h^{n+6,n+6}$}}
      (-2,7) circle (1pt) node[above right=3pt] {{$h^{n+7,n+7}$}}
      ;}

    {\draw[fill]     
      (-1,1) circle (1pt) node[above right=3pt] {{$h^{n,n+1}$}}
      (-1,2) circle (1pt) node[above right=3pt] {{$h^{n+1,n+2}$}}
      (-1,3) circle (1pt) node[above right=3pt] {{$h^{n+2,n+3}$}}
      (-1,4) circle (1pt) node[above right=3pt] {{$h^{n+3,n+4}$}}
      (-1,5) circle (1pt) node[above right=3pt] {{$h^{n+4,n+5}$}}
      (-1,6) circle (1pt) node[above right=3pt] {{$h^{n+5,n+6}$}}
      (-1,7) circle (1pt) node[above right=3pt] {{$h^{n+6,n+7}$}}
      ;}

    {\draw[fill]      
      (0,1) circle (1pt) node[above right=3pt] {$h^{n-1,n+1}$}
      (0,2) circle (1pt) node[above right=3pt] {$h^{n,n+2}$}
      (0,3) circle (1pt) node[above right=3pt] {$h^{n+1,n+3}$}
      (0,3.5) circle (1pt) node[right=3pt] {$h^{n+3,n+3}$}
      (0,4) circle (1pt) node[above right=3pt] {$h^{n+2,n+4}$}
      (0,4.5) circle (1pt) node[right=3pt] {$h_{2,2}^{n+4,n+4}$}
      (0.1,4.55) circle (1pt) node {}
      (0,5) circle (1pt) node[above right=3pt] {$h^{n+3,n+5}$}
      (0,5.5) circle (1pt) node[right=3pt] {$h^{n+5,n+5}$}
      (0,6) circle (1pt) node[above right=3pt] {$h^{n+4,n+6}$}
      (0,6.5) circle (1pt) node[right=3pt] {$h^{n+6,n+6}$}
      (0,7) circle (1pt) node[above right=3pt] {$h^{n+5,n+7}$}
      (0,7.5) circle (1pt) node[right=3pt] {$h^{n+7,n+7}$}
      ;}

    {\draw[fill]      
      (1,1) circle (1pt) node[above right=3pt] {$h^{n-2,n+1}$}
      (1,2) circle (1pt) node[above right=3pt] {$h^{n-1,n+2}$}
      (1,3) circle (1pt) node[above right=3pt] {$h^{n,n+3}$}
      (1,3.5) circle (1pt) node[right=3pt] {$h^{n+2,n+3}$}
      (1,4) circle (1pt) node[above right=3pt] {$h^{n+1,n+4}$}
      (1,4.5) circle (1pt) node[right=3pt] {$h_{2,2}^{n+4,n+4}$}
      (1.1,4.55) circle (1pt) node {}
      (1,5) circle (1pt) node[above right=3pt] {$h^{n+2,n+5}$}
      (1,5.5) circle (1pt) node[right=3pt] {$h^{n+4,n+5}$}
      (1,5.75) circle (1pt) node[right=3pt] {$h_{2,2}^{n+5,n+5}$}
      (1.1,5.8) circle (1pt) node {}
      (1,6) circle (1pt) node[above right=3pt] {$h^{n+3,n+6}$}
      (1,6.5) circle (1pt) node[right=3pt] {$h^{n+5,n+6}$}
      (1,6.75) circle (1pt) node[right=3pt] {$h_{2,2}^{n+6,n+6}$}
      (1.1,6.8) circle (1pt) node {}
      (1,7) circle (1pt) node[above right=3pt] {$h^{n+4,n+7}$}
      (1,7.5) circle (1pt) node[right=3pt] {$h^{n+6,n+7}$}
      (1,7.75) circle (1pt) node[right=3pt] {$h^{n+7,n+7}$}
      ;}

     {\draw[sq3sq1color] 
        (2.1,5.6) circle (0pt) node[right] {$\Sq^3\Sq^1$}
        ;}
      {\draw[incsq2sq1color,->] 
        (2.1,4.8) circle (0pt) node[right] {$\inc^{2}_{12} \Sq^2\Sq^1$}
        ;}
      {\draw[sq2partialcolor,->] 
        (2.1,4) circle (0pt) node[right] {$\Sq^2 \partial^{12}_{2}$}
        ;}
      {\draw[sq2prcolor] 
        (2.1,3.2) circle (0pt) node[right] {$\Sq^2 \pr$}
        ;}
      {\draw[sq2rhosq1color] 
        (2.1,2.4) circle (0pt) node[right] {$\Sq^2+\rho\Sq^1$}
        ;}
      {\draw[sq2color] 
        (2.1,1.6) circle (0pt) node[right] {$\Sq^2$}
        ;}
      {\draw[taupartialcolor,->] 
        (2.1,0.8) circle (0pt) node[right] {$\tau \partial^{12}_{2}$}
        ;}
      {\draw[taucolor] 
        (2.1,0) circle (0pt) node[right] {$\tau$}
        ;}
\end{tikzpicture}
\end{center}
\end{figure}

Fix a compatible pair $(F,\Lambda)$, where $F$ is a field of characteristic not two. 
For $n\in\ZZ$, 
Figure 4.2 
shows the $E^{1}$-page $E^{1}_{p,q,-n}(\unit_{\Lambda}) = \pi_{p,-n}\s_q(\unit_{\Lambda})$ of the weight $-n$th slice spectral sequence for $\unit_{\Lambda}$. 
Note that $E^{1}_{p,q,-n}(\unit_{\Lambda})=0$ for $p<-n$ or $q<0$, and $H^{p,n}=0$ for $n<0$.
In each bidegree, 
the corresponding group is decomposed into direct summands according to Theorem \ref{theorem:sphereslices}.
We write $h^{\ast,\ast}_{2,2}$ for motivic cohomology with $(\Lambda/2\times\Lambda/2)$-coefficients.
Each differential is assigned a color corresponding to a motivic Steenrod operation.

\begin{lemma}
\label{lemma:trivialenteringfirstdifferentials}
For $n\in\ZZ$ every $d_1$-differential entering the $-n$th column in the $-n$th slice spectral sequence for $\unit_{\Lambda}$ is trivial.
The $d^{1}$-differential $d^{1}\colon h^{n-1,n+1}\to h^{n+2,n+2}_{12}$ is also trivial.
\end{lemma}
\begin{proof}
By Lemma \ref{lem:first-diff-unit-1} this amounts to the triviality of $\Sq^{2}\pr\colon H^{n-1,n}\to h^{n-1,n}\to h^{n+1,n+1}$,
$\tau\partial^{12}_{2}\colon h^{n+2,n+2}_{12}\to h^{n+3,n+3}$, 
$\Sq^{2}\colon h^{i,i+1}\to h^{i+2,i+2}$ for all $i\geq n$, 
and finally the composition $\inc^{2}_{12}\Sq^{2}\Sq^{1}\colon h^{n-1,n+1}\to h^{n,n+1} \to h^{n+2,n+2}\to h^{n+2,n+2}_{12}$. 
Note that $\tau\partial^{12}_{2}$ factors through the group $h^{n+3,n+2}=0$, 
while $\Sq^{2}$ is trivial in the given range \cite[Corollary 6.2]{roendigs-oestvaer.hermitian}. 
\end{proof}

\begin{remark}
\label{remark:pi10}
Lemma \ref{lemma:trivialenteringfirstdifferentials} implies $E^{2}_{0,q,0}(\unit_{\Lambda})=E^{1}_{0,q,0}(\unit_{\Lambda})$.
For degree reasons it follows that $E^{2}_{1,q,0}(\unit_{\Lambda})=E^{1}_{1,q,0}(\unit_{\Lambda})$, $q\leq 2$.
Each red 
$d^{1}$-differential entering a solid dot representing a mod-$2$ motivic cohomology group is surjective
(recall that $\tau\colon h^{p,q}\to h^{p,q+1}$ is an isomorphism for $0\leq p\leq q$ by Voevodsky's solution of the Milnor conjecture,
cf.~\cite[Lemma 6.1]{roendigs-oestvaer.hermitian}).
In Theorem \ref{theorem:higher-diff-unit-zero-line} we show $E^{\infty}_{1,1,0}(\unit_{\Lambda})\cong h^{0,1}$, $E^{\infty}_{1,2,0}(\unit_{\Lambda})\cong h^{1,2}\oplus h^{2,2}_{12}$, 
$E^{\infty}_{1,3,0}(\unit_{\Lambda})\cong h^{2,3}/\tau\partial^{12}_{2} h^{1,2}_{12}$, and $E^{\infty}_{1,q,0}(\unit_{\Lambda})=0$ for $q>3$.
In Theorem \ref{theorem:1line}, 
see also Corollary \ref{corollary:1linecorollary}, 
we prove Morel's $\pi_{1}$-conjecture by solving the hidden extensions between these groups.
\end{remark}

\subsection{Higher slice differentials}
\label{sec:beyond-first-diff}

Throughout we fix a compatible pair $(F,\Lambda)$, where $F$ is a field of $\Char(F)\neq 2$.

\begin{lemma}
\label{lem:unit-kt-zero-line}
The unit maps $\unit_{\Lambda}\to \KQ_{\Lambda}$ and $\unit_{\Lambda}\to \KT_{\Lambda}$ induce isomorphisms 
\[ 
E_{n,0,n}^2(\unit_{\Lambda}) 
\xrightarrow{\iso} 
E^2_{n,0,n}(\KQ_{\Lambda}) 
\text{ and } 
E_{n,q,n}^2(\unit_{\Lambda}) 
\xrightarrow{\iso} E^2_{n,q,n}(\KT_{\Lambda}) 
\]
for all $n\in \ZZ$, $q> 0$.
\end{lemma}  
\begin{proof}
For degree reasons only powers of $\alpha_1$ contribute to the $n$th column of the $n$th slice spectral sequence for $\unit_{\Lambda}$.
Lemma~\ref{lem:slices-unit-kq-1} implies the unit maps $\unit_{\Lambda}\to \KQ_{\Lambda}$ and $\unit_{\Lambda}\to \KT_{\Lambda}$ induce inclusions
\begin{align*}
& H^{-n,-n}=E^{1}_{n,0,n}(\unit_{\Lambda}) \hookrightarrow H^{-n,-n}\directsum h^{-n-2,-n}\directsum \dotsm =E^{1}_{n,0,n}(\KQ_{\Lambda}) 
& & 
n\leq 0 \\
& h^{q-n,q-n}=E^{1}_{n,q,n}(\unit_{\Lambda}) \hookrightarrow h^{q-n,q-n}\directsum h^{q-n-2,q-n}\directsum \dotsm =E^{1}_{n,q,n}(\KT_{\Lambda}) 
& & q>0 \mathrm{\ or\ } n > 0. 
\end{align*}
Lemma \ref{lemma:trivialenteringfirstdifferentials} shows the first slice differential for $\unit_{\Lambda}$ maps trivially to the $n$th column, 
so that $E_{n,q,n}^{2}(\unit_{\Lambda})=E_{n,q,n}^{1}(\unit_{\Lambda})$. 
The $E^2$-page $E^{2}_{n,0,n}(\KQ_{\Lambda})$ can be computed in low degrees using \cite[Theorem 5.5]{roendigs-oestvaer.hermitian}; 
for weight zero, see \cite[Section 7]{roendigs-oestvaer.hermitian}.
If $n>0$, 
the groups $E^1_{n,0,n}(\unit_{\Lambda})$ and $E^1_{n,0,n}(\KQ_{\Lambda})$ are trivial.
If $n\leq 0$, 
$E^{2}_{n,0,n}(\KQ_{\Lambda})$ is given by the homology of the complex:
\[\xymatrix{ 
h^{-n-4,-n-1}\directsum h^{-n-6,-n-1}\directsum\dotsm 
\ar[d]_-{d^{1}(\KQ_{\Lambda})} \\
H^{-n,-n}\directsum h^{-n-2,-n}\directsum\dotsm 
\ar[d]_-{d^{1}(\KQ_{\Lambda})} \\
h^{-n,-n+1}\directsum h^{-n-2,-n+1}\directsum\dotsm }
\]
This implies an isomorphism $E_{n,0,n}^2(\unit_{\Lambda}) \xrightarrow{\iso} E^2_{n,0,n}(\KQ_{\Lambda})$ for $n\leq 0$.
For $q>0$, 
$E^2_{n,q,n}(\KT_{\Lambda})$ is given by the homology of the complex:
\[ \xymatrix{
h^{q-n-2,q-n-1}\directsum h^{q-n-4,q-n-1}\directsum\dotsm 
\ar[d]_-{d^{1}(\KT_{\Lambda})} \\
h^{q-n,q-n}\directsum h^{q-n-2,q-n}\directsum\dotsm 
\ar[d]_-{d^{1}(\KT_{\Lambda})} \\
h^{q-n,q-n+1}\directsum h^{q-n-2,q-n+1}\directsum\dotsm } 
\]
This implies an isomorphism $E_{n,q,n}^2(\unit_{\Lambda}) \xrightarrow{\iso} E^2_{n,q,n}(\KT_{\Lambda})$;
for $n=0$ see \cite[Theorem 6.3]{roendigs-oestvaer.hermitian}.
For $n>0$ we use the periodicity isomorphism $\eta\colon \Sigma^{1,1}\KT_{\Lambda}\to\KT_{\Lambda}$ \cite[Example 2.3]{roendigs-oestvaer.hermitian}.
\end{proof}

\begin{theorem}
\label{theorem:higher-diff-unit-zero-line}
For all $n,q\in \ZZ$ there are isomorphisms
\[ 
E^\infty_{n,q,n}(\unit_{\Lambda}) 
\cong
E^1_{n,q,n}(\unit_{\Lambda})
\cong
\begin{cases} 
H^{-n,-n} & n\leq 0 \mathrm{\ and\ } q=0 \\
h^{q-n,q-n} & n>0 \mathrm{\ or\ } q>0.
\end{cases} 
\]
\end{theorem}
\begin{proof}
The connectivity of the slices shows $E^1_{n,q,n}(\unit_{\Lambda})$ consists of infinite cycles.
The case $q=0$ follows by effectivity of $\unit_{\Lambda}$, 
which excludes nontrivial differentials entering $E^1_{n,q,n}(\unit_{\Lambda})$. 
Lemma~\ref{lem:unit-kt-zero-line} shows $E^2_{n,q,n}(\unit_{\Lambda})=E^1_{n,q,n}(\unit_{\Lambda})$ and that the unit map $\unit_{\Lambda}\to \KT_{\Lambda}$ induces an isomorphism 
$E^2_{n,q,n}(\unit_{\Lambda})\iso E^2_{n,q,n}(\KT_{\Lambda})$ for $q>0$. 
Since $E^2_{n,q,n}(\KT_{\Lambda}) = E^\infty_{n,q,n}(\KT_{\Lambda})$ by \cite[Theorem 6.3]{roendigs-oestvaer.hermitian} all differentials entering $E^2_{n,q,n}(\KT_{\Lambda})$, 
and hence $E^2_{n,q,n}(\unit_{\Lambda})$, 
are trivial.
In loc.~cit., 
$n=0$;
the general case follows from the isomorphism $\eta\colon\Sigma^{1,1}\KT_{\Lambda}\to \KT_{\Lambda}$. 
\end{proof}

\begin{corollary}
\label{cor:unit-nonneg-zero-line}
If $n>0$, 
the unit maps $\unit_{\Lambda}\to \KQ_{\Lambda}$ and $\unit_{\Lambda}\to \KT_{\Lambda}$ induce isomorphisms 
\[ 
\pi_{0,0}\slicecomp(\unit_{\Lambda}) 
\xrightarrow{\iso} 
\pi_{0,0}\slicecomp(\KQ_{\Lambda})
\text{ and }
\pi_{n,n}\slicecomp(\unit_{\Lambda}) 
\xrightarrow{\iso} 
\pi_{n,n}\slicecomp(\KT_{\Lambda}), 
\]
respectively. 
If $n<0$, 
$\pi_{n,n}\slicecomp(\unit_{\Lambda})$ is the $I$-adic completion of an extension of $H^{-n,-n}$ by $I^{-n+1}$, 
where $I$ denotes the fundamental ideal of the Witt ring.
\end{corollary}
\begin{proof}
The $\pi_{n,n}$-isomorphisms for $n\geq 0$ follow from Theorem~\ref{theorem:higher-diff-unit-zero-line} and the corresponding statements for $\KQ_{\Lambda}$ and $\KT_{\Lambda}$.
If $n<0$, 
$\unit_{\Lambda}\to \s_0(\unit_{\Lambda})$ (with homotopy fiber $\f_1(\unit_{\Lambda})$) induces a surjection $\pi_{n,n}\unit_{\Lambda} \to \pi_{n,n}\s_0(\unit_{\Lambda})\iso H^{-n,-n}$. 
The image of $\pi_{0,0}\f_{-n+1}(\KT_{\Lambda})$ in $\pi_{0,0}\KT_{\Lambda}$ is $I^{-n+1}$ by \cite[Theorem 6.12]{roendigs-oestvaer.hermitian}, 
whence Theorem~\ref{theorem:higher-diff-unit-zero-line} shows 
\[ 
\pi_{n,n}\f_1(\unit_{\Lambda}) \iso \pi_{0,0}\f_{-n+1}(\unit_{\Lambda})
\to 
\pi_{0,0}\f_{-n+1}(\KT_{\Lambda}) \]
induces an isomorphism on $I$-adic completions. 
\end{proof}

\begin{remark}
Suppose the slice filtration on $\pi_{n,n}\unit_{\Lambda}$ is Hausdorff, 
i.e., 
$\unit_{\Lambda}$ is convergent with respect to the slice filtration \cite[Definition 7.1]{voevodsky.open}. 
Theorem~\ref{theorem:higher-diff-unit-zero-line} and Corollary~\ref{cor:unit-nonneg-zero-line} imply 
\[ 
\pi_{0,0}\unit_{\Lambda} 
\xrightarrow{\iso} 
\pi_{0,0}\KQ_{\Lambda}
\text{ and } 
\pi_{n,n}\unit_{\Lambda} 
\xrightarrow{\iso} 
\pi_{n,n}\KT_{\Lambda}
\]
for $n>0$, 
and an extension 
\[ 
0
\to
I^{-n+1} 
\to 
\pi_{n,n}\unit_{\Lambda}
\to
H^{-n,-n}
\to
0,
\]
for $n<0$.
This is Morel's identification of the $0$-line \cite[Theorem 1.23, Corollary 1.25]{morel.field}.
\end{remark}

To compute $\pi_{n+1,n}{\unit_{\Lambda}}^\smash_\eta$ we set out to determine the $E^\infty$-terms of the $-n+1$st column of the $-n$th slice spectral sequence for $\unit_{\Lambda}$. 
Lemma \ref{lem:first-diff-unit-1} and the isomorphism $\tau\colon h^{p,q}\to h^{p,q+1}$ for $0\leq p\leq q$ imply that 
only terms from the first four slices may survive to the $-n+1$st column of the $E^{2}$-page of the $-n$th slice spectral sequence, 
cf.~Remark \ref{remark:pi10}.
Thus the only possibly nonzero differentials entering the $-n+1$st column on the $E^{2}$-page are
\[ 
E^2_{-n+2,0,-n} \to E^2_{-n+1,2,-n} 
\text{ and }
E^2_{-n+2,1,-n}\to E^2_{-n+1,3,-n}.
\]
To show that these $d^{2}$-differentials, 
as well as the $d^{3}$-differential 
\[
E^{3}_{-n+2,0,-n}\to E^{3}_{-n+1,3,-n},
\]
are trivial we shall reduce to global fields.
All $d^{r}$-differentials entering the $-n+1$st column for $r\geq 4$ are trivial for degree reasons, 
so this determines the desired $E^\infty$-terms.

We write $\cd_{p}(F)$ for the $p$-cohomological dimension of a field $F$ \cite[\S3.1]{serre.Gcohomology}.

\begin{lemma}
\label{lem:higher-diff-local-fields}
If $\cd_{2}(F)\leq 2$ and $\cd_{3}(F)\leq 4$, 
then every $r$th differential in the $-n$th slice spectral sequence entering the $-n+1$st column is trivial for $r\geq 2$.
\end{lemma}
\begin{proof}
Note that $E^{2}_{-n+2,0,-n}$ is a subgroup of $H^{n-2,n}$.
The latter group is trivial by definition when $n\leq 0$ and for $n=1,2$ by \cite{lichtenbaum.weight-2}, \cite[(4.1), (4.2)]{merkurjev.weighttwo}.
If $n\geq 3$, 
the possible targets are zero for fields of cohomological dimension at most three for the prime 2 and at most four for the prime 3. 
The group $E^{2}_{-n+2,1,-n}$ is a subquotient of $h^{n-1,n+1}$, 
which is trivial for $n\leq 0$.
If $n\geq 1$,
then $h^{n+2,n+3}=0$ since $\cd_{2}(F)\leq 2$, 
so that $E^{2}_{-n+2,3,-n}=0$. 
\end{proof}

\begin{lemma}
\label{lem:higher-diff-real-complex}
Let $F$ be $\RR$, $\CC$ or a global field of positive characteristic not two.
For $n\in \ZZ$, $r \geq 2$, 
every $d^{r}$-differential in the $-n$th slice spectral sequence for $\unit_{\Lambda}$ entering the $-n+1$st column is trivial.
\end{lemma}
\begin{proof}
Lemma~\ref{lem:higher-diff-local-fields} applies to $\CC$ and global fields of positive dimension \cite[\S4.2]{serre.Gcohomology}.
The mod-$2$ motivic cohomology ring of the reals is $h^{\ast,\ast}(\RR)=\Z/2[\tau,\rho]$.
Any $d^{1}$-differential restricting to $\Sq^2$ and entering the $-n+1$st column is an isomorphism:
up to a cup-product isomorphism by some $\tau$-power,
such a $d^{1}$-differential is cup-product by $\rho^2$ \cite[Corollary 6.2]{roendigs-oestvaer.hermitian}.
Lemma \ref{lem:first-diff-unit-1} implies the only possibly nontrivial $r$th differential is 
\begin{equation}
\label{eq:1} 
d^{2}
\colon
2H^{n-2,n}\to h^{n+2,n+2}_{12};
n\geq 0.
\end{equation}
The group $h^{n+2,n+2}_{12}(\RR)$ is cyclic of order two with a generator represented by $\nu\cdot \{-1,\dotsc,-1\}$, 
where $\nu\colon S^{7,4}\to S^{4,2}$ is the second motivic Hopf map, cf.~Example \ref{ex:Hopf-maps}. 
Taking real points 
yields the real Betti realization map 
$\pi_{p,q}\unit_{\Lambda}\to\pi_{p-q}\unit_{\Top}$,
which maps $\nu$ to $\eta_{\Top}\in \pi_1 \unit_{\Top}$, 
and likewise for every element of the form $\nu\cdot \{-1,\dotsc,-1\}$.
This implies that \eqref{eq:1} is trivial.
\end{proof}

\begin{lemma}
\label{lem:terms-zero-global-field}
Let $F$ be a global field of $\Char(F)\neq 2$.
Then for $n\geq 1$,
\begin{equation*}
h^{n+2,n+3}/\Sq^2h^{n,n+2}
=
h^{n+2,n+3}/\tau\partial^{12}_{2} h_{12}^{n+1,n+2}
=
0.
\end{equation*}
\end{lemma}
\begin{proof}
For $n\geq 1$, $h^{n+2,n+3}=0$ over global fields of positive characteristic \cite[\S4.2]{serre.Gcohomology}.
By the cup-product isomorphism $\tau\colon h^{i,i}\to h^{i,i+1}$ \cite[Corollary 6.1]{roendigs-oestvaer.hermitian} for $i\geq 0$ and Tate's computation of $h^{i,i}$, 
$i\geq 3$, 
for number fields \cite[Theorem A.2]{milnor.k-quadratic} it suffices to consider $\RR$.
For $n\geq 0$, 
\[ 
\Sq^2
\colon 
h^{n,n+2}\to h^{n+2,n+3} 
\]
is the cup-product map by $\rho^2$ \cite[Corollary 6.2]{roendigs-oestvaer.hermitian}, 
and hence it is an isomorphism.
Likewise, 
the second group vanishes since by Theorem \ref{theorem:steenrod-algebra} the composition
\[ 
h^{n+1,n+2}
\xrightarrow{\inc^{2}_{12}} 
h^{n+1,n+2}_{12} 
\xrightarrow{\partial^{12}_{2}} 
h^{n+2,n+2}
\]
coincides with $\Sq^1$,
i.e., 
the cup-product map by $\rho$ \cite[Corollary 6.2]{roendigs-oestvaer.hermitian}.
\end{proof}

\begin{lemma}
\label{lem:higher-diff-global-field}
Let $F$ be a global field of $\Char(F)\neq 2$. 
For $n\in \ZZ$, $r\geq 2$, 
all $r$th differentials in the $-n$th slice spectral sequence for $\unit_{\Lambda}$ entering the $-n+1$st column are trivial.
\end{lemma}
\begin{proof}
Lemma~\ref{lem:terms-zero-global-field} implies this for differentials entering $E^{2}_{-n+1,3,-n}$ and also for the first component of the differential entering 
$E^{2}_{-n+1,2,-n} \iso h^{n+1,n+2}/\Sq^2h^{n-1,n+1}\directsum h^{n+2,n+2}_{12}$. 
We may assume $\Char(F)=0$ by Lemma \ref{lem:higher-diff-local-fields}.
Then there is a canonically induced isomorphism 
\begin{equation*}
h^{n+2,n+2}_{12}(F)
\xrightarrow{\cong}
\bigoplus h^{n+2,n+2}_{12}(\RR),  
\end{equation*}
where the direct sum is indexed by the real places of $F$ \cite[Theorem 2.1]{bass-tate}. 
We conclude using Lemma~\ref{lem:higher-diff-local-fields} and base change.
\end{proof}

\begin{lemma}
\label{lem:differential-module-hom}
The $r$th slice differentials induce an $\mathbf{K}^{\M}_{\ast}\cong\bigoplus_{n\in\NN}H^{n,n}$-module map
\[ 
\bigoplus_{n\in \ZZ}
d^{r}_{p+n,q,n}(\E)
\colon 
\bigoplus_{n\in \ZZ} E^r_{p+n,q,n}(\E) 
\rightarrow 
\bigoplus_{n\in \ZZ} E^r_{p-1+n,q+r,n}(\E). 
\]
\end{lemma}
\begin{proof}
By naturality, any map $f$ of $\Lambda$-local motivic spectra induces a graded module map $\bigoplus_{n\in \ZZ} \pi_{p+n,n} f$ with respect to $\bigoplus_{n\in \ZZ} \pi_{n,n} \unit_{\Lambda}$. 
The first slice differential is induced by the naturally induced map $\s_q(\E)\rightarrow\Sigma^{1,0} \s_{q+1}(\E)$.
Each slice $\s_{q}(\E)$ is a module over the motivic ring spectrum $\s_{0}(\unit_{\Lambda})$,  
cf.~\cite[\S6 (iv),(v)]{grso} and \cite[Theorem 3.6.13(6)]{Pelaez}. 
Thus  $\s_{q}(\E)$ is an $\M \Lambda$-module. 
In particular,
$\eta\in \pi_{1,1}\unit$ acts trivially on both $\s_q(\E)$ and $\Sigma^{1,0} \s_{q+1}(\E)$.
Hence the first differential is an $\bigoplus_{n\in\NN}H^{n,n}$-module map. 
The other cases follow by construction.
\end{proof}

For $\underline{a}=(a_{1},\dots,a_{n})$, 
$a_{i}\in F^{\times}$, 
the Pfister quadric $Q_{\underline{a}}$ is the $(2^{n-1}-1)$-dimensional projective quadric defined by the quadratic form 
$q_{\underline{a}}=\langle\langle a_{1},\dots,a_{n-1}\rangle\rangle-\langle a_{n}\rangle$ \cite[\S2]{ovv}.
Its closed points $(Q_{\underline{a}})_{(0)}$ contains the subset $(Q_{\underline{a}})_{(0,\leq 2)}$ comprised of $x$ for which the degree $[F_{x}:F]\leq 2$ \cite[\S3]{ovv}.
By \cite[Theorem 3.2]{ovv} there is an exact sequence
\begin{equation}
\label{equation:OVVexactsequence}
\bigoplus_{x\in (Q_{\underline{a}})_{(0,\leq 2)}} \mathbf{K}^{\M}_{\ast}(F_{x})/2
\to
\mathbf{K}^{\M}_{\ast}(F)/2
\overset{{\underline{a}}}{\to}
\mathbf{K}^{\M}_{\ast+n}(F)/2
\to
\mathbf{K}^{\M}_{\ast+n}(F(Q_{\underline{a}}))/2.
\end{equation}
In the following we employ this sequence to give substantive applications of presheaves with transfers structure to computational motivic homotopy theory.

\begin{lemma}
\label{lem:sec-differential-to-13}
For $n\in \ZZ$ the $d^{2}$-differential
$
E^2_{-n+2,0,-n} 
\to 
E^2_{-n+1,2,-n} 
$
in the $-n$th slice spectral sequence for $\unit_{\Lambda}$ is trivial.
\end{lemma}
\begin{proof}
Lemma \ref{lem:first-diff-unit-1} shows the direct summand $h^{n+2,n+2}_{12}$ of $E^1_{-n+1,2,-n}$ survives to a direct summand of $E^2_{-n+1,2,-n}$ since the entering first differential is trivial, 
and we deduce
\begin{equation*} 
E^2_{-n+1,2,-n} 
\cong
h^{n+2,n+2}_{4}
\directsum 
h^{n+2,n+2}_{3}
\directsum 
h^{n+1,n+2}/\Sq^2h^{n-1,n+2}.
\end{equation*} 
To deal with $h^{n+2,n+2}_{3}$ we employ the mod-$3$ reduction map $\unit_{\Lambda} \to \unit_{\Lambda}/3$.
In low degrees, 
Corollary~\ref{corollary:small-slices} determines the induced map of slices. 
For the $E^2$-pages of the corresponding slice spectral sequences, 
there is a naturally induced commutative diagram:
\[ 
\xymatrix{
E^2_{-n+2,0,-n}(\unit_{\Lambda}) \ar[r] \ar[d]_{d^{2}({\unit_{\Lambda}})} &  E^2_{-n+2,0,-n}(\unit_{\Lambda}/3) = h^{n-2,n}_{3} \ar[d]^-{d^{2}({\unit_{\Lambda}/3})} \\
E^2_{-n+1,2,-n}(\unit_{\Lambda}) \ar[r] & E^2_{-n+1,2,-n}(\unit_{\Lambda}/3)=h^{n+2,n+2}_{3}} 
\]
The lower horizontal map is the projection.
Thus it suffices to show $d^{2}({\unit_{\Lambda}/3})$ is trivial.
By Lemma~\ref{lem:differential-module-hom},  
the second differential is a graded $\mathbf{K}^{\M}_{\ast}/3$-module map
\[ 
\bigoplus_{n\in\ZZ} d^{2}({\unit_{\Lambda}/3})
\colon 
\bigoplus_{n\in\ZZ} h^{n-2,n}_{3} 
\to 
\bigoplus_{n\in\ZZ} h^{n+2,n+2}_{3}. 
\]
Suppose $F$ contains a primitive third root of unity $\xi\in h^{0,1}_{3}$.  
The $\mathbf{K}^{\M}_{\ast}/3$-module $\bigoplus_{n\in\ZZ} h^{n-2,n}_{3}$ is then generated by $\xi^2\in h^{0,2}_{3}$, 
and it suffices to show $d^{2}({\unit_{\Lambda}/3})(\xi^2)=0\in h^{4,4}_{3}$.  
Let $F_{0}$ be the prime field of $F$.
By base change we may replace $F$ by $F_{0}(\xi)$.
We have $h^{4,4}_{3}(F_{0}(\xi))=0$ by \cite[\S3.3, 4.4]{serre.Gcohomology}.
Hence $d^{2}({\unit_{\Lambda}/3})(\xi^2)=0$, 
see Lemma~\ref{lem:higher-diff-local-fields}.
If $F$ does not contain $\xi$ as above, 
a transfer argument for the quadratic extension $F(\xi)/F$ 
implies the claim.

Next we analyze the component of the second differential entering 
$h^{n+1,n+2}/\Sq^2h^{n-1,n+1}$. The unit map $\unit_{\Lambda}\to \f_{0}(\KQ)_{\Lambda}$
induces a commutative diagram
\[ 
\xymatrix{
E^2_{-n+2,0,-n}(\unit_{\Lambda}) \ar[r] \ar[d]_{d^{2}({\unit_{\Lambda}})} &  E^2_{-n+2,0,-n}(\f_{0}(\KQ_{\Lambda}))  \ar[d]^-{d^{2}({\f_{0}(\KQ_{\Lambda})})} \\
E^2_{-n+1,2,-n}(\unit_{\Lambda}) \ar[r] & E^2_{-n+1,2,-n}(\f_{0}(\KQ_{\Lambda}))=h^{n+1,n+2}/\Sq^{2}h^{n-1,n+1}} 
\]
Here the lower horizontal map is the projection,
and the upper horizontal map is the inclusion of a direct summand.
We conclude using Lemma~\ref{lem:KQeff-sec-differential-to-12}, 
which shows triviality of 
\[ 
d^{2}({\f_{0}(\KQ_{\Lambda})})\colon E^2_{-n+2,0,-n}(\f_{0}(\KQ_{\Lambda})) 
\to 
E^2_{-n+1,2,-n}(\f_{0}(\KQ_{\Lambda})). 
\] 

Finally, 
we consider the component of $d^{2}(\unit_{\Lambda})$ entering $h^{n+2,n+2}_4$ using the canonical map $\unit_{\Lambda} \to\unit_{\Lambda}/12{\hyper}$ for the hyperbolic plane $\hyper=1-\epsilon$.
In $\pi_{0,0}\unit$ we have $12\hyper = 3\hyper^3$.
The slices and slice differentials of $\unit_{\Lambda}/{12\hyper}$ are readily computed. 
In particular, the map
\[ 
E^{2}_{-n+1,2,-n}(\unit_{\Lambda}) 
\to
E^{2}_{-n+1,2,-n}\bigl(\unit_{\Lambda}/{12\hyper}\bigr) 
\]
in slice degree two is a (split) injection.\footnote{Injectivity holds also for $\unit_{\Lambda} \to \unit_{\Lambda}/6{\hyper}$, but the proof of Lemma~\ref{lem:sec-differential-to-14} does not work for 
$\unit_{\Lambda}/6{\hyper}$.} 
Hence it suffices to show the composite map
\[ 
E^2_{-n+2,0,-n} (\unit_{\Lambda})
\to
E^2_{-n+2,0,-n} \bigl(\unit_{\Lambda}/{12\hyper}\bigr)
\xrightarrow{d^{2}} 
E^2_{-n+1,2,-n} \bigl(\unit_{\Lambda}/{12\hyper}\bigr)
\]
is zero. 
In slice degree zero there is a commutative diagram:
\[ 
\xymatrix{
E^{2}_{-n+2,0,-n}(\unit_{\Lambda}) \ar[d] \ar[r] & E^{2}_{-n+2,0,-n}\bigl(\unit_{\Lambda}/{12\hyper}\bigr) \ar[d] \\
H^{n-2,n} \ar[r]^{\pr^{\infty}_{24}} &
h^{n-2,n}_{24}
}
\]
The vertical inclusions are isomorphisms if $\Sq^2\colon h^{n-2,n} \to h^{n,n+1}$ is trivial, 
i.e., 
if $\rho^2=0$.
For $E^{2}_{-n+2,0,-n}\bigl(\unit_{\Lambda}/{12\hyper}\bigr)$ there is a short exact sequence
\begin{equation}
\label{eq:E2-20-12hyper}
0 
\to \ker(\pr^{24}_{2} \colon h^{n-2,n}_{24}\to h^{n-2,n}) 
\to E^{2}_{-n+2,0,-n}\bigl(\unit_{\Lambda}/{12\hyper}\bigr) 
\to \ker(\Sq^2\colon h^{n-2,n}\to h^{n,n+1}), 
\end{equation}
and similarly for $E^{2}_{-n+2,0,-n}(\unit_{\Lambda})$.
The rightmost map in~(\ref{eq:E2-20-12hyper}) is surjective, 
being the base change of $\pr^{24}_{2}\colon h^{n-2,n}_{24}\to h^{n-2,n}$,
which is surjective in bidegree $(0,2)$. 
Hence the boundary map $\partial^{2}_{12}\colon h^{0,2}\to h^{1,2}_{12}$ is trivial.
Since the $\KMil$-module $\directsum_{n\in \ZZ}h^{n-2,n}$ is generated by $\tau^2\in h^{0,2}$, 
the boundary and $\KMil$-module map $\partial^{2}_{12}\colon h^{n-2,n}\to h^{n-1,n}_{12}$ is trivial.
Thus the $\KMil$-module
\[ 
\bigoplus_{n\in \ZZ} E^{2}_{-n+2,0,-n}\bigl(\unit_{\Lambda}/{12\hyper}\bigr)  
\]
has $2$-primary component generated in degrees $(0,-2)$ and $(-1,-3)$, 
since the same holds for the outer terms in~(\ref{eq:E2-20-12hyper}).
We can identify the $\KMil$-module
\[ 
\bigoplus_{n\in \ZZ} \ker(\Sq^{2}\colon h^{n-2,n}\to h^{n,n+1}) 
\]
with $\{ x\in \directsum_{n\in \ZZ} h^{n,n} \vert \rho^2x = 0\}$ generated in degree at most $(1,1)$ by \cite[Theorem 3.3]{ovv} (see \cite[Theorem 2.1]{merkurjev-suslin.rost} if $\Char(F)$ is odd). 
We note the $2$-primary component of the $\KMil$-module 
\[ 
\bigoplus_{n\in \ZZ}
\ker(\pr^{24}_{2} \colon h^{n-2,n}_{24}\to h^{n-2,n})
= 
\bigoplus_{n\in \ZZ}
h^{n-2,n}_{3} \directsum \image(h^{n-2,n}_{4} \to h^{n-2,n}_{8}) 
\]
is generated in degree $(0,2)$ by comparing with $\directsum_{n\in \ZZ} h^{n-2,n}_{4}$.
It remains to prove that the $\KMil$-module generators map trivially under $d^{2}$. 

Let $g_0 \in E^{2}_{0,0,-2}\bigl(\unit_{\Lambda}/{12\hyper}\bigr)$ be a generator in degree $(0,-2)$.
If $g_0\in\directsum_{n\in \ZZ} h^{n-2,n}_{4}$,
then it is defined over the prime field $F_{0}\subset F$.
Naturality of the slice spectral sequence with respect to field extensions implies 
\[ 
d^{2}(g_{0})
\in
E^{2}_{-1,2,-2}\bigl(\unit_{\Lambda}/{12\hyper}\bigr)(F_{0}) 
= 
h^{3,4}/\Sq^{2}h^{1,3} \directsum h^{4,4}\directsum h^{4,4}_{12} (F_{0}) = 
\begin{cases} 
\Lambda/2\directsum \Lambda/2 & F_{0} = \QQ \\
0 & \Char(F)>0. 
\end{cases} 
\]
The inclusion $\QQ \hookrightarrow \RR$ induces an isomorphism on $E^{2}_{-1,2,-2}\bigl(\unit_{\Lambda}/{12\hyper}\bigr)\iso \ZZ/2\{\rho^4\}\directsum \ZZ/2\{\rho^4\nu\}$. 
The real Betti realization sends ${\unit}$ to $\mathbb{S}$, 
$\hyper$ to $0$, 
$\rho$ to $1$, 
and $\nu$ to $\eta_{\Top}$. 
Hence it defines a surjective homomorphism
\[ 
\pi_{-1,-2}\unit/{12\hyper} 
\to 
\pi_{1}\mathbb{S} \directsum \pi_{0}\mathbb{S}.
\]
It follows that $d^{2}(g_0)=0$.
If $g_{0}$ maps to a degree $(0,2)$ generator in the $\KMil$-module
\[ 
\bigoplus_{n\in \ZZ} 
\ker(\Sq^{2}\colon h^{n-2,n}\to h^{n,n+1}), 
\]
then $\rho^2=0$, 
and $-1$ is a sum of at most two squares in $F$ by the proof of \cite[Theorem 1.4]{milnor.k-quadratic}, \cite[Corollary 3.5]{elman-lam.pfister}. 
If $\rho=0$,
then $g_{0}$ is defined over the smallest subfield $F_{0}(\sqrt{-1})$ containing $\sqrt{-1}$.
It follows that $d^{2}(g_{0}) = 0$ since
\[ 
E^{2}_{-1,2,-2}\bigl(\unit_{\Lambda}/{12\hyper}\bigr)(F_{0}(\sqrt{-1})) 
= 
h^{3,4}/\Sq^{2}h^{1,3} 
\directsum 
h^{4,4}
\directsum 
h^{4,4}_{12} (F_{0}(\sqrt{-1})) 
= 
0.\]
If $\rho\neq 0$ but $\rho^2=0$, 
there exist elements $a,b\in F$ such that $a^2+b^2=-1$ and $g_{0}$ is defined over the field $F_{0}(a,b)$ of cohomological dimension at most 3 \cite[\S3.3, 4.4]{serre.Gcohomology}.
Hence $h^{4,4}\directsum h^{4,4}_{12}(F_{0}(a,b)) = 0$ and $d^{2}(g_{0})=0$. 
This shows $d^{2} = 0$ over every field for which $\rho^2=0$.

Suppose $g_1 \in E^{2}_{-1,0,-3}\bigl(\unit_{\Lambda}/{12\hyper}\bigr)$ is a generator in degree $(-1,-3)$.
It maps to a degree $(1,3)$ generator in the $\KMil$-module
\[ 
\bigoplus_{n\in \ZZ} 
\ker(\Sq^{2}\colon h^{n-2,n}\to h^{n,n+1}). 
\]
Every such a generator is of the form $\tau^2a$ with $\rho^2a=0$, 
where $a\in h^{1,1}(F)$ is represented by a unit in $F$. 
The exact sequence \eqref{equation:OVVexactsequence} of \cite[Theorem 3.2]{ovv} implies there exist finitely many quadratic field extensions $L_1,\dotsc,L_m$ of $F$ splitting $\rho^2$ 
(i.e., $\rho^2=0$ over each $L_j$), 
such that $a$ is in the image of the transfer map
\[ 
\sum_{j=1}^m \mathrm{tr}_j
\colon
\bigoplus_{j=1}^m h^{1,1}(L_j) 
\rightarrow 
h^{1,1}(F). 
\]
By \cite[Theorem 1.9]{hoyois.transfer}
the transfer map for $L_j/F$ is induced by the Spanier-Whitehead dual of the structure map
$ 
\Sigma^\infty {\Spec(L_j)}_+
\to 
\Sigma^\infty \Spec(F)_+=\unit_{\Lambda}, 
$
and hence it commutes with the slice differentials.
Since the differential 
$d^{2}({\unit_{\Lambda}/12{\hyper}})$ is zero for fields in which
$\rho^2=0$, 
also $d^{2}(\tau^2{a})=0$, 
which concludes the proof.
\end{proof}

\begin{lemma}
\label{lem:sec-differential-to-14}
For $n\in \ZZ$
the $d^{2}$-differential
$
E^2_{-n+2,1,-n} 
\to 
E^2_{-n+1,3,-n} 
$
in the $-n$th slice spectral sequence for $\unit_{\Lambda}$ is trivial.
\end{lemma}
\begin{proof}
By comparison with $\unit_{\Lambda}/{12\hyper}$ we find a naturally induced (split) injection
\[ 
E^{2}_{-n+1,3,-n}(\unit_{\Lambda}) 
\to 
E^{2}_{-n+1,3,-n}\bigl(\unit_{\Lambda}/{12\hyper}\bigr). 
\]
This reduces the proof to showing triviality of the composite map
\[ 
E^2_{-n+2,1,-n}(\unit_{\Lambda})
\to 
E^2_{-n+2,1,-n} \bigl(\unit_{\Lambda}/{12\hyper}\bigr)
\xrightarrow{d^{2}(\unit_{\Lambda}/{12\hyper})} 
E^2_{-n+1,3,-n} \bigl(\unit_{\Lambda}/{12\hyper}\bigr).
\]
We claim $d^{2}(\unit_{\Lambda}/{12\hyper})\colon E^2_{-n+2,1,-n} \bigl(\unit_{\Lambda}/{12\hyper}\bigr)\to E^2_{-n+1,3,-n} \bigl(\unit_{\Lambda}/{12\hyper}\bigr)$ is trivial.
The canonical pairing $\unit_{\Lambda}\smash \unit_{\Lambda}/{12\hyper} \to \unit_{\Lambda}/{12\hyper}$ induces 
\[ 
\pi_{1,1}\s_1(\unit_{\Lambda}) \times \pi_{-n+1,-n-1} \s_0\bigl(\unit_{\Lambda}/{12\hyper}\bigr) 
=
h^{0,0}\times h^{n-1,n+1}_{24} 
\to 
\pi_{-n+2,-n}\s_1\bigl(\unit_{\Lambda}/{12\hyper}\bigr) \]
sending $(1,x)$ to $\pr^{24}_{2}(x)\in h^{n-1,n+1}\hookrightarrow h^{n-1,n+1}\times h^{n,n+1} = \pi_{-n+2,-n}\s_1\bigl(\unit_{\Lambda}/{12\hyper}\bigr)$. 
Invoking Proposition~\ref{prop:leibniz} yields the induced pairing between $E^{2}$-terms
\[ 
E^{2}_{1,1,1}(\unit_{\Lambda}) \times E^{2}_{-n+1,0,-n-1}(\unit_{\Lambda}/12{\hyper})
\to 
E^{2}_{-n+2,1,-n}(\unit_{\Lambda}/12{\hyper}). 
\]
We claim it surjects onto the image of $E^{2}_{-n+2,1,-n}(\unit_{\Lambda}\to \unit_{\Lambda}/12{\hyper})$:
Product with $1\in h^{0,0}=E^{2}_{1,1,1}(\unit_{\Lambda})$ yields $E^{2}_{-n+2,0,-n}\bigl(\unit_{\Lambda}/{12\hyper}\bigr)\to \ker(\Sq^2\colon h^{n-2,n}\to h^{n,n+1})$ in \eqref{eq:E2-20-12hyper}, 
which is surjective by the proof of Lemma~\ref{lem:sec-differential-to-13}. 
Hence for every $x$ in the image of $E^{2}_{-n+2,1,-n}(\unit_{\Lambda}\to\unit_{\Lambda}/12{\hyper})$,
there exists $y\in E^{2}_{-n+1,0,-n-1}(\unit_{\Lambda}/12{\hyper})$ with $x=1\cdot y$. 
Proposition~\ref{prop:leibniz} implies the vanishing
\[ 
d^{2}(x) 
= 
d^{2}(1\cdot y) 
= 
d^{2}(1)\cdot y \pm 1 \cdot d^{2}(y) 
= 
0\cdot y \pm 1\cdot 0 
= 
0, 
\]
where $d^{2}(y)=0$ by the proof of Lemma~\ref{lem:sec-differential-to-13}. 
The complementary direct summand of $E^{2}_{-n+2,1,-n}(\unit_{\Lambda}/12{\hyper})$ is $h^{n,n+1}/\Sq^2\pr^{24}_{2}h^{n-2,n}_{24}=h^{n,n+1}/\Sq^2h^{n-2,n}$.
Here $g_{0}:=\tau \in h^{0,1}$,
which is defined over $F_{0}$, 
generates the $\KMil$-module $\directsum_{n\in \Z} h^{n,n+1}/\Sq^2h^{n-2,n}$.
It suffices to show $d^{2}(g_0)=0$ over $\QQ$.
We note that $g_0$ detects an element in $\pi_{2,0}\unit_{\Lambda}/12\hyper$ mapping to $\eta_{\Top}\in \ker(12\hyper:\pi_{1,0}\unit_{\Lambda} \to \pi_{1,0}\unit_{\Lambda})$;
its nontriviality follows from Betti realization.
It follows that $d^{2}(\unit_{\Lambda}/{12\hyper})\colon E^2_{-n+2,1,-n} \bigl(\unit_{\Lambda}/{12\hyper}\bigr)\to E^2_{-n+1,3,-n} \bigl(\unit_{\Lambda}/{12\hyper}\bigr)$ is trivial.
\end{proof}

\begin{lemma}
\label{lem:third-differential-to-14}
For $n\in \ZZ$ the $d^{3}$-differential
$
E^3_{-n+2,0,-n} 
\to 
E^3_{-n+1,3,-n} 
$
in the $-n$th slice spectral sequence for $\unit_{\Lambda}$ is trivial.
\end{lemma}
\begin{proof}
Lemmas \ref{lem:sec-differential-to-13} and \ref{lem:sec-differential-to-14} imply $E^2_{-n+2,0,-n} =E^3_{-n+2,0,-n}$, $E^2_{-n+1,3,-n}=E^3_{-n+1,3,-n}$, 
and injectivity of $E^3_{-n+1,3,-n}(\unit_{\Lambda}\to \unit_{\Lambda}/12\hyper)$.
It remains to prove the composite
\[ 
E^3_{-n+2,0,-n}(\unit_{\Lambda}) 
\to 
E^3_{-n+2,0,-n}(\unit_{\Lambda}/12\hyper) 
\xrightarrow{d^{3}({\unit_{\Lambda}/12\hyper})}
E^3_{-n+1,3,-n}(\unit_{\Lambda}/12\hyper) 
\]
is trivial.
To that end it suffices to consider the generators of the $\KMil$-module
\[ 
\bigoplus_{n\in\ZZ} E^3_{-n+2,0,-n}(\unit_{\Lambda}/12\hyper) 
= 
\bigoplus_{n\in\ZZ} \ker(h^{n-2,n}_{24}\xrightarrow{\Sq^2\pr^{24}_{2}} h^{n,n+1}).
\] 
We apply the same techniques as in the proofs of Lemmas~\ref{lem:sec-differential-to-13} and \ref{lem:sec-differential-to-14},
e.g., 
base change to field extensions of $F_{0}$ for which the generators map trivially by a cohomological dimension consideration, 
and the real Betti realization.
Further details are left to the reader.
\end{proof}

To summarize, 
Lemmas \ref{lem:first-diff-unit-1}, \ref{lem:sec-differential-to-13}, \ref{lem:sec-differential-to-14}, and \ref{lem:third-differential-to-14} imply, 
for all $q,n\in\ZZ$, 
the equality
\begin{equation}
\label{eq:higher-diff-triv} 
E^{\infty}_{-n+1,q,-n}(\unit_{\Lambda})
=
E^2_{-n+1,q,-n}(\unit_{\Lambda}). 
\end{equation}
In the remainder of this section, 
we shift focus from the motivic sphere spectrum $\unit$ to the hermitian $K$-theory spectrum $\KQ$, in order to study the unit map $\unit \to \KQ$.
However, some of the earlier arguments do not apply to $\KQ$ since 
it is not effective. 
By computation the (possibly) nontrivial entries for nonnegative $q$ in the relevant column are given by:
\begin{center}
\begin{tabular}{lll}
\hline
$q$ & $E^2_{-n+1,q,-n}(\KQ)$ \\ \hline
$2$ & $h^{n+1,n+2}/\Sq^{2}(h^{n-1,n+1})$ \\
$1$ & $h^{n,n+1}/\Sq^{2}\pr(H^{n-2,n})$ \\
$0$ & $H^{n-1,n}$  
\end{tabular}
\end{center}
These groups may be hit by higher differentials originating from the negative slices of $\KQ$,
potentially increasing the kernel of $\pi_{n+1,n}\unit_{\Lambda} \to \pi_{n+1,n}\KQ_{\Lambda}$.
(Some of the source groups of these differentials are trivial according to the Beilinson-Soul\'e vanishing conjecture.) 
However, 
since $\unit$ is effective,
$\unit\to \KQ$ factors over the effective cover $\f_{0}(\KQ)$ with trivial negative slices. 

\begin{proposition}
\label{prop:Einfty-KQ}
There are isomorphisms
\[ 
E^\infty_{-n+1,q,-n}(\f_0(\KQ)) = \begin{cases}
0 & q>2 \\
h^{n+1,n+2}/\Sq^{2}(h^{n-1,n+1}) & q=2 \\
h^{n,n+1}/\Sq^{2}\pr(H^{n-2,n}) & q=1 \\
H^{n-1,n}\oplus h^{n-5,n}\oplus h^{n-9,n}\oplus \dotsm & q=0 \\
0 & q<0.
\end{cases}
\]
\end{proposition}
\begin{proof}
Theorem \ref{theorem:slices-kq} implies that $E^{1}_{p,q,-n}(\f_0(\KQ))$ is given by
\[ 
\pi_{p,-n}\s_{q}(\f_0(\KQ)) 
= 
\begin{cases}
H^{2q-p,q+n}\directsum \bigoplus_{i<\frac{q}{2}} h^{2i+(q-p),q+n} & 0\leq q\equiv 0 \bmod 2 \\
\bigoplus_{i<\frac{q+1}{2}} h^{2i+(q-p),q+n} & 0\leq q\equiv 1 \bmod 2 \\
0 & 0>q.
\end{cases}
\]
The first differential for $\KQ$ is described in \cite[Theorem 5.5]{roendigs-oestvaer.hermitian}. 
It coincides with the first differential for $\f_0(\KQ)$ on nonnegative slices. 
From this we deduce the computation
\[ 
E^2_{-n+1,q,-n}(\f_0(\KQ)) = \begin{cases}
0 & q>2 \\
h^{n+1,n+2}/\Sq^{2}(h^{n-1,n+1}) & q=2 \\
h^{n,n+1}/\Sq^{2}\pr(H^{n-2,n}) & q=1 \\
0 & q<0.
\end{cases} \]
The case $q=0$ is special since the entering differential is zero for $\f_0(\KQ)$, but nonzero for $\KQ$ if $n\leq -3$. 
Hence $E^2_{-n+1,0,-n}(\f_0(\KQ))$ is the kernel of the map
\begin{align*}
H^{n-1,n}\oplus h^{n-3,n}\oplus \dotsm & \to  h^{n+1,n+1}\oplus h^{n-1,n+1} \oplus h^{n-3,n+1}\oplus h^{n-5,n+1}\oplus \dotsm \\ 
(x_{n-1},x_{n-3},\dotsc) & \mapsto (\Sq^3\Sq^1 x_{n-3},0,\tau x_{n-3}+\Sq^3\Sq^1 x_{n-7},0,\dotsc),
\end{align*}
i.e., 
$H^{n-1,n}\oplus h^{n-5,n}\oplus \dotsm$. 
This determines the $E^2$-page. 
The only possibly nontrivial $d_2$-differential is $E^2_{-n+2,0,-n}(\f_0(\KQ))\to E^2_{-n+1,2,-n}(\f_0(\KQ))$, whose triviality is
provided below in Lemma~\ref{lem:KQeff-sec-differential-to-12}. 
\end{proof}

\begin{lemma}
\label{lem:KQeff-sec-differential-to-12}
The second differential
\[ 
E^2_{-n+2,0,-n}(\f_{0}(\KQ))  
\to 
E^2_{-n+1,2,-n}(\f_{0}(\KQ)) 
\]
in the $-n$th slice spectral sequence for $\f_{0}(\KQ)$ is trivial.
\end{lemma}
\begin{proof}
Observe first that, 
analogous to the computation of $E^2_{-n+1,0,-n}(\f_0(\KQ))$, 
one obtains
\[ 
E^2_{-n+2,0,-n}(\f_0(\KQ)) 
\iso 
J^{n-2,n}\directsum j^{n-6,n} \directsum j^{n-10,n} \directsum \dotsm;
\]
$J^{n-2,n}:=\ker(H^{n-2,n} \xrightarrow{\pr^{\infty}_{2}} h^{n-2,n} \xrightarrow{\Sq^2} h^{n,n+1})$, 
$j^{n-2m,n}:=\ker(h^{n-2m,n} \xrightarrow{\Sq^2} h^{n-2m+2,n+1})$. 
The claim follows readily for $n\leq 1$, 
since then $E^{1}_{-n+2,0,-n}(\f_{0}(\KQ)) =0$.
It also follows if the cohomological dimension $\cd_{2}(F[\sqrt{-1}])\leq 2$, 
since then 
\[ 
E^2_{-n+1,2,-n}(\f_{0}(\KQ))
= 
h^{n+1,n+2}/\Sq^2h^{n-1,n+1}
= 
0
\text{ for } n\geq 2.
\]
In general, 
the cup-product map $\rho\colon h^{n,n}(F)\to h^{n+1,n+1}(F)$ is surjective when $n\geq \cd_{2}(F[\sqrt{-1}])$.
For $q\geq 0$ we recall from Theorem~\ref{theorem:slices-kq} the slices 
\[ 
\s_{q}(\f_{0}(\KQ/2)) 
= 
\bigvee_{i\leq q} \Sigma^{q+i,q} \MZ/2. 
\]
There is a naturally induced commutative diagram:
\[ 
\xymatrix{
E^{2}_{-n+2,0,-n} (\f_{0}(\KQ)) \ar[r] \ar[d]_{d^{2}} & 
E^{2}_{-n+2,0,-n} (\f_{0}(\KQ/2)) \ar[d]^{d^{2}} \\ 
E^{2}_{-n+1,2,-n} (\f_{0}(\KQ)) \ar[r]  & 
E^{2}_{-n+1,2,-n} (\f_{0}(\KQ/2))
}
\]
Here the lower horizontal map is split injective by a comparison of slices and $d^{1}$-differentials.
Hence it suffices to prove the composite
\[  
\xymatrix{
E^{2}_{-n+2,0,-n} (\f_{0}(\KQ)) \ar[r] &
E^{2}_{-n+2,0,-n} (\f_{0}(\KQ/2)) \ar[r]^{d^{2}} &
E^{2}_{-n+1,2,-n} (\f_{0}(\KQ/2))
}
\]
is trivial.  
The first map in the composite is trivial when $n \leq 2$ via identification with 
\[ 
\ker (H^{n-2,n}\xrightarrow{\Sq^2\pr^{\infty}_{2}} h^{n,n+1})
\to 
\ker (h^{n-2,n}\xrightarrow{\Sq^2} h^{n,n+1}).
\]
The source and target are trivial for $n<2$, 
and the connecting map $h^{0,2}\xrightarrow{\partial^{2}_{\infty}} H^{1,2}$ is injective since $H^{1,2}$ contains a (unique) element of order two \cite{levine.indec}, \cite{merkurjev-suslin.k3}, 
whence $H^{0,2}\to h^{0,2}$ is trivial.

Suppose $\sqrt{-1}\in F$, so that $\rho=0\in h^{1,1}$.
Then the $\KMil$-module 
\begin{equation}\label{eq:kmil-module}
\bigoplus_{n\in \ZZ}  E^{2}_{-n+2,0,-n} (\f_{0}(\KQ/2)) 
= 
\bigoplus_{n\in \ZZ} h^{n-2,n}
\directsum
\bigoplus_{n\in \ZZ} h^{n-6,n}
\directsum
\bigoplus_{n\in \ZZ} h^{n-10,n}
\directsum \dotsm
\end{equation}
is generated by the elements $\tau^{4k+2}\in h^{0,4k+2}$, $k\geq 0$, 
defined over $F_{0}(\sqrt{-1})$.
By cohomological dimension, these generators map trivially, 
and hence also~(\ref{eq:kmil-module}) by Lemma~\ref{lem:differential-module-hom}.

Suppose $\sqrt{-1}\notin F$, but $\rho^2=0$. 
Again we may choose $a,b\in F$ such that $a^2+b^2=-1$ \cite[Corollary 3.5]{elman-lam.pfister}. 
The $\KMil$-module~(\ref{eq:kmil-module}) is then generated by the elements $\tau^{4k+2}\in h^{0,4k+2}$, $k\geq 0$, defined over $F_{0}(a,b)$.
If $\Char(F)>0$ or $a$ is algebraic over $F_{0}$,
$\tau^{4k+2}$ maps trivially by cohomological dimension. 
In the remaining case ($\Char(F)=0$ and $a$ transcendental over $\QQ$),
$\tau^2$ is the only generator that may map nontrivially. 
Nevertheless, 
the $d^{2}$-differential 
\[\xymatrix{  
E^{2}_{-n+2,0,-n} (\f_{0}(\KQ/2)) 
\ar[r]^{d^{2}} &
E^{2}_{-n+1,2,-n} (\f_{0}(\KQ/2))}
\]
is trivial over $F_{0}(a,b)$ for all $n\geq 3$,
since the target is then trivial by cohomological dimension.
It follows that $d^{2}\colon E^2_{-n+2,0,-n}(\f_{0}(\KQ)_{})\to E^2_{-n+1,2,-n}(\f_{0}(\KQ)_{})$ is trivial over $F_{0}(a,b)$. 
This computes $\pi_{-n+1,-n} \f_{0}(\KQ)_{}$ via the slice filtration over $F_{0}(a,b)$:
\begin{equation}
\label{equation:ingeneralnontrivialextensions}
\xymatrix{
h^{n+1,n+2}= \f_{2}\pi_{-n+1,-n}\f_{0}(\KQ) \ar[r] &
\f_{1}\pi_{-n+1,-n}\f_{0}(\KQ) \ar[r] \ar[d] &
\pi_{-n+1,-n}\f_{0}(\KQ) \ar[d] \\
& h^{n,n+1} & H^{n-1,n}\directsum h^{n-5,n}\directsum h^{n-9,n}\directsum \dotsm 
} 
\end{equation}
In general, 
\eqref{equation:ingeneralnontrivialextensions} is a nontrivial extension of $\mathbf{K}^{\mathrm{MW}}_{\ast}$-modules. 
If $d^{2}\colon E^2_{-n+2,0,-n}(\f_{0}(\KQ)_{})\to E^2_{-n+1,2,-n}(\f_{0}(\KQ)_{})$ is trivial, 
the above computation extends to any $F$ by replacing $h^{n,n+1}$ with $h^{n,n+1}/\Sq^2(h^{n-2,n})$,
and similarly for $\f_{2}\pi_{-n+1,-n}\f_{0}(\KQ)$.

{\bf Claim:} There is a split extension of abelian groups 
\[ \xymatrix{
0\ar[r] & \f_{2}\pi_{-n+1,-n}\f_{0}(\KQ) \ar[r] &
\f_{1}\pi_{-n+1,-n}\f_{0}(\KQ) \ar[r] 
& h^{n,n+1}/\Sq^2 \ar[r] & 0.
} 
\]

{\bf Proof of Claim:} If $d^{2}$ is trivial for $\f_{0}(\KQ)$, 
then $\f_{2}\pi_{-n+1,-n}\f_{0}(\KQ)=h^{n+1,n+2}/\Sq^2$. 
The first nontrivial group $\f_{2}\pi_{2,1}\f_{0}(\KQ)=\pi_{2,1}\f_{0}(\KQ) = h^{0,1}$ is generated by the image of $\eta\eta_{\Top}$.
The extension for $\f_{1}\pi_{1,0}\f_{0}(\KQ)=\pi_{1,0}\f_{0}(\KQ)$ splits by comparison with the composite
\[ 
\pi_{1}^{\Top} \mathbb{S} \to \pi_{1,0}\unit\to \pi_{1,0}\f_{0}(\KQ) \to \pi_{1,0}\KQ \to \pi_{1}^{\Top} \mathbf{KO}, 
\]
obtained from Betti realization. 
The summand $h^{0,1}$ is generated by the image of $\eta_{\Top}$.
Define the map $h^{n,n+1}\to \f_{1}\pi_{-n+1,-n}\f_{0}(\KQ)$ by sending $\tau\phi$, 
$\phi\in \mathbf{K}^{\Mil}_{n}$, 
to the image of $[\phi]\eta_{\Top}$. 
Then 
\[ 
h^{n,n+1}
\to 
\f_{1}\pi_{-n+1,-n}\f_{0}(\KQ) 
\to 
\pi_{-n+1,-n}\s_{1}(\f_{0}(\KQ)) 
= 
h^{n,n+1} 
\]
is the identity since $\tau\in h^{0,1}=\pi_{1,0}\s_{1}(\unit_{\Lambda})$ detects $\eta_{\Top}\in \pi_{1,0}\unit_{\Lambda} = \f_{1}\pi_{1,0}\unit_{\Lambda}$. 
It follows that 
\begin{equation}
\label{eq:splitting-oneline-kq}
\f_{1}\pi_{-n+1,-n}\f_{0}(\KQ) \iso h^{n,n+1}/\Sq^2 \directsum h^{n+1,n+2}/\Sq^2
\end{equation}
as abelian groups over any field with trivial $d^{2}\colon E^2_{-n+2,0,-n}(\f_{0}(\KQ)_{})\to E^2_{-n+1,2,-n}(\f_{0}(\KQ)_{})$.  
The $\KMW$-module structure of $\bigoplus_{n\in \ZZ}\f_{1}\pi_{-n+1,-n}\f_{0}(\KQ)$ is given by $\eta\cdot (x,y) = (0,x)$ and for $u\in F^{\times}$, 
$[u]\cdot (x,y) =\bigl([u]\cdot x,[u]\cdot y\bigr)$, 
compatible with~(\ref{eq:splitting-oneline-kq}). 
The abelian group extension
\[ 
0
\to 
\f_{1}\pi_{-n+1,-n}\f_{0}(\KQ) 
\to
\pi_{-n+1,-n}\f_{0}(\KQ)
\to 
H^{n-1,n}\directsum h^{n-5,n}\directsum h^{n-9,n}\directsum \dotsm
\to 
0
\]
is 
classified by an element 
\begin{align*} 
x=(x_1,x_2) 
& \in 
\Ext_{\Ab} (H^{n-1,n}\directsum h^{n-5,n}\directsum h^{n-9,n}\directsum \dotsm, h^{n,n+1}/\Sq^2 \directsum h^{n+1,n+2}/\Sq^2)
\end{align*}
via the identification~(\ref{eq:splitting-oneline-kq}). Observe that $x_1\neq 0$ over $F=\QQ$:
The third algebraic $K$-group $\pi_{-1,-2}\kgl(\QQ) = \pi_{-1,-2}\KGL(\QQ) = K_{3}(\QQ) \cong \ZZ/48$ is the target of $\pi_{-1,-2}\f_{0}(\KQ)(\QQ)$ via the forgetful map, 
which on extensions takes the form:
\[ \xymatrix{
\f_{1}\pi_{-1,-2}\f_{0}(\KQ) \iso h^{2,3}(\QQ)/\Sq^2 \directsum h^{3,4}(\QQ)/\Sq^2
\ar[r] \ar[d]_{(\partial^{2}_{\infty},0)} & 
\pi_{-1,-2}\f_{0}(\KQ)(\QQ)
\ar[r] \ar[d]_{\mathrm{forget}} & 
H^{1,2} \ar[d]^{\id} \\
\f_{1}\pi_{-1,-2}\kgl \iso H^{3,3}(\QQ)\iso \ZZ/2
\ar[r]  & 
\pi_{-1,-2}\kgl(\QQ)
\ar[r]  & 
H^{1,2}(\QQ) \cong \ZZ/24 
}
\]
The vertical maps are surjective, and it follows that the top row is not a trivial extension. 
(Note that $h^{3,4}(\QQ)/\Sq^{2}=0$.)

{\bf Claim:}
$x_2=0$ over $\QQ(a,b)$ (and more generally over any field of characteristic not two).

{\bf Proof of Claim:} The degree two extension $\QQ(a,b)/\QQ(a)$ yields an injection $H^{1,2}(\QQ(a)) \hookrightarrow H^{1,2}(\QQ(a,b))$ by \cite[Corollary 4.6]{levine.indec}. 
For the purely transcendental field extension $\QQ(a)/\QQ$ we have $H^{1,2}(\QQ)\cong H^{1,2}(\QQ(a)) \cong \ZZ/24$ by \cite[p.~327]{levine.indec}. 
Moreover,  
on $2$-torsion subgroups \cite[Corollary 4.4]{levine.indec} implies ${}_{2}H^{1,2}(\QQ(a))={}_{2}H^{1,2}(\QQ(a,b))$.

If the extension
\[ 
0
\to 
h^{3,4}(\QQ(a,b))
\to 
E(\QQ(a,b)) 
\to 
H^{1,2}(\QQ(a,b))
\to 
0
\]
corresponding to $x_2$ is nontrivial, 
then so is the extension between the $2$-torsion subgroups
\[ 
0
\to 
h^{3,4}(\QQ(a,b))
\to 
E^\prime(\QQ(a,b)) 
\to 
{}_{2}H^{1,2}(\QQ(a,b))
\to 
0.
\]
Since $h^{3,4}(\QQ)/\Sq^{2}h^{1,3}(\QQ)=0$ the extension corresponding to $x_2$ over $\QQ$ implies $E(\QQ)=H^{1,2}(\QQ)$, 
and hence $E^\prime(\QQ)={}_{2}H^{1,2}(\QQ)={}_{2}H^{1,2}(\QQ(a,b))$.
By naturality there is a splitting ${}_{2}H^{1,2}(\QQ(a,b))\to E^\prime(\QQ(a,b))$,
which shows $x_2=0$ over $\QQ(a,b)$.
It follows that 
$$
\pi_{-1,-2}\f_{0}(\KQ)(\QQ(a,b))\to \pi_{-1,-2}\f_{0}(\KQ/2)(\QQ(a,b)
$$ 
restricts to an injection on the summand $h^{3,4}(\QQ(a,b))$.
We conclude that 
$$
d^{2}\colon E^2_{0,0,-2}(\f_{0}(\KQ/2))\to E^2_{-1,2,-2}(\f_{0}(\KQ/2))
$$ 
is trivial over $\QQ(a,b)$.  
Base change implies the same result if $\rho^2=0$.
Using the generator of $E^2_{0,0,-2}(\f_{0}(\KQ/2))=h^{0,2}$ we deduce triviality of
$$
d^{2}\colon E^2_{-n+2,0,-n}(\f_{0}(\KQ/2))\to E^2_{-n+1,2,-n}(\f_{0}(\KQ/2)).
$$ 

In case the pure symbol $\rho^2\neq 0$, 
we conclude by applying to $\f_{0}(\KQ/2)$ the argument with \eqref{equation:OVVexactsequence} from \cite[Theorem 3.3]{ovv} (\cite[Theorem 2.1]{merkurjev-suslin.rost} for fields of odd characteristic) 
as in the proof of Lemma \ref{lem:sec-differential-to-13}.
\end{proof}

\section{The $1$-line}
\label{section:1line}

Throughout this section we fix a compatible pair $(F,\Lambda)$, where $F$ is a field of $\Char(F)\neq 2$.
Theorem~\ref{theorem:eta-slice-completion} shows the slice spectral sequence for $\unit_{\Lambda}$ determines the homotopy groups of the $\eta$-completion ${\unit_{\Lambda}}^\wedge_\eta$.
A straightforward calculation of $E^2$-pages using Lemma \ref{lem:first-diff-unit-1} shows $\pi_{n+1,n}{\unit_{\Lambda}}^\wedge_\eta = 0$ for $n\geq 3$, 
and $\pi_{n+2,n}{\unit_{\Lambda}}^\wedge_\eta= 0$ for $n\geq 5$. 
For all $n\in \ZZ$ we deduce 
\[ 
\pi_{n+1,n}{\unit_{\Lambda}}^\wedge_\eta[\tfrac{1}{\eta}]
= 
\pi_{n+2,n}{\unit_{\Lambda}}^\wedge_\eta[\tfrac{1}{\eta}] 
=
0. 
\]
Combined with Lemma \ref{lem:arithmetic-square}, 
the arithmetic square for $\eta$,
we obtain short exact sequences
\[ 
0 
\to \pi_{n,n}{\unit_{\Lambda}} 
\to \pi_{n,n} {\unit_{\Lambda}}^\wedge_\eta\directsum \pi_{n,n}{\unit_{\Lambda}}[\tfrac{1}{\eta}] 
\to \pi_{n,n}{\unit_{\Lambda}}^\wedge_\eta[\tfrac{1}{\eta}]\to 0,
\]
and 
\[ 
0 
\to \pi_{n+1,n}{\unit_{\Lambda}} 
\to \pi_{n+1,n}{\unit_{\Lambda}}^\wedge_\eta\directsum \pi_{n+1,n}{\unit_{\Lambda}}[\tfrac{1}{\eta}] 
\to 0. 
\]

\begin{theorem}
\label{theorem:pi1-eta-inv-sphere}
The groups $\pi_{n+1,n}{\unit_{\Lambda}}[\tfrac{1}{\eta}]$ and $\pi_{n+2,n}{\unit_{\Lambda}}[\tfrac{1}{\eta}]$ are trivial for all $n\in\ZZ$.
\end{theorem}
\begin{proof}
This is shown for $\unit$ in \cite[Theorem 8.3]{roendigs.etainv}.
\end{proof}

\begin{corollary}
\label{corollary:1linecorollary}
For all $n\in\Z$ there is a canonically induced isomorphism
\[ 
\pi_{n+1,n}\unit_{\Lambda}
\xrightarrow{\cong}
\pi_{n+1,n}{\unit_{\Lambda}}^\wedge_\eta. 
\]
\end{corollary}

\begin{proposition}
\label{prop:eff-kq-conv}
The canonical map $\f_0(\KQ) \to \slicecomp (\f_0(\KQ))$ induces an isomorphism
\[ 
\pi_{n+k,n}\f_0(\KQ) 
\xrightarrow{\cong}
\pi_{n+k,n}\slicecomp (\f_0(\KQ)) 
\]
for all integers $n$ and $k\equiv 1,2\bmod 4$.
\end{proposition}
\begin{proof}
Since $\KT = \KQ[\tfrac{1}{\eta}]$ and $\eta$ acts invertibly on $\s_\ast(\KT)$, 
the columns in the respective slice spectral sequences for $\KQ$ and $\KT$ agree outside a finite range. 
The computation of the $E^\infty$ page for $\KT$ from \cite[Theorem 6.3]{roendigs-oestvaer.hermitian} implies $\pi_{n+k,n}\slicecomp(\KQ)[\tfrac{1}{\eta}]=0$ 
for $k\not\equiv 0\bmod 4$.
Example~\ref{ex:kq-sc-eta} shows the slice- and $\eta$-completions agree on $\f_0(\KQ)$. 
In the $\eta$-arithmetic square for $\f_0(\KQ)$, 
see Lemma~\ref{lem:arithmetic-square}, 
we observe that $(\f_0(\KQ))[\tfrac{1}{\eta}]$ is the homotopy colimit of the diagram
\[ 
\f_0(\KQ) \to \f_{-1}\Sigma^{-1,-1}\KQ \to \f_{-2}\Sigma^{-2,-2}\KQ \to \dotsm, 
\]
which agrees with $\KQ[\tfrac{1}{\eta}] = \KT$.
The result follows from the vanishing $\pi_{n+k,n}\KT = 0$ for $k\not\equiv 0\bmod 4$. 
\end{proof}

\begin{lemma}
\label{lem:kernel-eta}
Let $U_n$ be the kernel of the naturally induced map
\[ 
\pi_{n+1,n} \unit_{\Lambda}  
\rightarrow 
\pi_{n+1,n} \f_0(\KQ_{\Lambda}).
\]
The graded group $\bigoplus_{n\in \ZZ} U_n$ has a naturally induced $\bigoplus_{n\in \ZZ} \pi_{n,n}\unit_{\Lambda}/(\eta)$-module structure.
\end{lemma}
\begin{proof}
We show that multiplication with $\eta$ induces the trivial map $\eta\colon U_{n}\to U_{n+1}$ for all $n\in \ZZ$.
For any $a\in U_n$ there is a natural number $q$ such that $a\in\f_q\pi_{n+1,n}\unit_{\Lambda}$ and $a\notin\f_{q+1}\pi_{n+1,n}\unit_{\Lambda}$. 
Then $a$ maps to a nontrivial element $b$ in a subquotient of $\pi_{n+1,n}\s_q(\unit_{\Lambda})$.
Since $a\in U_n$, 
$b$ maps trivially to the corresponding subquotient of $\pi_{n+1,n}\s_q(\KQ_{\Lambda})$ on $E^2=E^\infty$-pages by Propositions \ref{prop:Einfty-KQ} and \ref{prop:eff-kq-conv}.
Lemmas \ref{lem:slices-unit-kq-1} and \ref{lem:slices-unit-kq-2} show there are only two such subquotients on the $E^2$-page of the $n$th slice spectral sequence for $\unit_{\Lambda}$ by \eqref{eq:higher-diff-triv},
namely
\[ 
h^{-n+2,-n+3}/\tau\partial^{12}_{2}(h^{-n+1,-n+2}_{12})+\Sq^2(h^{-n,-n+2})
\text{ and }
h^{-n+2,-n+2}_{12}. 
\]
If $b\in h^{-n+2,-n+2}_{12}$, then it lifts to an element $b^\prime\in \pi_{n+1,n}\s_2(\unit_{\Lambda})$ and 
$\eta b^\prime\in\pi_{n+2,n+1}\s_3(\unit_{\Lambda}) \iso h^{-n+1,-n+2}\directsum h^{-n+3,-n+2}\iso h^{-n+1,-n+2}$.
By Lemma~\ref{lem:slices-unit-hopf}, 
$\eta b^\prime=\partial^{12}_{2} b^\prime=0$, 
and hence $\eta b = 0$. 
If $b\in h^{-n+2,-n+3}/\bigl(\tau\partial^{12}_{2}(h^{-n+1,-n+2}_{12})+\Sq^2(h^{-n,-n+2})\bigr)$, then it lifts to an element $b^\prime$ in $\pi_{n+1,n}\s_3(\unit_{\Lambda})$. 
Lemma \ref{lem:first-diff-unit-1} shows the product $\eta b^\prime\in\pi_{n+2,n+1}\s_4(\unit_{\Lambda}) \iso h^{-n+2,-n+3}$ is hit by the $d^{1}$-differential exiting tridegree $(n+3,3,n+1)$,  
and hence $\eta b = 0$.
Since multiplication with $\eta$ is natural with respect to the slice filtration, 
it follows that $\eta a =0$. 
\end{proof}

\begin{theorem}
\label{theorem:1line}
Let $(F,\Lambda)$ be a compatible pair, where $F$ is a field of $\Char(F)\neq 2$.
The unit map $\unit_{\Lambda}\to \KQ_{\Lambda}$ induces an exact sequence
\[ 
0 
\rightarrow 
\mathbf{K}^{\M}_{2-n}\otimes \Lambda/24 
\rightarrow 
\pi_{n+1,n} \unit_{\Lambda}
\rightarrow 
\pi_{n+1,n}\f_0(\KQ_{\Lambda}). 
\]
If $n\geq -4$, 
the rightmost map is surjective.
In particular, 
since $\pi_{n+1,n}\f_0(\KQ_{\Lambda})=0$ for $n\geq 2$,
$\pi_{3,2}\unit_{\Lambda}\cong \Lambda/24$ and $\pi_{n+1,n}\unit_{\Lambda}=0$ for $n\geq 3$.
\end{theorem}

\begin{proof}
The well-known vanishing $\pi_{n+1,n}\f_0(\KQ)=\pi_{n+1,n}\KQ=GW^{[n]}_{1-n}(F)=W^{2n-1}(F)=0$ for $n>1$ follows for example from \cite[Proposition~6.3]{schlichting}.
In the proof we shall make use of the computations in Proposition~\ref{prop:Einfty-KQ}.

Lemma \ref{lem:first-diff-unit-1} and (\ref{eq:higher-diff-triv}) show the following motivic cohomology groups in the $-n$th slice spectral sequence for $\unit_{\Lambda}$ 
are the only terms contributing to the $1$-line.
\begin{center}
\begin{tabular}{lll}
\hline
$q$ & $E^{\infty}_{-n+1,q,-n}(\unit_{\Lambda})$ \\ \hline
$0$ & $H^{n-1,n}$  \\
$1$ & $h^{n,n+1}/\Sq^{2}\pr(H^{n-2,n})$ \\
$2$ & $h^{n+2,n+2}_{12}\oplus h^{n+1,n+2}/\Sq^{2}(h^{n-1,n+1})$ \\
$3$ & $h^{n+2,n+3}/\tau\partial^{12}_{2}(h^{n+1,n+2}_{12})+\Sq^{2}(h^{n,n+2})$\\
\end{tabular}
\end{center}
Lemmas \ref{lem:slices-unit-kq-1} and \ref{lem:slices-unit-kq-2} imply a split injection $E^{2}_{-n+1,0,-n}(\unit_{\Lambda})\to E^{2}_{-n+1,0,-n}(\f_0(\KQ))$ and an isomorphism 
$E^{2}_{-n+1,1,-n}(\unit_{\Lambda})\to E^{2}_{-n+1,1,-n}(\f_0(\KQ))$.
Likewise, 
$h^{n+1,n+2}/\Sq^{2}(h^{n-1,n+1})$
--- a direct summand of $E^{2}_{-n+1,2,-n}(\unit_{\Lambda})$ --- maps isomorphically to $E^{2}_{-n+1,2,-n}(\f_0(\KQ))$, 
while the other direct summand $h_{12}^{n+2,n+2}$ of $E^2_{-n+1,2,-n}(\unit_{\Lambda})$ maps trivially to $E^2_{-n+1,2,-n}(\f_0(\KQ_{\Lambda}))$ 
and similarly for $h^{n+2,n+3}/\tau\partial_{2}^{12}(h^{n+1,n+2}_{12})+\Sq^{2}(h^{n,n+2}) = E^2_{-n+1,3,-n}(\unit_{\Lambda})$.

Every element of $h^{n,n+2}$ is of the form $\tau^{2} c$ for some class $c\in h^{n,n}$, 
cf.~Remark \ref{remark:pi10}.
We have $\Sq^{2}(\tau^{2} c)=\rho^{2}\tau c$ by \cite[Corollary 6.2]{roendigs-oestvaer.hermitian}.
Combined with Theorem \ref{theorem:steenrod-algebra} in the case $s=1$ we conclude there is an isomorphism
\[
h^{n+2,n+3}/\tau\partial^{12}_{2}(h^{n+1,n+2}_{12})+\Sq^{2}(h^{n,n+2})
\cong
h^{n+2,n+2}/\partial^{12}_{2}(h^{n+1,n+2}_{12}).
\]

The change of coefficients long exact sequence
\[ 
\dotsm 
\to 
h^{n+1,n+2} 
\to 
h^{n+1,n+2}_{24}
\to h^{n+1,n+2}_{12} 
\xrightarrow{\partial^{12}_{2}} 
h^{n+2,n+2} 
\to 
\dotsm 
\]
induces a short exact sequence
\[ 
0 
\to 
h^{n+2,n+2}/\partial^{12}_{2}(h^{n+1,n+2}_{12})
\to 
h^{n+2,n+2}_{24} 
\to 
h^{n+2,n+2}_{12} 
\to 
0. 
\]
Lemma \ref{lem:kernel-eta} shows that any hidden extension or filtration shift in the slice spectral sequence is a graded $\mathbf{K}^{\M}_{\ast}$-module.
The group $\Ext_{\mathbf{K}^{\M}_{\ast}}(h^\ast_{12},h^\ast)$ identifies with $\mathbf{K}^{\M}_{\ast}/2$, 
whose degree zero part is $\Lambda/2$, 
i.e., 
there are only two possible degree zero extensions. 
By comparing $\pi_{3,2}\unit_{\Lambda}$ with $\pi_{3}\unit_{\Top}\cong\Z/24$ via topological or \'etale realization (the latter involves $2$-adic completion,
see Remark~\ref{rem:etale-realization})
it follows that the extension is nontrivial. 

Surjectivity of the rightmost map for $n\geq -4$ follows since the induced map between the $n$th slice spectral sequences is a surjection on the $E^{2}=E^{\infty}$-pages in the range we consider
by Propositions~\ref{prop:Einfty-KQ} and~\ref{prop:eff-kq-conv}.
Here we use $H^{-1,2}=0$ to conclude for $n=-4$ \cite[(4.2)]{merkurjev.weighttwo}.
\end{proof}

\begin{corollary}
\label{cor:hurewicz-1-line-surjective}
For every $n\in\ZZ$ the unit map $\unit_\Lambda\to \M \Lambda$ induces a surjection 
\[ 
\pi_{n+1,n}\unit_\Lambda
\to 
\pi_{n+1,n}\M \Lambda = H^{-n-1,-n}. 
\]
\end{corollary}

\begin{remark}\label{rem:etale-realization}
In odd characteristic the proof of Theorem~\ref{theorem:1line} uses the \'etale realization functor 
\[
\mathrm{Et}_2\colon\SH(F) \rightarrow \SH^{\wedge}_2 
\]
to the 2-completed profinite stable homotopy category over the \'etale homotopy type of a separable closure of $F$ \cite[Theorem 31]{quick}.
The examples in \cite[p.~743]{quick} show $\mathrm{Et}_2(S^{s,t}) = \bigl(S^s\bigr)^\wedge_2$.
By construction we have $\mathrm{Et}_2(\eta_{\Top})=\eta_{\Top}$, 
while the description of $\eta$ via the Hopf construction on the multiplicative group $\G$ implies $\mathrm{Et}_2(\eta)=\eta_{\Top}$.
\end{remark}

\begin{remark}
\label{rem:image-oneline-kq}
As an abelian group the kernel of $\pi_{3,2}\unit_{\Lambda}\to\pi_{3,2}\f_0(\KQ_{\Lambda}) =0$ is generated by the second motivic Hopf map 
$\nu\colon\Sigma^{3,2}\unit_{\Lambda}\to\unit_{\Lambda}$.
Hence as a graded $\mathbf{K}^{\M}_{\ast}$-module, 
the kernel of \[ \bigoplus_{n\in \ZZ}\pi_{-n+1,-n}\unit_{\Lambda}\to\bigoplus_{n\in \ZZ}\pi_{-n+1,-n}\f_{0}(\KQ_{\Lambda})\] is generated by $\nu$.
In principle, 
the image of $\pi_{-n+1,-n}\unit_{\Lambda}\to\pi_{-n+1,-n}\f_0(\KQ_{\Lambda})$ can be described as an extension of the three groups:
\begin{equation}
\label{eq:image-kq} 
H^{n-1,n}, 
\quad \quad
h^{n,n+1}/\Sq^2\pr (H^{n-2,n}), 
\quad \quad
h^{n+1,n+2}/\Sq^2(h^{n-1,n+1}). 
\end{equation}
As graded $\mathbf{K}^{\M}_{\ast}$-modules, 
the second group in \eqref{eq:image-kq} is generated by the topological Hopf map $\eta_{\Top}\colon\Sigma^{1,0}\unit_{\Lambda}\to\unit_{\Lambda}$, 
while the third group is generated by $\eta\eta_{\Top}\colon\Sigma^{2,1}\unit_{\Lambda}\to\unit_{\Lambda}$. 
Hence the extension of the last two groups is generated by $\eta_{\Top}$ as a graded $\mathbf{K}^{\MW}_{\ast}$-module. 
The group $H^{n-1,n}$ is not necessarily generated by a $2$-torsion element.
For example,
$H^{1,2}(\mathbb{Q})\iso\ZZ/24$ and $H^{1,2}(\mathbb{Q}(\sqrt{-1}))$ contains $\Z$ as a direct summand,  
see e.g., 
\cite[p.~542, 564]{merkurjev-suslin.k3}.
The computation of $\pi_{3,2}\unit_{\Lambda}$ implies the relation
$\eta^2\eta_{\Top}=12 \nu$, because $\eta^2\eta_{\Top}$ is an element
of order two and nontrivial by complex or \'etale realization, 
cf.~Remark~\ref{rem:etale-realization}.
\end{remark}

\begin{example}
\label{ex:unit-not-surj}
The map 
$
\pi_{-4,-5}\unit_\Lambda
\to 
\pi_{-4,-5}\f_0(\KQ)_\Lambda 
$
is never surjective because the summand $h^{0,5}\cong \ZZ/2$ of $E^{\infty}_{-n+1,0,-n}(\f_0(\KQ))$ is not in its image.
The map
$
\pi_{-n+1,-n}\unit_\Lambda
\to 
\pi_{-n+1,-n}\KQ_\Lambda 
$
may also fail to be surjective.
Over $F=\overline{\mathbb{Q}}$, 
$
\pi_{-6,-7}\unit
\to 
\pi_{-6,-7}\KQ 
$
is not a surjection. 
The slice $\s_{-2}(\KQ)$ gives rise to a nontrivial element of $\pi_{-6,-7}\KQ$.
Here we rely on the computations for hermitian $K$-theory in \cite[\S5,7]{roendigs-oestvaer.hermitian}.
The differential $d^{\KQ}_{r}$ exiting $E^{r}_{-6,-2,-7}(\KQ)$ is trivial when $r\geq 3$, 
while
$
d^{\KQ}_{2}
\colon
E^{2}_{-6,-2,-7}(\KQ)
= 
h^{0,5}\directsum H^{2,5} 
\to 
H^{7,7}
$
maps $h^{0,5}$ trivially, since $H^{7,7}=\mathbf{K}^{\M}_{7}$ contains no elements of order two. 
The only possible nontrivial $r$th differential entering $E^{r}_{-6,-2,-7}(\KQ)$ is 
$ 
E^{2}_{-5,-4,-7}(\KQ) 
= 
H^{-3,3} 
\to 
h^{0,5}\directsum H^{2,5}.
$
Here $H^{-3,3}=0$ since $\overline{\mathbb{Q}}$ is a filtered colimit of number fields, whence the Beilinson-Soul\'e vanishing conjecture holds.
\end{example}

\begin{example}\label{ex:reals}
  It is instructive to compare the short exact sequence of
  Theorem~\ref{theorem:1line} with the computation 
  $ \pi_{-n+1,-n}(\unit^{\wedge}_{2})\iso \ZZ/2 $ over $\mathbb{R}$
  for $n>1$  
  in \cite[Section 8.3, Figure 4]{dugger-isaksen.real}. 
  Theorem~\ref{theorem:1line} and Remark~\ref{rem:image-oneline-kq}
  provide
  an extension
  \[ 0 \to \mathbf{K}^{\M}_{n+2}(\mathbb{R})/24 \to
  \pi_{-n+1,-n} \unit \to \mathbf{kq}_{-n} \to 0, \]
  where $\mathbf{kq}_{-n}$ is an extension of $H^{n-1,n}$ by
  $h^{n,n+1}/\Sq^2\pr (H^{n-2,n})$. In particular,
  $\mathbf{kq}_{-2}$ is an extension of $h^{2,3}=\{0,\tau\rho^2\}$
  by the $2$-divisible group $H^{1,2}$ 
  (which contains a unique nontrivial element of order $2$).
  If both $\mathbf{kq}_{-2}$ and $\pi_{-1,-2} \unit$  were trivial extensions, 
  the $2$-adic completion
  of $\pi_{-1,-2}\unit$ would consist of four elements.
  Thus one of the extensions is nontrivial,
  whence $\pi_{-1,-2}\unit$ is the direct sum of 
  $\mathbf{K}^{\M}_{4}(\mathbb{R})/24\iso \{0,\rho^{4}\}$
  and a nontrivial $2$-divisible group. 
\end{example}

\appendix

\section{The Steenrod algebra and its dual}
\label{sec:steenrod-algebra-its}

This section lists results on the Steenrod algebra which are used in the preceding sections.
In the following the base scheme contains no points of residue characteristic two.

\begin{theorem}[Voevodsky \cite{Voevodsky.reduced}, Hoyois-Kelly-{\O}stv{\ae}r \cite{hko.positivecharacteristic}, Spitzweck \cite{spitzweckmz}]
\label{theorem:dual-steenrod-algebra}
There exists a weak equivalence of right $\MZ/2$-modules
\begin{equation}
\label{eq:dual-steenrod-alg}
\MZ/2\smash \MZ/2 
\xrightarrow{\sim} 
\MZ/2 \vee \Sigma^{1,0} \MZ/2\vee \bigvee_{(i,j)\in I} \Sigma^{i,j}  \MZ/2,
\end{equation}
where $I\subset \N\times \N$ consists of pairs $(i,j)$ of integers with $i\geq 2j>0$. 
With respect to this weak equivalence, 
the unit and multiplication maps are given by 
\begin{align*}
(\id,0,\dotsc)\colon & \MZ/2\iso \MZ/2\smash \unit \to \MZ/2\smash \MZ/2  \\
(\id,\Sq^1,0,\dotsc)\colon & \MZ/2\iso \unit\smash \MZ/2 \to \MZ/2\smash \MZ/2 \\
(\id,0,\dotsc)\colon &    \MZ/2\smash \MZ/2 \to \MZ/2.
\end{align*}
\end{theorem}

The identification of $\MZ/2\smash \MZ/2$ in Theorem~\ref{theorem:dual-steenrod-algebra} induces a weak equivalence
\begin{equation}\label{eq:dual-steenrod} 
\MZ/2\smash_{\MZ} \MZ/2 \xrightarrow{\sim} 
\MZ/2 \vee \Sigma^{1,0} \MZ/2. 
\end{equation}

\begin{corollary}
\label{cor:mult-mot-coh}
Via \eqref{eq:dual-steenrod-alg} the external multiplication $(f,g)\mapsto f\smash g$ induces the pairing
\begin{align*}
h^{a,b}\directsum h^{c,d} & \to h^{a+c,b+d} \directsum h^{a+c+1,b+d}\directsum \dotsm;
(f,g) 
\mapsto 
\bigl(f\cdot g,f\cdot \Sq^1(g),\dotsm\bigr).
\end{align*}
\end{corollary}  

\begin{lemma}
\label{lem:steenrod-alg-weight-zero}
For abelian groups $A$ and $B$ there are isomorphisms
\[ 
[\MA,\Sigma^{s,0}\MB] 
\iso 
\begin{cases} 
\Hom(A,B) & s=0 \\
\Ext(A,B) & s=1 \\
0 & \mathrm{otherwise}.
\end{cases}
\]
Moreover, 
if $\MA\to \Sigma^{1,0}\MB$ is induced by the extension
$ 
0 
\to 
B
\to 
C
\to 
A 
\to 
0 
$,
then its cone is $\Sigma^{1,0}\MB\to \Sigma^{1,0}\MC$.
\end{lemma}

\begin{example}\label{ex:steenrod-alg-weight-zero}
Let $A$ and $B$ be finite cyclic groups of order $a$ and $b$, 
respectively.
Let $\inc^a_b$ denote an inclusion $A\hookrightarrow B$, 
and let $\pr^{a}_{b}$ denote a quotient map $A\to B$.
If $\inc^a_b\colon A\hookrightarrow B$ is an inclusion of finite cyclic groups, 
with quotient $\pr^b_c\colon B\to C$, 
Lemma~\ref{lem:steenrod-alg-weight-zero} shows there exists a unique map
\[ 
\partial^c_a\colon \M C \to \Sigma^{1,0}\M A 
\]
representing the cone of $\M B\to \M C$.
For example, $\partial^2_2=\Sq^1$. 
If no confusion can arise, the ubiquitous index ``2'' will be left out.
\end{example}

\begin{theorem}
\label{theorem:steenrod-algebra}
For $r\geq 1$ the homotopy cofiber sequences
\[ 
\MZ/{2^{r-1}}
\rightarrow
\MZ/2^{r} 
\rightarrow 
\MZ/2 
\xrightarrow{\partial^2_{2^{r-1}}} 
\Sigma^{1,0}\MZ/2^{r-1} 
\]
and 
\[ 
\MZ/{2}
\rightarrow 
\MZ/2^{r} 
\rightarrow 
\MZ/2^{r-1} 
\xrightarrow{\partial^{2^{r-1}}_2}
\Sigma^{1,0}\MZ/2 
\]
induce isomorphisms
\begin{align*}
[\MZ/2,\Sigma^{s,1}\MZ/2^{r}] &
\cong 
\begin{cases}
\inc^2_{2^r} \circ h^{0,1} & s=0 \\
\inc^2_{2^r} \circ h^{1,1} \oplus \partial^{2}_{2^r} \circ h^{0,1} & s=1 \\
\inc^2_{2^r} \circ h^{0,0}\{\Sq^2\} \oplus \partial^{2}_{2^r} \circ h^{1,1} & s=2 \\
\inc^2_{2^r} \circ h^{0,0}\{\Sq^2\Sq^1\} \oplus \partial^{2}_{2^r} \circ h^{0,0}\{\Sq^2\} & s=3 \\
\partial^{2}_{2^r} \circ h^{0,0}\{\Sq^2\Sq^1\} & s=4 \\
0 & s\geq 5,
\end{cases} 
\\
[\MZ/2^r,\Sigma^{s,1}\MZ/2] &
\cong
\begin{cases}
h^{0,1}  \circ \pr^{2^r}_{2} & s=0 \\
h^{1,1}\circ \pr^{2^r}_{2} \oplus  h^{0,1} \circ \partial^{2^r}_{2}& s=1 \\
h^{0,0}\{\Sq^2\} \circ \pr^{2^r}_{2}\oplus  h^{1,1} \circ \partial^{2^r}_{2}& s=2 \\
h^{0,0}\{\Sq^1\Sq^2\} \circ  \pr^{2^r}_{2}\oplus  h^{0,0}\{\Sq^2\}\circ \partial^{2^r}_{2}& s=3 \\
h^{0,0}\{\Sq^1\Sq^2\}\circ  \partial^{2^r}_{2} & s=4 \\
0 &  s\geq 5.
\end{cases}
\end{align*}
\end{theorem}

\section{Ext groups for complex bordism}
\label{sec:ext-groups-complex}

Nearly all of the results in this section go back to Novikov's paper~\cite{novikov}.
If $(s,t)\neq (0,0)$, the group $\Ext_{\MU_\ast\MU}^{s,t}(\MU_\ast,\MU_\ast)$ is finite \cite[Corollary 2.1]{novikov}, \cite[Property 2.2]{Zahler}.
Moreover, 
\begin{align*}
\Ext_{\MU_\ast\MU}^{s,0}(\MU_\ast,\MU_\ast) 
&
\cong 
\begin{cases} 
\ZZ & s=0 \\
0 & s\neq 0. 
\end{cases}
\end{align*}

It is convenient to display the Ext-groups groups for the Brown-Peterson spectrum $\BP$ at a prime $p$, 
and use $\Ext_{\MU_\ast \MU}^{s,t}(\MU_\ast,\MU_\ast)=\bigoplus_{p}\Ext_{\BP_\ast \BP}^{s,t}(\BP_\ast,\BP_\ast)$ from~\cite{novikov}.
Then 
\begin{equation*}
\Ext_{\BP_\ast\BP}^{s,t}(\BP_\ast,\BP_\ast) 
= 
0 
\mathrm{\ for\ } 2s(p-1)>t, \ \mathrm{ and\ for }\  2p-2 \nmid t, \\
\end{equation*}
by \cite[Proposition 4.4.2, Edge Theorem 5.1.23]{ravenel.green} and \cite[Property 2.1]{Zahler}.

For $p=2$, \cite[Edge Theorem 5.1.23, Theorem 5.2.6(c), (d)]{ravenel.green} shows 
\begin{align*}
\Ext_{\BP_\ast\BP}^{1,2q}(\BP_\ast,\BP_\ast) 
&
\cong 
\begin{cases}
\ZZ/2\{\alpha_{q}\}         & q\equiv 1,3 \bmod 4 \\
\ZZ/4\{\alpha_{2/2}\}       & q=2 \\
\ZZ/2^{n+3}\{\alpha_{q/n+3}\}    & q=2^{n+1}r \equiv 0,2 \bmod 4; q>2, 2\nmid r.
\end{cases}
\end{align*}

In the proof of Lemma \ref{lem:mult-oddslice-1slice} we use the normalized cobar complex obtained from the standard $\BP$-cosimplicial resolution at the prime $2$, 
and corresponding representatives for the generators in $\Ext_{\BP_{\ast}\BP}^{1,2q}(\BP_{\ast},\BP_{\ast})$.
At the prime $2$,
the choice of Hazewinkel generators gives $\BP_{\ast} = \ZZ_{(2)}[v_{1},v_{2},\dotsc]$ and $\BP_{\ast}\BP = \BP_{\ast}[t_{1},t_{2},\dotsc]$, 
where $\lvert v_i\rvert = \lvert t_i\rvert = 2(2^{i}-1)$. 
If $q$ is odd, 
then $\alpha_q$ is represented by $\tfrac{d(v_{1}^q)}{2}$.
Here $d$ is the differential in the Adams degree $2q$ part $\Tot\{q\}$ of the 
cobar complex for $\BP$.
In particular, 
$\alpha_{1}$ is represented by $[t_{1}]$ in 
\[ 
\Tot\{1\} 
= 
\Bigl( \ZZ_{(2)}\{v_{1}\} \xrightarrow{d=2}\ZZ_{(2)}\{[t_{1}]\}\Bigr). 
\] 
Further,
$\alpha_{2/2}$ is represented by 
$\tfrac{d(v_{1}^{2})}{4}=v_{1}[t_{1}]+[t_{1}^{2}]$.
If $q=2^{n+1}r>2$ with
$r$ odd, 
$\alpha_{q/n}$ is represented by 
\begin{equation*}
\frac{d(v_{1}^{q}+2qv_{1}^{q-3}v_{2})}{2^{n+3}}.
\end{equation*}
For all $q\in \NN$, 
$H_{0}(\Tot\{2q-1\}\otimes \Tot\{1\})=0$ and the generator of $H_{1}(\Tot\{2q-1\}\otimes \Tot\{1\})=\ZZ/2$ is represented by 
\[ 
\frac{d(v_{1}^{2q-1}\otimes v_{1})}{2}
=
\frac{d(v_{1}^{2q-1})}{2}\otimes v_{1} + v_{1}^{2q-1}\otimes [t_{1}].
\] 
Moreover, 
$H_{2}(\Tot\{2q-1\}\otimes \Tot\{1\})$ contains a direct summand $\ZZ/2$ generated by 
\[
\frac{d(v_{1}^{2q-1})}{2}\otimes [t_{1}].
\]
Using the multiplication in the $\BP$-cobar complex there is an induced map
\[ 
\Tot\{2q-1\}\otimes \Tot\{1\} 
\to 
\Tot\{2q\}.
\]
By the proof of Lemma \ref{lem:slices-unit-kq-2} there is an induced inclusion 
\[ 
H_{1}(\Tot\{2q-1\}\otimes \Tot\{1\}) 
= 
\ZZ/2
\to 
H_{1}(\Tot\{2q\}) 
= 
\ZZ/a_{2q}.
\]
Since $\Tot\{2q\}$ is degreewise free there exists a map $\Tot\{2q\} \to H_{1}(\Tot\{2q\})[1]=\ZZ/a_{2q}[1]$ inducing the identity on $H_{1}$.

In order to describe the multiplication map to $H_{1}(\Tot\{2q\})[1]$ completely, 
it remains to identify the composite
\[
\Bigl( \ZZ_{(2)}\{v_{1}^{2q-1}\otimes [t_{1}]\} 
\xrightarrow{d}
\ZZ_{(2)}\{\tfrac{d(v_{1}^{2q-1})}{2}\otimes [t_{1}]\}\Bigr)[1]
\to 
\Tot\{2q-1\}\otimes \Tot\{1\} 
\to 
\Tot\{2q\} 
\to 
H_{1}(\Tot\{2q\})[1]. 
\] 
It represents either the trivial map, or the connecting map for the short exact sequence
\[ 
0
\to 
\ZZ/2
\to 
\ZZ/2a_{2q} 
\to 
\ZZ/a_{2q} 
\to 
0. 
\]

\begin{lemma}
\label{lem:mult-oddslice-1slice}
Let $\Tot\{k\}$ denote the Adams degree $2k$ part of the cobar complex for $\BP_\ast$ at the prime $2$. 
The multiplication map $\Tot\{2q-1\}\otimes \Tot\{1\} \to \Tot\{2q\}$ induces the trivial map of chain complexes
\[ 
\bigl(\ZZ_{(2)}\{v_{1}^{2q-1}\otimes [t_{1}]\}
\xrightarrow{2} 
\ZZ_{(2)}\{\tfrac{d(v_{1}^{2q-1})}{2}\otimes [t_{1}]\}\bigr)[1] 
\to
H_{1}(\Tot\{2q\})[1]. 
\]
\end{lemma}
\begin{proof}
The generator $v_{1}^{2q-1}\otimes [t_{1}]$ in degree $1$ maps to $v_{1}^{2q-1}[t_{1}]$. 
The element $\alpha_{2q/n+3}\in \Tot\{2q\}_{1}$ in the target is represented by 
\[
g_{2q}
= \frac{d(v_1^{2q} + 4q v_1^{2q-3}v_2)}{2^{n+3}}, 
\]
where $2q=2^{n+1}r$, $r$ odd. 
We conclude by showing $\{g_{2q},v_{1}^{2q-1}[t_{1}]\}$ extends to a basis of the free $\ZZ_{(2)}$-module $\Tot\{2q\}_{1}$. 
To wit,
the coefficient of $v_{1}^{2q-4}v_{2}[t_{1}]$ for $2q\geq 4$ (of $[t_{1}^{2}]$ for $2q=2$) in $g_{2q}$ is odd with respect to the usual monomial basis of $\Tot\{2q\}_{1}$: 
The coefficient of $v_{1}^{2q-4}v_{2}[t_{1}]$ in $d(v_{1}^{2q-3}v_{2})$ is $4q-6$.
The coefficient of $v_{1}^{2q-4}v_{2}[t_{1}]$ in $d(v_{1}^{2q})$ is zero.
Thus when $2q\geq 4$, the coefficient of $v_{1}^{2q-4}v_{2}[t_{1}]$ in $g_{2q}$ is the odd integer
\begin{equation*}
\frac{(4q-6)4q}{2^{n+3}}=
\frac{(2^{n+2}r - 6)2^{n+2}r}{2^{n+3}} = (2^{n+1}r-3)r,
\end{equation*}
which completes the argument.
\end{proof}

Letting $\overline{\alpha}_{q}$ be the generic notation for a generator of the cyclic group $\Ext_{\BP_\ast\BP}^{1,2q}(\BP_\ast,\BP_\ast)$, 
\cite[Proposition 6.1]{Zahler} shows
\begin{align*}
\Ext_{\BP_\ast\BP}^{s,2(s+t)}(\BP_\ast,\BP_\ast) 
&
\cong 
\begin{cases}
0                                     & t=1, s\geq 2 \\ 
\ZZ/2\{\alpha_{1}^{s}\}                &  t=0, s\geq 1 \\
\ZZ/2\{\alpha_{1}^{s-1}\overline{\alpha}_{t+1}\}       & 2\leq t\leq 6, s\geq 4. \\
\end{cases}
\end{align*}

For $p$ odd, 
\cite[Theorem 5.2.6(a), (b)]{ravenel.green} shows 
\begin{align*}
\Ext_{\BP_\ast\BP}^{1,2k(p-1)}(\BP_\ast,\BP_\ast) 
&
\cong 
\ZZ/p^{n+1}\{\alpha_{k/n}\}, k=p^{n}r \mathrm{\ for\ } p\nmid r. 
\end{align*}

For the beta family of elements on the $2$-line of the Adams-Novikov spectral sequence we refer to \cite[Chapter 4, \S4]{ravenel.green}.

{\bf Acknowledgments.}
We thank Aravind Asok and Kyle Ormsby for discussions at an early stage of this work, 
and Dan Isaksen for pointing out \cite{andrews-miller}.
We also thank the participants of the USC $K$-theory summer school 2018, 
and in particular Aravind Asok, Bert Guillou, and Glen Wilson, for comments. 
The authors gratefully acknowledge support by the RCN grant no.~239015 "Special Geometries",
RCN Frontier Research Group Project no.~250399 "Motivic Hopf Equations'', 
and the DFG Priority Program 1786 "Homotopy theory and algebraic geometry''.
{\O}stv{\ae}r is supported by a Friedrich Wilhelm Bessel Research Award from the Humboldt Foundation.

\begin{small}

\vspace{0.1in}

\begin{center}
Institut f\"ur Mathematik, Universit\"at Osnabr\"uck, Germany.\\
e-mail: oliver.roendigs@uni-osnabrueck.de
\end{center}
\begin{center}
Institut f\"ur Mathematik, Universit\"at Osnabr\"uck, Germany.\\
e-mail: markus.spitzweck@uni-osnabrueck.de
\end{center}
\begin{center}
Department of Mathematics, University of Oslo, Norway.\\
e-mail: paularne@math.uio.no
\end{center}
\end{small}
\end{document}